\documentclass[a4paper,11pt]{article}
\usepackage[T1]{fontenc}
\usepackage[utf8]{inputenc}
\usepackage[english]{babel}
\usepackage{amsmath, amssymb, amsthm, amsfonts, mathtools, mathrsfs}	
\usepackage{cancel}
\usepackage{enumerate}
\usepackage[shortlabels]{enumitem}
\usepackage{geometry}
\usepackage{xcolor}

\usepackage{cite} 

\usepackage{bbm}

\usepackage{orcidlink}

\newcommand{\R}{\mathbb{R}}
\newcommand{\N}{\mathbb{N}}

\newcommand{\C}{\mathbb{C}}

\numberwithin{equation}{section}

\theoremstyle{plain}
\newtheorem{theorem}{Theorem}[section] 

\theoremstyle{definition}
\newtheorem{definition}[theorem]{Definition}
\newtheorem{remark}[theorem]{Remark}

\DeclarePairedDelimiter{\abs}{\lvert}{\rvert}
\DeclarePairedDelimiter{\norm}{\lVert}{\rVert}
\DeclarePairedDelimiter{\duality}{\langle}{\rangle}
\DeclareMathOperator{\dom}{dom}

\DeclareMathOperator{\divergence}{div}
\DeclareMathOperator*{\argmin}{arg\,min}

\renewcommand{\div}{\divergence} 
\newcommand{\eps}{\varepsilon}
\renewcommand{\phi}{\varphi}
\renewcommand{\bar}{\overline}
\renewcommand{\vec}{\mathbf}

\def\genspazio #1#2#3#4#5{#1^{#2}(#5,#4;#3)}
\def\spazio #1#2#3{\genspazio {#1}{#2}{#3}T0}

\def\LT {\spazio L}
\def\HT {\spazio H}
\def\C #1#2{\mathcal{C}^{#1}([0,T];#2)}
\def\Cw #1#2{\mathcal{C}_w^{#1}([0,T];#2)}

\def\Lx #1{L^{#1}(\Omega)}
\def\Lt #1{L^{#1}(0,T)}
\def\Lqt #1{L^{#1}(Q_T)}
\def\Hx #1{H^{#1}(\Omega)}
\def\Wx #1{W^{#1}(\Omega)}
\def\Cx #1{\mathcal{C}^{#1}(\bar{\Omega})}

\def\Accorpa #1#2 #3 {\gdef #1{\eqref{#2}--\eqref{#3}}%
	\wlog{}\wlog{\string #1 -> #2 - #3}\wlog{}}

\def\ls{<}
\def\gs{>}
\def\mezzo {\frac{1}{2}}
\def\de {\mathrm{d}}
\def\ddt {\frac{\de}{\de t}}

\def\weakstar {\stackrel{\ast}{\rightharpoonup}}

\def\n {\vec{n}}
\def\hh {\mathbbm{h}}
\def\phib {\bar{\phi}}
\def\mub {\bar{\mu}}
\def\sigmab {\bar{\sigma}}
\def\ub {\bar{u}}
\def\vb {\bar{v}}

\def\Uad {\mathcal{U}_{\text{ad}}}
\def\Vad {\mathcal{V}_{\text{ad}}}

\usepackage{hyperref}

\begin{document}
		
	\begin{center}
		
		\LARGE{\textbf{Optimal distributed control for a viscous non-local tumour growth model}}
		
		\vskip0.5cm
		
		\large{\textsc{Matteo Fornoni$^1$} \orcidlink{0000-0002-9787-047X}} \\
			\normalsize{e-mail: \texttt{matteo.fornoni01@universitadipavia.it}} \\
		\vskip0.5cm
				
		\small{$^1$Department of Mathematics "F. Casorati", Università degli Studi di Pavia, \\
			 via Ferrata 5, 27100, Pavia, Italy}
		
		\vskip0.5cm
		
	\end{center}

\begin{abstract}
\noindent
In this paper, we address an optimal distributed control problem for a non-local model of phase-field type, describing the evolution of tumour cells in presence of a nutrient. 
The model couples a non-local and viscous Cahn-Hilliard equation for the phase parameter with a reaction-diffusion equation for the nutrient. 
The optimal control problem aims at finding a therapy, encoded as a source term in the system, both in the form of radiotherapy and chemotherapy, which could lead to the evolution of the phase variable towards a desired final target. 
First, we prove strong well-posedness for the system of non-linear partial differential equations. 
In particular, due to the presence of a viscous regularisation, we can also consider double-well potentials of singular type and cross-diffusion terms related to the effects of chemotaxis. 
Moreover, the particular structure of the reaction terms allows us to prove new regularity results for this kind of system. 
Then, turning to the optimal control problem, we prove the existence of an optimal therapy and, by studying Fréchet-differentiability properties of the control-to-state operator and the corresponding adjoint system, we obtain the first-order necessary optimality conditions.

\vskip3mm


\vskip3mm


\end{abstract}

\section{Introduction}

In recent years the mathematical modelling community has devoted much attention towards the understanding of intrinsic phenomena behind tumour growth, as well as the possibility of simulating medical treatments, which could drive the evolution of tumours. 
Without being by any means exhaustive, we cite \cite{BC1997_freeboundary, BP2003_modelling, Lorenzo2022_review, Lowengrub2010_review, cristini:lowengrub} and references therein. 
Regarding the mathematical analysis of these models, the main questions concern not only well-posedness, regularity and long-time behaviour of solutions, but also related issues such as optimal control or parameter identification problems, see for example \cite{AACGV2017, CRW2021, CSS2021_secondorder, GLS2021_sparseoc, GLSS2016, RSS2021}. 

In the present work, we consider a diffuse interface model, which is a non-local and viscous variant of the model originally introduced in \cite{HZO2012}. 
We start by recalling the original model in \cite{HZO2012}, which can be written in the following form:
\begin{alignat}{2}
	& \partial_t \phi = \div (m(\phi) \nabla \mu) + P(\phi) (\sigma + \chi (1-\phi) - \mu) \qquad && \text{in } Q_T,  \label{eq:philoc}\\
	& \mu = AF'(\phi) - B \Delta \phi- \chi \sigma \qquad && \text{in } Q_T,  \label{eq:muloc} \\
	& \partial_t \sigma = \div (n(\phi) \nabla (\sigma + \chi (1 - \phi))) - P(\phi) (\sigma + \chi (1-\phi) - \mu) \qquad && \text{in } Q_T, \label{eq:sigmaloc}
\end{alignat}
together with homogenous Neumann boundary conditions and initial conditions.
The main variable of the model is $\phi$, a phase-field which represents the difference in volume fractions between tumour cells and healthy cells coexisting in the spatial domain $\Omega \subseteq \R^N$, $N=2,3$. 
Physically, $\phi$ should assume only values between $-1$ and $1$, where $\phi \simeq 1$ is identified as the tumour phase and $\phi \simeq -1$ as the healthy phase. 
As a phase-field approximation, $\phi$ is free to assume values in $(-1,1)$, corresponding to a mixture of the two cell types; however, the interfacial region is generally kept small through an appropriate parameter. 
The evolving tumour region is, then, recovered as a level set of $\phi$ at values close to $1$. Besides being easier to deal with from a numerical point of view, one of the main advantages of diffuse interface models, when compared to free-boundary ones, where $\phi$ would assume only the values $-1$ or $1$, is the possibility of accounting for topological changes in the tumour mass, which can be quite common in certain stages of its evolution. 
The other variable $\sigma$ of the system is related to the concentration of a nutrient in the extracellular water, such as oxygen or glucose. 
The idea is that tumour cells proliferate by absorbing such substance, whereas the concentration of nutrient consequently varies depending on the areas in which the tumour evolves. 
Also in this case, $\sigma$ should physically be between $0$ and $1$, where $\sigma \simeq 0$ represents low concentrations of nutrient and $\sigma \simeq 1$ high ones.
In particular, $\phi$ satisfies a Cahn-Hilliard equation, which is well-suited in describing phase-separation phenomena (see \cite{miranville} and references therein), while the nutrient variable $\sigma$ satisfies a reaction-diffusion equation, and the two equations are coupled through an explicit reaction term, which models the interaction between nutrient and evolving tumour. 
We also account for another coupling term, associated with the parameter $\chi \ge 0$, that is responsible for the chemotaxis phenomenon, which concerns the natural movement of tumour cells towards regions with higher concentrations of nutrient.  

As we briefly said at the beginning, we intend to consider a non-local and viscous variant of the system \eqref{eq:philoc}--\eqref{eq:sigmaloc}.
Regarding the non-local aspect of the model, this is achieved by substituting the local term $-B \Delta \phi$ with one involving a convolution operator with a symmetric kernel $J$, possibly accounting for more complex interactions between cells, involving long-range competition. 
In this way, $\phi$ satisfies a non-local Cahn-Hilliard equation (see \cite{GL1997, GGG2023}), which can be seen as an approximation of the local one, when the kernel $J$ is suitably peaked around zero (see \cite{DRST2020, DST2021_1, DST2021_2}). Analysis and application of non-local phase-field models on tumour growth phenomena is still flourishing; we just mention some interesting works in this direction \cite{RSS2021, FLS2021, FLNOW2019}.
The second modification applied to the original model is the addition of a viscosity term $\tau \phi_t$ in the equation for the chemical potential $\mu$. Physically, it can be interpreted as a further dissipation acting on the system, coming from internal frictions (see \cite{N1988_viscousCH}). From the mathematical point of view, however, it acts as a regularising term, since it allows some additional regularity estimates. Moreover, we also add two external sources $u$ and $v$, which will then represent possible therapies in our optimal control problem.
 
In its most general form, the model that we are going to analyse is the following:
\begin{alignat}{2}
	& \partial_t \phi = \div (m(\phi) \nabla \mu) + P(\phi) (\sigma + \chi (1-\phi) - \mu) - \hh(\phi) u \qquad && \text{in } Q_T,  \label{eq:phi}\\
	& \mu = \tau \partial_t \phi + AF'(\phi) + Ba \phi - BJ \ast \phi - \chi \sigma \qquad && \text{in } Q_T,  \label{eq:mu} \\
	& \partial_t \sigma = \div (n(\phi) \nabla (\sigma + \chi (1 - \phi))) - P(\phi) (\sigma + \chi (1-\phi) - \mu) + v \qquad && \text{in } Q_T, \label{eq:sigma}
\end{alignat}
where $T>0$ is a fixed final time, $\Omega \subset \R^N$, $N=2,3$, is an open bounded sufficiently smooth domain and $Q_T = \Omega \times (0,T)$. 
We complement it with the following homogeneous Neumann boundary conditions and initial conditions:
\begin{alignat}{2}
	& m(\phi) \partial_{\n} \mu = n(\phi) \partial_{\n} (\sigma + \chi (1-\phi)) = 0 \qquad && \text{on } \partial \Omega \times (0,T), \label{bc} \\
	& \phi(0) = \phi_0, \quad \sigma(0) = \sigma_0 \qquad && \text{in } \Omega, \label{ic}
\end{alignat}
where $\n$ is the exterior normal unit vector to $\partial \Omega$. 

We now briefly introduce all the parameters that appear in the system. 
The auxiliary variable $\mu$ can be interpreted as a chemical potential in the Cahn-Hilliard equation, and, up to the additional viscosity term $\tau \partial_t \phi$, with $\tau \gs 0$, it is the variational derivative of the non-local free energy of Ginzburg-Landau type 
\[ \mathcal{E}(\phi,\sigma) = \int_{\Omega} A F(\phi) \, \de x + \int_{\Omega} \int_{\Omega} \frac{B}{4} J(x - y)(\phi(x) - \phi(y))^2 \, \de x \, \de y + \int_{\Omega} \frac{1}{2} \abs{\sigma}^2 + \chi \sigma (1-\phi) \, \de x. \]
The function $F$ is generally a potential of double-well type, possessing two global minima in correspondence of the two pure phases $-1$ and $1$. In most applications, typical choices of such potentials are a regular or polynomial one, a singular or logarithmic one, with derivatives exploding at $\pm 1$, or a double-obstacle one, which are respectively of the following form:
\begin{align*}
	& F_{\text{reg}}(s) = \frac{1}{4} (1 - s^2)^2, \quad s \in \R, \\
	& F_{\text{sing}}(s) = \frac{\theta}{2} \left[ (1+s) \log (1+s) + (1-s) \log (1-s) \right] - \frac{\theta_0}{2} s^2, \quad s \in (-1,1), \quad 0 \ls \theta \ls \theta_0, \\
	& F_{\text{obst}}(s) = \begin{cases}
		c(1-s^2) & \text{if } s \in [-1,1], \\
		+ \infty & \text{otherwise,}
	\end{cases} \quad c \gs 0.
\end{align*}
Actually, the global minima of the logarithmic potential $F_{\text{sing}}$ are not exactly in $-1$ and $1$, but in two intermediate values, which can be sufficiently close to them, depending on the choice of the parameters $\theta$ and $\theta_0$.
In this paper, we will consider only potentials of regular or singular type, excluding the double obstacle case, which can generally be treated with slightly different methods, as in \cite{CGS2017_deepquench, S2020_deepquench}. 
From a physical perspective, potentials of singular type are more fitting, since they can effectively limit the evolution of $\phi$ in the interval $[-1,1]$; in this case it is crucial to prove the so-called separation property, see the recent work \cite{GGG2023} for a detailed overview of the matter. 
This property is related to the fact that, during its evolution, the phase variable stays strictly inside the physical interval $(-1,1)$, which implies that the derivative of the potential does not actually blow-up. This is to be expected, since the variational structure of the system drives the evolution of $\phi$ towards the minima of the potential, which, in case of the singular one, are strictly inside the interval. In general this property is not always easy to prove and its validity is still an open problem in some situations, see however the recent works \cite{G2023_separation, P2023_separation} for substantial advances in the three dimensional non-local case. When considering a regular potential, instead, this property actually means that the phase parameter stays globally bounded during its evolution.
Here, thanks to the presence of the viscosity term, we are able to simplify these arguments and consider both regular and singular potentials, even in presence of reaction terms and sources. 
Proving this kind of property is important since it allows to effectively treat the singular potential as a regular one and compute its derivatives, and this is crucial to prove higher regularity for the solution of the system.
The second term in the energy stands as a non-local approximation of the local term $B \abs{\nabla \phi}^2/2$, which would penalise steep transitions in the interfaces between the two phases. 
The convolution kernel $J$ is generally taken of Newton or Bessel type; however, when studying non-local to local asymptotics of Cahn-Hilliard equations, one can also consider families of kernels suitably peaking around zero, see for instance \cite{DRST2020, DST2021_1, DST2021_2}. 
Finally, the parameters $A \gs 0$ and $B \gs 0$ are related to the width of the diffuse interface between the two phases, and are generally of the form $A= 1/\eps$ and $B =\eps$, for $\eps \ll 1$; whereas the parameter $\chi \ge 0$ and its corresponding term are related to the chemotactic effect.

Going back to the system \eqref{eq:phi}--\eqref{ic}, the functions $m(\phi)$ and $n(\phi)$ are called mobility functions and regulate the diffusion processes of the two variables. Generally, they can be constant or bounded both above and below, but in some cases, especially when $F$ is singular, $m$ can be degenerate on $\pm 1$. For instance, when dealing with the logarithmic potential $F_{\text{sing}}$, one can choose $m(s)=1-s^2$, in order to further enforce the physical condition and to compensate the singularities of the potential, see \cite{FLR2017, EG1996, FGG2021}. 
The function $P$ is a proliferation function that calibrates the strength of the reaction terms, which, in turn, are written in this form due to some chemical phenomenological laws (see \cite{HZO2012}). 
Typically, $P$ can be of the form 
\[ P(s) = \begin{cases}
	P_0 & \text{if } s \le -1, \\
	\frac{P_1 - P_0}{2}(s+1) + P_0 & \text{if } -1 \ls s \ls 1, \\
	P_1 & \text{if } s \ge 1,
\end{cases} \]
so that $P(-1)=P_0$ and $P(1)=P_1$, with $0 \le P_0 \ls P_1$, i.e. the proliferation is faster in the tumour phase. 
Generally, one takes $P_0 = 0$ and $P_1 = 1$.
Note that, for technical reasons, we will be forced to assume that $P$ is strictly positive, albeit still possibly very close to zero.
This will be needed in proving the existence of strong solutions to \eqref{eq:phi}--\eqref{ic}; in case of a singular potential, it will be used also for weak solutions (see Remark \ref{strategy:wsols} below).  
However, this is not unprecedented; for example, the same hypothesis was needed in \cite{CRW2021} to study long-time dynamics of a similar tumour growth model.  
Observe also that this kind of reaction term $P(\phi)(\sigma + \chi (1-\phi) - \mu)$, with opposite signs in \eqref{eq:phi} and \eqref{eq:sigma}, implies the conservation of mass for the sum of $\phi$ and $\sigma$, in absence of external sources (i.e. $u=v=0$). Indeed, by multiplying \eqref{eq:phi} and \eqref{eq:sigma} by $1$, integrating over $\Omega$ and summing up, we obtain the identity:
\[ \ddt \int_\Omega \phi(t) + \sigma(t) \, \de x = \int_\Omega \left( - \hh(\phi(t)) u + v \right) \, \de x. \]
Finally, the external sources $u$ and $v$ are the controls that we are going to use in our optimal control problem to direct the evolution of the system towards a certain objective. 
They can be interpreted as external therapies acting on the tumour. 
In particular, $u$ can represent a radiotherapy directly acting on the proliferation of the tumour cells, whereas $v$ can be seen as a chemotherapy acting on the tumour through drugs. 
The function $\hh(\phi)$ is an additional parameter that can be used to distribute the radiotherapy through particular strategies, see for example \cite{GLR2018, CSS2021}.

Regarding the well-posedness and regularity of this kind of system, there are already many contributions to the local version of the model; we just cite \cite{FGR2015_TumGrowth, CGRS2015, EG2019}. 
On the non-local model, there are fewer ones, but we recall \cite{FLR2017}, where the authors showed well-posedness for the system without viscosity (i.e.~$\tau = 0$) in the case of a regular potential and bounded mobility or of a singular potential and degenerate mobility. 
Then, for the same system, but without chemotaxis (i.e.~$\chi = 0$), in \cite{FLS2021} the authors proved strong regularities for the solutions of the system, when considering singular potentials and degenerate mobility. 
Finally, we also cite \cite{SS2021}, where, with an additional regularising term $\alpha \partial_t \mu$ in \eqref{eq:phi} and $\tau \gs 0$, the authors proved strong well-posedness for a similar system with full chemotaxis, constant mobility and both regular and singular potentials. 
However, they consider a different reaction term, introduced by Garcke et al. (see \cite{GLSS2016, GL2017}) through thermodynamic and kinetic considerations. 
Moreover, they also perform an asymptotic analysis for $\alpha \to 0$, but they are not able to recover the full strong regularity for the solutions, and only for regular potentials, if part of the chemotactic effect is neglected. 
Regarding these themes, the main novelty of this work is to prove strong well-posedness for \eqref{eq:phi}--\eqref{ic}, in the case of constant mobilities, with full chemotaxis, positive viscosity and both regular and singular potentials, i.e.~for the system:
\begin{alignat}{2}
	& \partial_t \phi = m \Delta \mu + P(\phi) (\sigma + \chi (1-\phi) - \mu) - \hh(\phi) u \qquad && \text{in } Q_T,  \label{eq:phicost}\\
	& \mu = \tau \partial_t \phi + AF'(\phi) + Ba \phi - BJ \ast \phi - \chi \sigma \qquad && \text{in } Q_T,  \label{eq:mucost} \\
	& \partial_t \sigma = n (\Delta \sigma - \chi \Delta \phi) - P(\phi) (\sigma + \chi (1-\phi) - \mu) + v \qquad && \text{in } Q_T, \label{eq:sigmacost} \\
	& m \partial_{\n} \mu = n \partial_{\n} (\sigma - \chi \phi) = 0 \qquad && \text{on } \partial \Omega \times (0,T), \label{bccost} \\
	& \phi(0) = \phi_0, \quad \sigma(0) = \sigma_0 \qquad && \text{in } \Omega, \label{iccost}
\end{alignat}
where we take $m(\phi) \equiv m \gs 0$ and $n(\phi) \equiv n \gs 0$ constant.
This can be done without resorting to the additional regularisation $\alpha \partial_t \mu$, used in \cite{SS2021}, which, in general, has little physical meaning.
The crucial point is that our different reaction term, with respect to \cite{SS2021}, together with the additional assumption that $P$ is strictly positive, still allows us to prove some further regularity on $\mu$. For more details, see Theorems \ref{thm:weaksols}, \ref{thm:strongsols} and \ref{thm:contdep} below. 

The other main novelty, as well as the key objective of this paper, is the analysis of an optimal distributed control problem on a non-local tumour growth model, with different therapy techniques. Indeed, up to the author's knowledge, this is the first time that such a problem is considered for tumour growth phenomena in a non-local framework (see for instance \cite{FGS2020} for other optimal control problems in non-local phase-field systems related to different applications). 
While we can count many contributions in the local case (see for instance \cite{CGRS2017, CSS2021_secondorder, GLR2018}), in our setting we only have some applications to parameter identification (see \cite{RSS2021}) or inverse identification problems (see \cite{FLS2021}). 
In general, optimal control problems are being widely studied in applied mathematics related to bio-medical issues due to their applicability in predicting and directing possible effects of therapies, we cite for instance \cite{CGLMRR2021, PCLT2018}. 
In the present work, we consider the following optimal control problem:   

\bigskip
\noindent(CP) \textit{Minimise the cost functional}
\begin{equation}
	\begin{split}
		\mathcal{J}(\phi, \sigma, u, v) & = \, \frac{\alpha_{\Omega}}{2} \int_{\Omega} |\phi(T) - \phi_{\Omega}|^2 \,\de x + \frac{\alpha_Q}{2} \int_{0}^{T} \int_{\Omega} |\phi - \phi_Q|^2 \,\de x\,\de t \\
		& \quad + \frac{\beta_{\Omega}}{2} \int_{\Omega} |\sigma(T) - \sigma_{\Omega}|^2 \,\de x + \frac{\beta_Q}{2} \int_{0}^{T} \int_{\Omega} |\sigma - \sigma_Q|^2 \,\de x\,\de t \\ 
		& \quad + \frac{\alpha_u}{2} \int_{0}^{T} \int_{\Omega} |u|^2 \,\de x\,\de t + \frac{\beta_v}{2} \int_{0}^{T} \int_{\Omega} |v|^2 \,\de x\,\de t,
	\end{split}
\end{equation}
\textit{subject to the control constraints}
\begin{equation} 
	\begin{split}
	& u \in \mathcal{U}_{ad} := \{ u \in L^{\infty}(Q_T) \cap H^1(0,T;L^2(\Omega)) \mid u_{\text{min}} \le u \le u_{\text{max}} \text{ a.e. in } Q_T, \\ 
	& \qquad \qquad \qquad \norm{u}_{H^1(0,T;\Lx 2)} \le M \}, \\
	& v \in \mathcal{V}_{ad} := \{ v \in L^{\infty}(Q_T) \mid v_{\text{min}} \le v \le v_{\text{max}} \text{ a.e. in } Q_T \},
	\end{split}
\end{equation}
\textit{and to the state system \eqref{eq:phicost}-\eqref{iccost} with constant mobilities}.
\bigskip

\noindent Here $\alpha_\Omega, \alpha_Q, \beta_\Omega, \beta_Q, \alpha_u, \beta_v$ are non-negative parameters that can be used to select which targets have to be privileged. Some, but not all of them, can clearly be equal to zero, and their relative ratios determine which strategy is being pursued. 
The function $\phi_\Omega$ is a final target for the tumour distribution, for instance one that could be suitable for surgery or possibly one that, after some time, stabilises on a certain shape; whereas $\phi_Q$ is a desired evolution. 
In the same way, $\sigma_\Omega$ and $\sigma_Q$ are respectively a final target and a desired evolution for the nutrient. 
Finally, the last two terms in the cost functional penalise large use of radiotherapy or drugs, which could still harm the patient. 
The aim of the optimal control problem is, then, to find the best therapies $u$ and $v$, which can lead the evolution of the tumour to the desired targets.  
Indeed, this is a quite standard tracking-type cost functional, which mostly measures the $L^2$ distance between the current solutions and some fixed targets. 
Regarding (CP), in Theorem \ref{thm:excont} we show the existence of an optimal pair $(\ub, \vb)$. Then, through careful analysis of the control-to-state operator and the introduction of the adjoint system, in Theorem \ref{thm:optcond} we prove the first-order necessary optimality conditions, which have to be satisfied by the optimal pair. 
We would like to stress that proving global boundedness of the solution in the right function spaces, as well as the separation property for the tumour phase parameter, is of utmost importance when analysing the control-to-state operator. 
On this theme, let us point out that in Theorem \ref{thm:frechet} we prove the Fr\'echet differentiability of said operator. While G\^ateaux-differentiability would be sufficient to derive first-order necessary conditions, we choose to prove a stronger notion of differentiability since our regularity setting allows to get it with relative ease. The point is that the key step in getting higher regularity for the solution is proving the separation property, and, once one has that, all the rest follows practically for free. However, one cannot do without the separation property, even if resorting to G\^ateaux-differentiability, since for example it is needed for the first continuous dependence estimate, as well as the well-posedness of the linearised system (see Theorems \ref{thm:contdep} and \ref{thm:linearised}).
Finally, let us remark that one could go further with the study of the optimal control problem (CP) and try to prove also second-order sufficient conditions, as done for instance in \cite{CSS2021_secondorder}. While very interesting from an applicative point of view, this is presently out of the scope of this paper and is left to further investigations.

The plan of the paper is the following. 
In section \ref{sec:prelim}, we introduce the notation and the main hypotheses that we are going to use throughout the paper, and we also prove the existence of weak solutions to \eqref{eq:phi}--\eqref{ic}. 
In section \ref{sec:strongsols}, we focus on strong well-posedness of the system with constant mobilities, by proving regularity and continuous dependence results. 
In section \ref{sec:optcont}, we shift our attention to the optimal control problem, by first proving an existence result. 
Then, we introduce the linearised system and prove the Fréchet-differentiability of the control-to-state operator. 
Finally, after studying the adjoint system, we shall prove the main result of the paper, namely the first-order necessary optimality conditions for (CP).  

\section{Preliminaries and existence of weak solutions}
\label{sec:prelim}

We now introduce some notation that will be used throughout the paper. We denote with $\Omega \subset \R^N$, $N=2,3$ an open bounded domain with boundary $\partial \Omega$ of class $\mathcal{C}^2$, whereas $T \gs 0$ is a fixed final time. 
The $\mathcal{C}^2$ requirement for $\partial \Omega$ is needed for elliptic regularity estimates in Section \ref{sec:strongsols}, while for weak solutions one can just assume that $\partial \Omega$ is Lipschitz. 
For convenience, we also denote $Q_t = \Omega \times (0,t)$, for any $t \in (0,T]$.  

Next, we recall the usual conventions regarding the Hilbertian triplet used in this context. 
If we define
\[ H = L^2(\Omega), \quad V = H^1(\Omega), \quad W = \{ u \in H^2(\Omega) \mid \partial_{\n} u = 0 \text{ on } \partial \Omega \}, \]
then we have the continuous and dense embeddings:
\[ W \hookrightarrow V \hookrightarrow H \cong H^* \hookrightarrow V^* \hookrightarrow W^*. \]
We denote by $\duality{\cdot, \cdot}_V$ the duality pairing between $V^*$ and $V$ and by $(\cdot, \cdot)_H$ the scalar product in $H$.
Regarding Lebesgue and Sobolev spaces, we will use the notation $\norm{\cdot}_{p}$ for the $\Lx p$-norm and $\norm{\cdot}_{k,p}$ for the $\Wx {k,p}$-norm, with $k \in \N$ and $1 \le p \le \infty$. 
Moreover, we observe that, by elliptic regularity theory, an equivalent norm on $W$ is
\[ \norm{u}^2_W = \norm{u}^2_H + \norm{\Delta u}^2_H. \]
Finally, we recall the Riesz isomorphism $\mathcal{N}: V \to V^*$:
\[ \duality{\mathcal{N}u, v}_V := \int_{\Omega} \nabla u \cdot \nabla v + uv \, \de x \quad \forall u,v \in V. \]
It is well-known that for $u\in W$ we have $\mathcal{N}u = - \Delta u + u \in H$ and that the restriction of $\mathcal{N}$ to $W$ is an isomorphism from $W$ to $H$. 
Additionally, by the spectral theorem, there exists a sequence of eigenvalues $0 < \lambda_1 \le \lambda_2 \le \dots$, with $\lambda_j \to +\infty$, and a family of eigenfunctions $w_j \in W$ such that $\mathcal{N}w_j = \lambda_j w_j$, which forms an orthonormal basis in $H$ and an orthogonal basis in $V$. 
In particular, $w_1$ is constant.

Finally, we recall some useful inequalities that will be used throughout the paper: 
	\begin{itemize}
	\item \emph{Gagliardo-Nirenberg inequality}. Let $\Omega \subset \R^N$ bounded Lipschitz, $m\in\N$, $1 \le r,q \le \infty$, $j\in\N$ with $0\le j \le m$ and $j/m \le \alpha \le 1$ such that
	\[ \frac{1}{p} = \frac{j}{N} + \left( \frac{1}{r} - \frac{m}{N} \right) \alpha + \frac{1-\alpha}{q}, \]
	then \[ \norm{D^j f}_{L^p(\Omega)} \le C \, \norm{f}^\alpha_{W^{m,r}(\Omega)} \norm{f}^{1-\alpha}_{L^q(\Omega)}. \]
	In particular, we recall the following versions with $N=2,3$:
	\begin{equation}
	\label{gn:ineq}
		\begin{split}
		& \norm{f}_{\Lx4} \le C \norm{f}^{1/2}_{\Hx1} \norm{f}^{1/2}_{\Lx2} \quad \text{ if } N = 2, \\
		& \norm{f}_{\Lx3} \le C \norm{f}^{1/2}_{\Hx1} \norm{f}^{1/2}_{\Lx2} \quad \text{ if } N = 3.
		\end{split}
	\end{equation}
	\item \emph{Agmon's inequality}. Let $\Omega \subset \R^N$ bounded Lipschitz, $0 \le k_1 < N/2 < k_2$ and $0 < \alpha <1$ such that $N/2 = \alpha k_1 + (1-\alpha) k_2$, then
	\[ \norm{f}_{L^\infty(\Omega)} \le C \, \norm{f}^\alpha_{H^{k_1}(\Omega)} \norm{f}^{1-\alpha}_{H^{k_2}(\Omega)}. \]
	In particular, we recall the following versions with $N=2,3$:
	\begin{equation}
	\label{agmon}
		\begin{split}
		& \norm{f}_{\Lx\infty} \le C \norm{f}^{1/2}_{\Hx1} \norm{f}^{1/2}_{\Lx2} \quad \text{ if } N = 2, \\
		& \norm{f}_{\Lx\infty} \le C \norm{f}^{1/2}_{\Hx2} \norm{f}^{1/2}_{\Hx1} \quad \text{ if } N = 3.
		\end{split}
	\end{equation}
\end{itemize}
Note that all constants $C > 0$ mentioned above depend only on the measures of the sets and the parameters, not on the actual functions. 

Now we introduce the structural assumptions on the parameters of our model \eqref{eq:phi}--\eqref{ic}:
\begin{enumerate}[font = \bfseries, label = A\arabic*., ref = \bf{A\arabic*}]
	\item\label{ass:coeff} $A,B>0$, $\tau > 0$, $\chi \ge 0$.
	\item\label{ass:j} $J \in W^{1,1}_{\text{loc}}(\R^N)$ is a symmetric convolution kernel, namely $J(z) = J(-z)$ for any $z \in \R^N$. Moreover, we suppose that
	\[ a(x) := (J \ast 1)(x) = \int_{\Omega} J(x-y) \, \de y \ge 0 \quad \text{a.e. } x \in \Omega, \]
	and also that
	\[	a^\ast := \sup_{x\in\Omega} \, \int_{\Omega} \abs{J(x-y)} \, \de y < +\infty, \qquad
	b := \sup_{x\in\Omega} \, \int_{\Omega} \abs{\nabla J(x-y)} \, \de y < +\infty.  \]
	\item\label{ass:yosida} $F: \R \to \R \cup \{+\infty\}$ can be written as 
	\[ F = F_1 + F_2,  \]
	where $F_1$ is such that the effective domain of $F_1$ is $\dom F_1 = [-l,l]$, for some $l \in (0,+\infty]$, and it is extended by $+\infty$ outside of $[-l,l]$. Moreover $F_1 \in \mathcal{C}^2((-l,l)) \cap \mathcal{C}^0([-l,l])$ and it is convex, i.e. $F_1''(s) \ge 0$ for any $s \in (-l,l)$. \\ Conversely, $F_2 \in \mathcal{C}^2(\R)$, with $F_2': \R \to \R$ Lipschitz and $F_2'(0)=0$.
	\item\label{ass:fc0} There exists $c_0 > 0$ such that
	\[ A F''(s) + B a(x) \ge c_0 \quad \forall s \in (-l,l) \quad \text{a.e. } x \in \Omega. \] 
	\item\label{ass:p} $P \in \mathcal{C}^0(\R) \cap L^\infty(\R)$ and there exists $P_0 > 0$ such that 
		\[ P(s) \ge P_0 > 0 \quad \forall s \in \R. \]   
	Moreover, call $P_\infty := \norm{P}_{L^\infty(\R)}$.
	\item\label{ass:mn} $m,n \in \mathcal{C}^0(\R)$ and there exist $\underline{m}, \overline{m}, \underline{n}, \overline{n} \in \R_+$ such that: 
	\[ 0<\underline{m} \le m(s) \le \overline{m} < +\infty \quad \text{and} \quad 0<\underline{n} \le n(s) \le \overline{n} < +\infty \quad \forall s\in \R. \] 
	\item\label{ass:uv} $u \in L^2(0,T;H)$, $\hh \in \mathcal{C}^0 \cap L^\infty(\R)$ and $v \in L^2(0,T;V^*)$.	Moreover, call $\hh_\infty := \norm{\hh}_{L^\infty(\R)}$.
	\item\label{ass:initial} $\phi_0 \in V$ with $F(\phi_0) \in L^1(\Omega)$ and $\sigma_0 \in H$.
\end{enumerate}

\begin{remark}
	Regarding hypotheses \ref{ass:yosida} and \ref{ass:fc0} on the potential, we wish to point out the following aspects. Firstly, $l=1$ corresponds to the physical case and here $F$ can be taken as the logarithmic potential $F_{\text{sing}}$. Moreover, if $l=+\infty$, $F$ is actually a regular potential. 
	Due to the strong hypothesis \ref{ass:p} on the proliferation function $P$, in principle $F$ could have any type of growth at infinity. 
	However, by assuming polynomial growth for $F$ as in \cite{FLR2017}, one can relax the hypothesis on $P$.
	Indeed, $P$ can have polynomial growth up to a certain order and the hypothesis $P_0 \gs 0$ can be avoided. 
	This will be more clear in Remark \ref{regular:potential}. 
	Secondly, hypothesis \ref{ass:yosida} is needed to work with a Yosida approximation of the singular potential within the proof. 
	In particular, the logarithmic singular potential shown before, and also the polynomial double-well, satisfy this hypothesis. 
\end{remark}

\begin{remark}
	Concerning hypothesis \ref{ass:j}, instead, it is easily seen that Newton and Bessel type kernels satisfy these requirements (see \cite{BRB2011}). 
	Another class of kernels which satisfy \ref{ass:j} and \ref{ass:fc0} is for example the one used in  \cite{DST2021_1} for non-local-to-local asymptotics. 
	Without going into the details, they are a class of kernels $J_\eps$, depending on $\eps \gs 0$, for which $a^*_\eps$ and $b_\eps$ are finite for any fixed $\eps$, therefore $J_\eps \in W^{1,1}_{\text{loc}}(\R^N)$, but they blow up to $+\infty$ as $\eps \to 0$. 
	In this way, also \ref{ass:fc0} is satisfied if $\eps$ is small enough. 
\end{remark}

Finally, we would like to stress that, in the following, we will extensively use the symbol $C>0$ to denote positive constants, which may change from line to line. 
They will depend only on $\Omega$, $T$, the parameters and on the norms of the fixed functions introduced in hypotheses \ref{ass:coeff}--\ref{ass:initial} and possible subsequent ones. 
Sometimes, we will also add subscripts on $C$ to highlight some particular dependences of these constants. 
Above all, when possible, we will focus on the dependence on the parameter $\tau$, since one may be interested in sending it to $0$ in some future applications.

We can now start with the first result about existence of global weak solutions for our tumour growth system \eqref{eq:phi}--\eqref{ic}.

\begin{theorem}
	\label{thm:weaksols}
	Under assumptions \emph{\ref{ass:coeff}--\ref{ass:initial}}, there exists a weak solution $(\phi, \mu, \sigma)$ to \eqref{eq:phi}--\eqref{ic}, with 
	\begin{align*}
		& \phi \in H^1(0,T;H) \cap L^\infty(0,T;V), \\
		& \mu \in L^2(0,T;V), \\
		& \sigma \in H^1(0,T;V^*) \cap L^\infty(0,T,H) \cap L^2(0,T;V), 
	\end{align*}
	which satisfies 
	\[ \phi(0) = \phi_0 \quad \text{in } V \quad \text{and} \quad \sigma(0) = \sigma_0 \quad \text{in } H \]
	and the following variational formulation for a.e. $t \in (0,T)$ and for any $\zeta \in V$:
	\begin{align}
		& (\phi_t, \zeta)_H + (m(\phi) \nabla \mu, \nabla \zeta)_H = (P(\phi)(\sigma + \chi(1-\phi) - \mu), \zeta)_H - (\hh(\phi) u, \zeta)_H, \label{varform:phi} \\
		& (\mu,\zeta)_H = \tau (\phi_t, \zeta)_H + (AF'(\phi) + Ba \phi - BJ \ast \phi - \chi \sigma,\zeta)_H, \label{varform:mu} \\
		& \duality{\sigma_t,\zeta}_V + (n(\phi) \nabla (\sigma + \chi (1-\phi)), \nabla \zeta)_H = - (P(\phi)(\sigma + \chi(1-\phi) - \mu), \zeta)_H + \duality{v,\zeta}_V. \label{varform:sigma}
	\end{align}
	In particular, there exists a constant $C>0$, depending only on the parameters of the model and on the data $\phi_0$, $\sigma_0$ and $u,v$, such that: 
	\begin{equation}
		\label{weaknorms:est}
		\begin{split}
		& \tau \norm{\phi}_{H^1(0,T;H) \cap L^\infty(0,T,V)} + \norm{F(\phi)}_{L^\infty(0,T;L^1(\Omega))} + \norm{\mu}_{ L^2(0,T;V)} \\ 
		& \quad + \norm{\sigma}_{H^1(0,T;V^*) \cap L^\infty(0,T,H) \cap L^2(0,T;V)} \le C.
		\end{split}
	\end{equation}
\end{theorem}

\begin{remark}
		\label{regular:potential}
		In the case of a \emph{regular potential of polynomial growth}, one can still prove Theorem \ref{thm:weaksols} with less restricting hypotheses on the function $P$. Indeed, following \cite{FLR2017}, one can replace hypotheses \ref{ass:yosida}, \ref{ass:p} and \ref{ass:uv}, respectively, with the following: 
		\begin{enumerate}[font = \bfseries, label = A\arabic*'., ref = \bf{A\arabic*'}]
			\setcounter{enumi}{2} 
			\item\label{A3p} $F \in \mathcal{C}^2(\R)$ and there exist $c_1 \in \R$ and $c_2 > \frac{\chi^2}{A}$ such that 
			\[ F(s) \ge c_2 \abs{s}^2 - c_1 \quad \forall s \in \R. \]
			Moreover, there exist $c_3 >0$ and $c_4 \ge 0$ such that
			\[ \abs{F'(s)} \le c_3 F(s) + c_4 \quad \forall s \in \R. \]
			\addtocounter{enumi}{1}
			\item\label{A5p} $P \in \mathcal{C}^0(\R)$ and there exist $c_5 >0$ and $q \in [1,4]$ such that
			\[ 0 \le P(s) \le c_5 (1+\abs{s}^q) \quad \forall s \in \R. \]   
			\addtocounter{enumi}{1}
			\item\label{A7p} $u \in L^\infty(Q_T)$, $\hh \in \mathcal{C}^0 \cap L^\infty(\R)$ and $v \in L^2(0,T;V^*)$.	
		\end{enumerate} 
		In the following Remark \ref{regular:proof}, we provide some hints as to how to modify the proof in this case.
\end{remark}

\begin{remark}
	\label{strategy:wsols}
	Before starting the proof, we would like to briefly comment our strategy. As usual with Cahn-Hilliard type equations, the main issue is to recover in some way also an $L^2$ bound on $\mu$ from the standard energy estimate used in this context. This is generally done by means of an estimate on the mean value of $\mu$, which then allows the application of Poincaré-Wirtinger inequality. To do this, one either starts studying the evolution of the mean value of $\phi$, as in standard Cahn-Hilliard-Oono equations (see for instance \cite{DG2015, GGM2017, H2022}), or directly imposes some growth conditions on the potential $F$, as in \cite{FLR2017}. In our framework, both these possibilities have to be ruled out, since our reaction term is too complex to give some information on $\phi_\Omega(t)$ and we need uniform estimates independent on the regularisation of our possibly singular potential. Therefore, hypothesis \ref{ass:p} on the strict positivity of $P$, as well as its boundedness, comes into play, together with the structure of our reaction term. Both properties of the proliferation function are needed already in the first energy estimate. Clearly, this is an alternative way to the additional regularisation $\alpha \mu_t$ introduced in \cite{SS2021}.
\end{remark}

\begin{proof}[Proof of Theorem $\ref{thm:weaksols}$]
	For simplicity of exposition, we proceed with formal a priori estimates. 
	However, we would like to give some details on how to construct a nice approximation framework. 
	First, one needs to approximate the possibly singular potential with a one-parameter family of regular ones. 
	By hypothesis \ref{ass:yosida}, we can employ a Yosida approximation of $\partial F_1$, where with this symbol we denote the subdifferential in the sense of convex analysis, which coincides with $F_1'$ inside $(-l,l)$. 
	Moreover, we recall that we have extended $F_1$ to $+ \infty$ outside $[-l,l]$, so that it is proper, convex and lower-semicontinuous by hypothesis. 
	Then, we can approximate $\partial F_1$ with a family 
	\[ F'_{1,\lambda}: \R \to \R, \quad F'_{1,\lambda} := \frac{I - (I - \lambda \partial F_1)^{-1}}{\lambda}, \quad \lambda \gs 0,  \]
	where $I$ stands for the identity operator. 
	By the properties of the regularisation, we know that $F'_{1,\lambda}$ is $1/\lambda$-Lipschitz for any $\lambda \gs 0$. 
	Then, we employ the Moreau regularisation of $F_1$ by defining
	\[ F_{1,\lambda}: \R \to \R, \quad F_{1,\lambda}(s) := F_1(0) + \int_0^s F'_{1,\lambda}(r) \, \de r, \quad s \in \R. \]
	Finally, we set 
	\[ F_\lambda = F_{1,\lambda} + F_2. \]
	The second level of approximation is a Faedo-Galerkin discretisation scheme with discrete spaces $W_n$, generated by the eigenvectors of the operator $\mathcal{N}$.
	Note that, within this discretisation, all eigenfunctions satisfy homogeneous Neumann boundary conditions, and this turns useful when integration by parts is needed in the following estimates. 
	For more details about the fully approximated problem, we refer the reader to the proof of \cite[Theorem 2.1]{SS2021}. 
	Here, we assume to be working with both approximations, and we show uniform estimates on the variables, independent of the regularisation parameters. 
	For simplicity, however, we suppress the dependences on these additional parameters.

	For the main a priori energy estimate, we substitute $\zeta=\mu$ in \eqref{varform:phi}, $\zeta=-\phi_t$ in \eqref{varform:mu} and $\zeta=\sigma + \chi (1 - \phi)$ in \eqref{varform:sigma}. Note that all these substitutions are allowed in the discretisation framework. By adding the resulting identities together and using standard results on differentiation in Bochner spaces, we get: 
	\begin{equation}
		\begin{split}
			\label{wsols:est1}
			\ddt E(t) + \tau \norm{\phi_t}^2_H + \norm{\sqrt{m(\phi)} \nabla \mu }^2_H + \norm{ \sqrt{n(\phi)} \nabla (\sigma + \chi (1-\phi)) }^2_H \qquad \qquad \\
			+ \norm{ \sqrt{P(\phi)} (\sigma + \chi (1-\phi) - \mu))}^2_H = - (\hh(\phi) u, \mu)_H + \duality{v,\sigma + \chi (1-\phi)}_V, 
		\end{split}
	\end{equation}
	where
	\begin{equation*}
		\begin{split}
			E(t) & := \int_{\Omega} A F(\phi) + \frac{B}{2} a \phi^2 - \frac{B}{2} (J \ast \phi) \phi + \frac{\sigma^2}{2} + \chi \sigma (1-\phi) \, \de x  \\
			& := E_J(t) + \int_{\Omega} A F(\phi) + \frac{\sigma^2}{2} + \chi \sigma (1-\phi) \, \de x ,
		\end{split}
	\end{equation*}
	with 
	\[ E_J(t) := \int_{\Omega} \frac{B}{2} a \phi^2 - \frac{B}{2} (J \ast \phi) \phi \, \de x  = \int_{\Omega} \int_{\Omega} \frac{B}{4} J(x - y)(\phi(x) - \phi(y))^2 \, \de x  \, \de y. \]
	In particular, we used hypothesis \ref{ass:j} on the symmetry of $J$ to get the corresponding term in the expression of $E_J(t)$.
	Now, we can rewrite the term involving $P(\phi)$ and, by using \ref{ass:p}, Cauchy-Schwarz and Young's inequalities with carefully chosen constants, we can estimate it in the following way:
	\begin{align*}
		& \norm{ \sqrt{P(\phi)} (\sigma + \chi (1-\phi) - \mu))}^2_H  \\
		& \quad =  (P(\phi) \mu, \mu)_H + ( P(\phi) (\sigma + \chi (1-\phi)), \mu)_H + ( P(\phi) (\sigma + \chi (1-\phi) -\mu), \sigma + \chi(1 - \phi) )_H \\
		& \quad \ge P_0 \norm{\mu}^2_H 	- 2 P_\infty \norm{\sigma + \chi(1-\phi)}_H \norm{\mu}_H - P_\infty \norm{\sigma + \chi (1-\phi)}^2_H \\
		& \quad \ge \frac{P_0}{2} \norm{\mu}^2_H - C \norm{\phi}^2_H - C \norm{\sigma}^2_H.	
	\end{align*}
	Then, we can also treat the terms on the right-hand side of \eqref{wsols:est1} by means of Cauchy-Schwarz and Young's inequalities, together with hypothesis \ref{ass:uv}, indeed: 
	\begin{align*}
		& (\hh(\phi) u ,\mu)_H \le \hh_\infty \norm{u}_H \norm{\mu}_H \le \frac{P_0}{8} \norm{\mu}^2_H + C \norm{u}^2_H, \\
		& \duality{v, \sigma + \chi (1-\phi)}_V \le \norm{v}_{V^*} \norm{\sigma + \chi (1-\phi)}_V \le \frac{\underline{n}}{2} \norm{\sigma + \chi (1-\phi)}^2_V + C \norm{v}^2_{V^*}.
	\end{align*}
	Now, by using \ref{ass:j} and Cauchy-Schwarz and Young's inequalities, we can estimate the energy $E(t)$ from below as
	\begin{align*}
		E(t) & = A \int_\Omega F(\phi) \, \de x  + \frac{B}{2} \overbrace{ \norm{\sqrt{a} \phi}^2_H }^{\ge 0} + \frac{1}{2} \norm{\sigma}^2_H - \int_\Omega \frac{B}{2} (J \ast \phi) \phi \, \de x  + \int_\Omega \chi (1-\phi) \sigma \, \de x  \\
		& \ge A \int_\Omega F(\phi) \, \de x  + \frac{1}{2} \norm{\sigma}^2_H - \frac{B}{2} a^* \norm{\phi}_H^2 - \chi \norm{\sigma}_H \norm{1-\phi}_H \\
		& \ge A \int_\Omega F(\phi) \, \de x  + \frac{1}{4} \norm{\sigma}^2_H - \left( \frac{B}{2} a^* + \chi^2 \right) \norm{\phi}^2_H - C,  
	\end{align*}
	where $C$ depends only on the measure of $\Omega$ and $\chi$. In a similar way, one can also estimate the initial energy $E(0)$ from above as
	\[ E(0) \le A \norm{F(\phi_0)}_{1} + \left( B a^* + \frac{\chi^2}{2} \right) \norm{\phi_0}^2_H + \norm{\sigma_0}^2_H \le C < +\infty, \] 
	owing to assumption \ref{ass:initial}. 
	It may be useful to note that at the discretisation level, one has to work with $F_\lambda(\Pi_n \phi_0)$, where $\Pi_n \phi_0$ is the projection of $\phi_0$ onto the discrete spaces. 
	In this case, one has to use the fact that, by construction, $F_\lambda$ as at most quadratic growth and the hypothesis on $\phi_0 \in H$, together with the properties of projections. 
	Then, after passing to the limit in the Galerkin discretisation, one can also recover an upper bound independent of $\lambda$, by using the fact that $F_\lambda \le F$ by construction and the hypothesis $F(\phi_0) \in \Lx1$. 
	For more details, see \cite[Theorem 2.1]{SS2021}. From the previous inequality and by using \ref{ass:mn}, we get:
	\begin{equation}
	\label{wsols:est2}
		\begin{split}
			& \ddt \left( A \int_\Omega F(\phi) \, \de x  + \frac{1}{4} \norm{\sigma}^2_H - \left( \frac{B}{2} a^* + \chi^2 \right) \norm{\phi}^2_H \right) + \frac{P_0}{2} \norm{\mu}^2_H \\
			& \qquad + \underline{m} \norm{\nabla \mu}^2_H + \tau \norm{\phi_t}^2_H + \frac{\underline{n}}{2} \norm{\nabla (\sigma + \chi (1-\phi))}^2_H \\
			& \quad \le C \norm{\phi}^2_H + C \norm{\sigma}^2_H + C \norm{u}^2_H + C \norm{v}^2_{V^*} + C.
		\end{split}
	\end{equation}
	Now, to compensate the negative term $ - ( \frac{B}{2} a^* + \chi^2 ) \norm{\phi}^2_H $, we substitute $\zeta=\phi$ in \eqref{varform:phi} and we multiply both sides by a constant $M \gs 0$ such that 
	\[ M \gs B a^* + 2 \chi^2. \]
	Then, by using hypotheses \ref{ass:p}, \ref{ass:mn}, \ref{ass:uv} and Cauchy-Schwarz and Young's inequalities, we get:
	\begin{align}
	\label{wsols:est3}
		\frac{M}{2} \ddt \norm{\phi}^2_H & = - M (m(\phi) \nabla \mu, \nabla \phi)_H + M ( P(\phi)(\sigma + \chi (1-\phi)  -\mu), \phi )_H - M (\hh(\phi) u, \phi)_H \nonumber \\
		& \le \frac{\underline{m}}{4} \norm{\nabla \mu}^2_H + \frac{P_0}{8} \norm{\mu}^2_H + C \norm{\nabla \phi}^2_H + C \norm{\sigma}^2_H + C \norm{\phi}^2_H + C \norm{u}^2_H.
	\end{align}
	Finally, to compensate also the term $C \norm{\nabla \phi}^2_H$, we substitute $\zeta = - \Delta \phi$ in \eqref{eq:mu}, which is possible within Galerkin's discretisation, and, integrating by parts, without getting any extra boundary terms, due to the discrete spaces, we obtain:
	\begin{align*}
		(\nabla \mu, \nabla \phi)_H & = \tau (\nabla \phi_t, \nabla \phi)_H + ((AF''(\phi) + Ba)\nabla \phi, \nabla \phi)_H \\
		& \quad + (B\nabla a \,\phi, \nabla \phi)_H - (B \nabla J \ast \phi, \nabla \phi)_H - \chi (\nabla \sigma, \nabla \phi)_H.
	\end{align*}
	Next, we add and subtract $\chi^2 \norm{\nabla \phi}^2$ on the right-hand side, and by using \ref{ass:j}, \ref{ass:fc0}, Cauchy-Schwarz and Young's inequalities, we obtain that 
	\begin{equation}
	\label{wsols:est4}
		\begin{split}
		& c_0 \norm{\nabla \phi}_H^2 + \frac{\tau}{2} \ddt \norm{\nabla \phi}^2_H  \\
		& \quad \le \frac{\underline{m}}{4} \norm{\nabla \mu}^2_H + \frac{\underline{n}}{4} \norm{\nabla (\sigma + \chi (1-\phi))}^2_H + C \norm{\nabla \phi}^2_H + C \norm{\phi}^2_H.  
		\end{split}
	\end{equation}
	At this point, we can add together \eqref{wsols:est2}, \eqref{wsols:est3} and \eqref{wsols:est4} and integrate on $(0,t)$, for any $t \in (0,T)$ to get: 
	\begin{align*}
		& \left( \frac{M}{2} - \frac{B}{2} a^* - \chi^2 \right) \norm{\phi(t)}^2_H + \frac{1}{4} \norm{\sigma(t)}^2_H + A \int_\Omega F(\phi(t)) \, \de x  + \frac{\tau}{2} \norm{\nabla \phi (t)}^2_H \\
		& \qquad + \tau \int_0^t \norm{\phi_t}^2_H \, \de s + \frac{P_0}{4} \int_0^t \norm{\mu}^2_H \, \de s + \frac{\underline{m}}{2} \int_0^t \norm{\nabla \mu}^2_H \, \de s + \frac{\underline{n}}{4} \int_0^t \norm{\nabla (\sigma + \chi (1-\phi))}^2_H \, \de s \\
		& \quad \le C E(0) + C \int_0^T \norm{\phi}^2_V \, \de s + C \int_0^T \norm{\sigma}^2_H \, \de s + C \int_0^T \norm{u}^2_H \, \de s + C \int_0^T \norm{v}^2_{V^*} \, \de s + CT.
	\end{align*}
	Now, we can apply Gronwall's inequality and infer that 
	\begin{equation}
		\label{wsols:est5}
		\begin{split}
			& \tau \norm{\phi}_{H^1(0,T;H) \cap L^\infty(0,T,V)} + \norm{F(\phi)}_{L^\infty(0,T;L^1(\Omega))} + \norm{\mu}_{ L^2(0,T;V)} \\ 
			& \quad + \norm{\sigma}_{L^\infty(0,T,H)} + \norm{\sigma + \chi (1-\phi)}_{L^2(0,T;V)} \le \bar{C},
		\end{split}
	\end{equation}
	where the constant $\bar{C}$ depends only on $\Omega$, $T$ and all the parameters introduced in hypotheses \ref{ass:coeff}--\ref{ass:initial}, but not on the Galerkin approximation parameter $n$. 
	At this stage, $\bar{C}$ can still depend on $\lambda$, but this dependence can be eliminated after passing to the limit as $n\to +\infty$, as explained before. 
	Moreover, it is straight-forwards to deduce that also 
	\[ \norm{\sigma}_{L^2(0,T;V)} \le \bar{C}. \]
	Finally, by comparison in \eqref{eq:sigma}, from \eqref{wsols:est5} it follows that 
	\[ \norm{\sigma}_{H^1(0,T;V^*)} \le \bar{C}. \]
	Note that, to be rigorous within Galerkin's approximation, one would have to test \eqref{eq:sigma} for any $\zeta \in W_n$ and use the orthogonality properties of the eigenvectors of $\mathcal{N}$, in order to recover the same estimate.
	
	Now, all that remains to do is to pass to the limit in the approximation framework. Indeed, by the uniform estimates found before, one can extract weakly and strongly convergent subsequences, that can be used to pass to the limit in the discretised variational form as $n \to +\infty$. Then, by weak lower-semicontinuity and the properties of the Yosida approximation, one can recover uniform estimates in $\lambda$ and then pass to the limit as $\lambda \searrow 0$. This is done very similarly to \cite[Theorem 2.1]{SS2021}, so we leave the details to the interested reader. 
	Finally, regarding initial data, we observe that $\phi \in \HT1H \cap \LT\infty V \hookrightarrow \Cw0V$  and $\sigma \in \HT1{V^*} \cap \LT2V \hookrightarrow \C0H$ by standard results, so they make sense respectively in $V$ and $H$.
	This concludes the proof of Theorem \ref{thm:weaksols}. 
\end{proof}
			
\begin{remark}
	In case of a singular potential, i.e. for $l \ls +\infty$, the uniform boundedness of $F(\phi)$ in $\LT \infty {\Lx1}$ readily implies that $\abs{\phi(x,t)} \le l$ for a.e. $(x,t) \in Q_T$, since $F$ is extended to $+\infty$ outside of $[-l,l]$. 
	Moreover, if one also assumes that 
	\[ \lim_{s \to (\pm l)^{\mp}} F'(s) = \pm \infty, \]
	which is the case for the logarithmic potential, one can also recover that $\abs{\phi(x,t)} \ls l$ for a.e. $(x,t) \in Q_T$. 
	Indeed, by comparison in \eqref{eq:mu}, since $\mu \in \LT 2 H$, also $F'(\phi) \in \LT 2 H$, which readily implies the thesis. 
	By improving this kind of hypothesis (see \ref{ass:flim}), in Section \ref{sec:strongsols} we will also be able to prove the so-called strict separation property for $\phi$.
\end{remark}

\begin{remark}
	\label{regular:proof}
	As anticipated before in Remark \ref{regular:potential}, we would like to give some hints regarding the differences in the proof when considering only a regular potential of polynomial growth, actually leaving most of the details to the interested reader. 
	
	The main energy estimate is still \eqref{wsols:est1}, however this time one can keep the reaction term $\norm{ \sqrt{P(\phi)} (\sigma + \chi (1-\phi) - \mu))}^2_H$ on the left-hand side and estimate differently the energy $E(t)$ and the term $(\hh(\phi) u, \mu)_H$. Indeed, by using H\"older and Young's inequalities and the new hypothesis \ref{A3p}, one has:  
	\begin{align*}
		E(t) & \ge E_J(t) + \frac{A}{2} \int_\Omega F(\phi) \, \de x  + \frac{A}{2} \int_\Omega F(\phi) \, \de x  + \frac{1}{2} \norm{\sigma}^2_H - \chi \norm{\sigma}_H \norm{1-\phi}_H \\
		& \ge E_J(t) + \frac{A}{2} \int_\Omega F(\phi) \, \de x  + \left( \frac{A}{2} c_2 - \frac{\chi}{4\alpha} \right) \norm{\phi}^2_H + \left( \frac{1}{2} - \chi \alpha \right) \norm{\sigma}^2_H - C,
	\end{align*}
	where this holds for any $\alpha > 0$. Now, if $\chi = 0$, one keeps the estimate as it is, otherwise if $\chi >0$, one can choose $\alpha = \frac{1}{\chi} \left( \frac{1}{2} - \delta \right)$ with $\delta \in \left( 0, \frac{1}{2} \right)$ and obtain: 
	\[ E(t) \ge E_J(t) + \frac{A}{2} \int_\Omega F(\phi) \, \de x   + \underbrace{ \left( \frac{A}{2} c_2  - \frac{\chi^2}{4(1/2 - \delta)} \right) }_{ := \gamma \gs 0} \norm{\phi}^2_H + \delta \norm{\sigma}^2_H - C. \]
	 Then, one can test \eqref{eq:mu} with $\zeta=1$, which is also possible within Galerkin's discretisation, since the first eigenfunction of $\mathcal{N}$ is constant, and integrate on $(0,t)$, for any $t\in (0,T)$. Consequently, by using \ref{ass:j}, \ref{A3p}, the continuous embedding of $L^2(\Omega)$ into $L^1(\Omega)$ and Young's inequality, one gets:
	\begin{align*}
		\int_0^t \int _\Omega \mu \, \de x  \, \de s & = \int_0^t \int_\Omega \tau \phi_t + F'(\phi) + Ba\phi - B J \ast \phi - \chi \sigma \, \de x  \, \de s \\
		& \le \int_0^t \left( \tau \norm{\phi_t}_{1} + \int_\Omega (c_3 F(\phi) + c_4) \, \de x  + 2Ba^* \norm{\phi}_{1} + \chi \norm{\sigma}_{1} \right) \, \de s \\
		& \le \eps \tau \int_0^t \norm{\phi_t}^2_H \, \de s + \int_0^t \left( c_3 \int_\Omega F(\phi) \, \de x  + C (\norm{\phi}^2_H + \norm{\sigma}^2_H) \right) \, \de s + C_{\eps}, 
	\end{align*}
	where $\eps >0$ is yet to be chosen. Therefore, by using also \ref{A7p}, one can estimate the source term as follows: 
	\begin{align*}
		\int_0^t (\hh(\phi) u, \mu)_H \, \de t & \le \hh_\infty \norm{u}_{L^\infty(Q_T)} \int_0^t \abs{ (\mu,1)_H } \, \de s \\
		& \le \int_0^t \frac{\tau}{2}  \norm{\phi_t}^2_H \, \de s + \int_0^t \left( c_3 \int_\Omega F(\phi) \, \de x  + C ( \norm{\phi}^2_H + \norm{\sigma}^2_H ) \right) \, \de s + C,
	\end{align*}
	by choosing $\eps = 1/(2 \, (\hh_\infty + \norm{u}_{L^\infty(Q_T)}))$ and renaming the constants. 
	Then, by integrating in time and using Gronwall's lemma, the following uniform estimate can be obtained:
	\begin{equation}
		\label{energy:est}
		\begin{split}
			& \tau \norm{\phi}_{H^1(0,T;H) \cap L^\infty (0,T;H)} + \norm{\sigma}_{L^\infty(0,T;H)} + \norm{\nabla \mu}_{L^2(0,T;H)} + \norm{E_J}_{L^\infty(0,T)} \\ 
			& \quad + \norm{\nabla(\sigma + \chi(1-\phi))}_{L^2(0,T;H)} + \norm{F(\phi)}_{L^\infty(0,T;L^1(\Omega))} \\ 
			& \quad + \norm{\sqrt{P(\phi)} (\sigma + \chi (1-\phi) - \mu)}_{L^2(0,T;H)} \le C. 
		\end{split}
	\end{equation}
	Moreover, observe that by testing again \eqref{eq:mu} with $\zeta=1$, one can see that
	\[ \frac{1}{\abs{\Omega}} \abs{\mu_\Omega} \le \tau \norm{\phi_t}_{1} + \norm{F(\phi)}_{1} + 2Ba^* \norm{\phi}_{1} + \chi \norm{\sigma}_{1} \in L^2(0,T), \]
	uniformly with respect to the parameters. Therefore, by using Poincaré-Wirtinger's inequality, it follows that also:
	\begin{equation*}
		\norm{\mu}_{L^2(0,T;V)} \le C.
	\end{equation*}
	Then, one tests \eqref{eq:mu} by $-\Delta \phi$ to recover the remaining $L^\infty(0,T;V)$ estimate on $\phi$ and proceeds by comparison for the $H^1(0,T;V^*)$ estimate for $\sigma$. Here we have the second main difference, which concerns the reaction term $R := P(\phi)(\sigma + \chi (1-\phi) -\mu)$. As a matter of fact, observe that formally, by using the embedding $L^{6/5}(\Omega) \hookrightarrow V^*$ and H\"older's inequality, one has:
	\begin{align*}
		\norm{R}_{V^*} & = \norm{P(\phi)(\sigma + \chi (1-\phi) -\mu)}_{V^*} \le C  \norm{P(\phi)(\sigma + \chi (1-\phi) -\mu)}_{6/5} \\
		& \le C \norm{\sqrt{P(\phi)}}_{3} \norm{\sqrt{P(\phi)}(\sigma + \chi (1-\phi) - \mu)}_H  \, \in L^2(0,T).
	\end{align*}	
	Indeed, one already knows that $\sqrt{P(\phi)}(\sigma + \chi (1-\phi) - \mu) \in L^2(0,T;H)$ by \eqref{energy:est} and, by using hypothesis \ref{A5p} with $q \le 4$, one can estimate:
	\begin{align*}
		\norm{\sqrt{P(\phi)}}_{3} & = \left( \int_\Omega P(\phi)^{3/2} \, \de x \right)^{1/3} \le \left( \int_\Omega (c_5(1+\abs{\phi}^q))^{3/2} \, \de x \right)^{1/3} \\
		& \le C + C \norm{\phi}^{q/2}_{3q/2}  \le C + C \norm{\phi}^2_{6} \, \in L^\infty(0,T),
	\end{align*}	
	thanks to the embedding $V \hookrightarrow L^6(\Omega)$. Hence, one obtains that
	\[ \norm{R}_{L^2(0,T;V^*)}  \le C, \]
	and by comparison in \eqref{eq:sigma} also:
	\[ \norm{\sigma_t}_{L^2(0,T;V^*)}  \le C.  \]
	Finally, one can similarly pass to the limit in the discretisation framework. For more details, we refer to \cite[Theorem 2.1]{FLR2017}, where the same system without viscosity is studied. This would conclude the proof of Theorem \ref{thm:weaksols}.
\end{remark}

\section{Strong well-posedness}
\label{sec:strongsols}

From now on, we consider the simplified version \eqref{eq:phicost}-\eqref{iccost} of our starting system, obtained by considering constant mobilities. Without loss of generality, we can fix the values of the constant mobilities as $ m = n = 1 $. Then, we have the following system: 
\begin{alignat}{2}
	& \partial_t \phi = \Delta \mu + P(\phi) (\sigma + \chi (1-\phi) - \mu) - \hh(\phi) u \qquad && \text{in } Q_T,  \label{eq:phi2}\\
	& \mu = \tau \partial_t \phi + AF'(\phi) + Ba \phi - BJ \ast \phi - \chi \sigma \qquad && \text{in } Q_T,  \label{eq:mu2} \\
	& \partial_t \sigma = \Delta \sigma - \chi \Delta \phi - P(\phi) (\sigma + \chi (1-\phi) - \mu) + v \qquad && \text{in } Q_T, \label{eq:sigma2}
\end{alignat}
paired with boundary and initial conditions:
\begin{alignat}{2}
	& \partial_{\n} \mu = \partial_{\n} (\sigma - \chi \phi) = 0 \qquad && \text{on } \partial \Omega \times (0,T), \label{bc2} \\
	& \phi(0) = \phi_0, \quad \sigma(0) = \sigma_0 \qquad && \text{in } \Omega. \label{ic2}
\end{alignat}
Moreover, we also assume stronger hypotheses on the parameters and on the data. 
First let us observe that, to gain further regularity, we would need to assume that $J \in  W^{2,1}_{\text{loc}}(\R^N)$, but this hypothesis is incompatible with widely used convolution kernels, such as those of Newton or Bessel type. 
However, following  \cite[Definition 1]{BRB2011}, we can still introduce a suitable class of kernels, which includes the ones mentioned before and satisfies our needs. 
Indeed, we recall the following definition:
\begin{definition}
	\label{def:admissible}
	A convolution kernel $J \in W^{1,1}_{\text{loc}}(\R^N)$ is \emph{admissible} if it satisfies the following conditions: 
	\begin{itemize}
		\item $J \in \mathcal{C}^3(\R^N \setminus \{0\})$.
		\item $J$ is radially symmetric and non-increasing, i.e. $J(\cdot) = \tilde{J}(\abs{\cdot})$ for a non-increasing function $\tilde{J} : \R_+ \to \R$.
		\item There exists $R_0$ such that $r \mapsto \tilde{J}''(r)$ and $r \mapsto \tilde{J}'(r)/r$ are monotone on $(0,R_0)$.
		\item There exists $C_N>0$ such that $\abs{D^3 J(x)} \le C_N \abs{x}^{-N-1}$ for any $x \in \R^3 \setminus \{0\}$.
	\end{itemize}
\end{definition}

We also need stronger hypotheses on $F$ and on the initial data to guarantee the strict separation property. Indeed, we assume the following:

\begin{enumerate}[font = \bfseries, label = B\arabic*., ref = \bf{B\arabic*}]
	\item\label{ass:j2} $J \in W^{2,1}_{\text{loc}}(\R^N)$ or $J$ is \emph{admissible} in the sense of Definition \ref{def:admissible}.
	\item\label{ass:flim} $F \in \mathcal{C}^4((-l,l))$ and the constant $c_0$ of hypothesis \ref{ass:fc0} is such that $c_0 > \chi^2 \ge 0$. Moreover, we assume that
	\[ \lim_{s \to (\pm l)^{\mp}} (AF'(s) - \chi^2 s) = \pm \infty. \]
	\item\label{ass:ph2} $P, \hh \in \mathcal{C}^1 \cap W^{1,\infty}(\R)$ and there exists $P_0 > 0$ such that
	\[ P(s) \ge P_0 > 0 \quad \forall s\in\R. \]
	Moreover, call $P'_\infty = \norm{P'}_{L^\infty(\R)}$ and $\hh'_\infty = \norm{\hh'}_{L^\infty(\R)}$.
	\item\label{ass:uv2} $u \in L^\infty(Q_T) \cap H^1(0,T;H)$ and $v \in L^\infty(Q_T)$.
	\item\label{ass:initial2} $\sigma_0 \in V \cap L^\infty(\Omega)$, whereas $\phi_0 \in H^2(\Omega)$ and it is \emph{separated}, i.e.~there exists $s_0 \in (0,l)$ such that 
	\[ \norm{\phi_0}_{\Lx\infty} \le s_0. \]
\end{enumerate}

\begin{remark}
	\label{admissible}
	We recall that if $J$ satisfies \ref{ass:j2}, then, by \cite[Lemma 2]{BRB2011}, for any $p \in (1,+\infty)$ there exists a constant $b_p>0$ such that:
	\[ \norm{\nabla (\nabla J \ast \psi)}_{L^p(\Omega)^{3\times 3}} \le b_p \norm{\psi}_{L^p(\Omega)} \quad \forall \psi \in L^p(\Omega). \]
	This is just what one needs to control second-order derivatives of the convolution term $J \ast \phi$.
\end{remark}

\begin{remark}
	Assumption \ref{ass:flim} is needed to have some coercivity, due to the presence of chemotaxis, and a control on how $F'$ blows up at the extrema, in order to prove the separation property.
	Moreover, on hypothesis \ref{ass:initial2}, note that $\phi_0$ being separated does not take away from practical situations. Indeed, if $F$ is regular, it simply means that $\phi_0 \in \Lx \infty$. Whereas if $F$ is singular, we recall that the actual minima of the logarithmic potential are not exactly in $\pm 1$, but in two intermediate values close to them (i.e.~the actual pure phases in this setting). Therefore, we are just asking that $F'(\phi_0)$ does not explode, in a uniform way.  
\end{remark}

\begin{remark}
	\label{data:reg}
	We also observe that, since $\phi_0 \in H^2(\Omega)$ and it is separated, one can freely differentiate the potential $F'(\phi_0)$. Therefore, by comparison in the variational formulations of \eqref{eq:phi2} and \eqref{eq:mu2} at time $0$, we also have that:
	\[ \mu(0) \in W  \quad \text{and} \quad \phi_t(0) \in H. \] 
	Indeed, by considering \eqref{varform:phi} and \eqref{varform:mu} at time $t=0$, we have that for any $\zeta \in V$
	\begin{align*}
		& (\phi_t(0), \zeta)_H + ( \nabla \mu(0), \nabla \zeta)_H = (P(\phi_0)(\sigma_0 + \chi(1-\phi_0) - \mu(0)), \zeta)_H - (\hh(\phi_0) u(0), \zeta)_H,  \\
		& (\mu(0),\zeta)_H = \tau (\phi_t(0), \zeta)_H + (AF'(\phi_0) + Ba \phi_0 - BJ \ast \phi_0 - \chi \sigma_0,\zeta)_H.
	\end{align*}
	Then, by substituting the first one into the second one and isolating the terms containing $\mu(0)$, we infer that $\mu(0)$ satisfies the following variational problem:
	\[ \tau (\nabla \mu(0), \nabla \zeta)_H + ((1+\tau P(\phi_0)) \mu(0), \zeta)_H = (f,\zeta)_H \quad \text{for any } \zeta \in V, \]
	where 
	\[ f := \tau P(\phi_0)(\sigma_0 + \chi(1-\phi_0)) - \tau \hh(\phi_0) u(0) + AF'(\phi_0) + Ba\phi_0 - B J \ast \phi_0 - \chi \sigma_0 \in H. \]
	Then, $\mu(0)$ is a weak solution of the following elliptic partial differential equation:
	\begin{alignat}{2}
		 - \tau \Delta \mu(0) + (1 + \tau P(\phi_0)) \mu(0) & = f \qquad && \hbox{in $\Omega$,} \\
		 \partial_\n \mu(0) & = 0 \qquad && \hbox{on $\partial \Omega$,}
	\end{alignat}
	where the homogeneous Neumann boundary condition is a consequence of the test functions $\zeta$ being in $V = \Hx 1$. Therefore, since, by \ref{ass:initial2}, \ref{ass:p} and Sobolev embeddings, we have that $f \in H$, $(1 + \tau P(\phi_0)) \in \Cx 0$ and it is non-negative, by standard elliptic regularity theory we can infer that $\mu(0) \in W$. Then, by comparison, we also have that $ \phi_t(0) \in H$. 
\end{remark}

We have the following result about existence of strong solutions:

\begin{theorem}
	\label{thm:strongsols}
	Under assumptions \emph{\ref{ass:coeff}--\ref{ass:fc0}} and \emph{\ref{ass:j2}--\ref{ass:initial2}}, there exists a strong solution to \eqref{eq:phi2}--\eqref{ic2}, with the following regularity:
	\begin{align*}
		& \phi \in W^{1,\infty}(0,T;H) \cap H^1(0,T;H^2(\Omega)) \cap L^\infty(0,T;H^2(\Omega)), \\
		& \mu \in L^\infty(0,T;W), \\
		& \sigma \in H^1(0,T;H) \cap L^\infty(0,T,V) \cap L^2(0,T;H^2(\Omega)). 
	\end{align*}	
	In particular, there exists a constant $C>0$, depending only on the parameters of the model and on the data $\phi_0$, $\sigma_0$ and $u,v$, such that: 
	\begin{equation}
		\begin{split}
		\label{strongnorms:est}
		 & \tau \norm{\phi}_{W^{1,\infty}(0,T;H) \cap H^1(0,T;H^2(\Omega)) \cap L^\infty(0,T;H^2(\Omega))} + \norm{\mu}_{L^\infty(0,T;W)} \\ 
		 & \quad + \norm{\sigma}_{H^1(0,T;H) \cap L^\infty(0,T,V) \cap L^2(0,T;H^2(\Omega))} \le C. 
		 \end{split}
	\end{equation}
	Moreover, $\phi$ satisfies the strict separation property, i.e. there exists $s^* \in (s_0, l)$ such that 
	\begin{equation}
		\label{separation}
		\sup_{t \in [0,T]} \norm{\phi(t)}_{\Lx\infty} \le s^*. 
	\end{equation}
\end{theorem}

\begin{proof}
	We proceed with only formal estimates for the sake of exposition. These can be made rigorous by going back to the Galerkin discretisation and by using finite differences operators for time derivatives of higher order. 
	
	Before starting, we observe that, from the weak regularities of Theorem \ref{thm:weaksols}, we can immediately have more regularity on some terms, without extra assumptions. In particular, regarding the reaction term, we can say that
	\[ \norm{R}_{L^2(0,T;L^6(\Omega))} \le C,  \]
	since $P \in L^\infty$ and $\sigma + \chi (1-\phi) - \mu \in L^2(0,T;V) \hookrightarrow L^2(0,T;L^6(\Omega))$. In turn, this implies, by comparison in \eqref{eq:phi2}, that
	\[ \Delta \mu = \phi_t - R + \hh(\phi) u \in L^2(0,T;H) , \]
	which means that we also have a uniform bound on $\mu$ in $L^2(0,T;W)$, i.e.
	\begin{equation}
		\label{mu:l2w}
		\norm{\mu}_{L^2(0,T;W)} \le C.
	\end{equation}
	
	For the first estimate, we test equation \eqref{eq:sigma2} by $\partial_t (\sigma - \chi \phi)$ and, by using Cauchy-Schwarz and Young's inequalities, we get:
	\begin{align*}
		& \norm{\sigma_t}^2_H + \frac{1}{2} \ddt \norm{\nabla (\sigma - \chi \phi)}^2_H \\
		& \quad = \chi (\sigma_t, \phi_t)_H + (P(\phi) (\sigma + \chi (1-\phi) - \mu) , \sigma_t - \chi \phi_t)_H + (v, \sigma_t - \chi \phi_t)_H \\
		& \quad \le \frac{1}{2} \norm{\sigma_t}^2_H + C \norm{\phi_t}^2_H + C \norm{\sigma + \chi (1-\phi) -\mu}^2_H + C \norm{v}^2_H. 
	\end{align*}
	Then, by integrating on $(0,t)$, for any $t\in (0,T)$, and by using Gronwall's lemma, hypothesis \ref{ass:initial2} and \eqref{weaknorms:est}, we infer that
	\begin{equation}
		\label{sigma:h1h}
		\norm{\sigma}_{H^1(0,T;H)} + \norm{\sigma - \chi \phi}_{L^\infty(0,T;V)} \le C.
	\end{equation}
	Also, since $\phi \in L^\infty(0,T;V)$ by \eqref{weaknorms:est}, we can conclude that
		\begin{equation}
		\label{sigma:linfv}
		\norm{\sigma}_{L^\infty(0,T;V)} \le C.
	\end{equation}
	Next, we test again \eqref{eq:sigma2} by $- \Delta (\sigma - \chi \phi)$ and with analogous estimates we get:
	\begin{equation*}
		\norm{\Delta (\sigma - \chi \phi)}^2_H \le \frac{1}{2} \norm{\Delta(\sigma-\chi\phi)}^2_H + C \norm{\sigma_t}^2_H + C \norm{\sigma + \chi (1-\phi) -\mu}^2_H + C \norm{v}^2_H.
	\end{equation*}
	Then, again by integrating on $(0,T)$ and using \eqref{weaknorms:est} and \eqref{sigma:h1h}, we infer that
	\begin{equation}
		\label{sigmaphi:l2w}
		\norm{\sigma - \chi \phi}_{L^2(0,T;W)} \le C.
	\end{equation}
	Now, for the main estimate, we test \eqref{eq:phi2} by $\partial_t \mu$ and the time-derivative of \eqref{eq:mu2} by $- \partial_t \phi$ and sum them up. To be precise, these time-derivatives should be done by using the approximation $\partial_t f \simeq (T_h(f) - f)/h$, where $T_h(f) = f(\cdot + h)$ for $h \in \R$, and then by sending $h\to 0$. However, to give the idea of the procedure, we stick to formal estimates. Indeed, after cancellations, we obtain:
	\begin{align*}
		& \frac{1}{2} \ddt \norm{\nabla \mu}^2_H + \frac{\tau}{2} \ddt \norm{\phi_t}^2_H + ((AF''(\phi) + Ba) \phi_t, \phi_t)_H \\
		& \quad = (P(\phi) (\sigma + \chi (1-\phi) - \mu), \mu_t)_H - (\hh(\phi)u, \mu_t)_H + B(J \ast \phi_t, \phi_t)_H + \chi (\sigma_t, \phi_t)_H. 
	\end{align*}
	By using assumptions \ref{ass:j}, \ref{ass:fc0} and Cauchy-Schwarz and Young's inequalities, we infer that
	\begin{align*}
		& \frac{1}{2} \ddt \norm{\nabla \mu}^2_H + \frac{\tau}{2} \ddt \norm{\phi_t}^2_H \\
		& \quad \le C \norm{\phi_t}^2_H + C \norm{\sigma_t}^2_H +  \underbrace{(P(\phi) (\sigma + \chi (1-\phi) - \mu), \mu_t)_H}_{:= I_1} - \underbrace{(\hh(\phi)u, \mu_t)_H}_{:= I_2}. 
	\end{align*}
	Now, in order to estimate the term $I_1$, we use Leibniz's rule in time, generalised H\"older's inequality, Young's inequality, hypotheses \ref{ass:ph2}, \ref{ass:uv2} and the embeddings $V \hookrightarrow L^6(\Omega)$ and $W \hookrightarrow L^\infty(\Omega)$, as follows:
	\begin{align*}
		I_1 & = \ddt (P(\phi) (\sigma + \chi (1-\phi)), \mu)_H 
		- (P'(\phi) \phi_t (\sigma + \chi (1-\phi)), \mu)_H 
		- (P(\phi) (\sigma_t - \chi \phi_t ), \mu )_H \\
		& \quad - \frac{1}{2} \ddt (P(\phi) \mu, \mu)_H 
		+ (P'(\phi) \phi_t \, \mu, \mu)_H 	\\	
		& \le - \ddt \left[ \frac{1}{2} (P(\phi)\mu, \mu)_H 
		- (P(\phi) (\sigma + \chi (1-\phi)), \mu)_H  \right] 
		+ P'_\infty \norm{\phi_t}_H \norm{\sigma + \chi (1-\phi)}_{4} \norm{\mu}_{4} \\
		& \quad + P_\infty \norm{\sigma_t - \chi \phi_t}_H \norm{\mu}_H + P'_\infty \norm{\phi_t}_H \norm{\mu}_{\infty} \norm{\mu}_H  \\
		& \le - \ddt \left[ \frac{1}{2} (P(\phi)\mu, \mu)_H - (P(\phi) (\sigma + \chi (1-\phi)), \mu)_H  \right] + C \norm{\mu}^2_V  \\
		& \quad + C \underbrace{\norm{\sigma + \chi (1-\phi)}^2_V}_{\in \, L^\infty(0,T)} \norm{\phi_t}^2_H 
		+ C \norm{\sigma_t - \chi \phi_t}^2_H 
		+ C \underbrace{\norm{\mu}^2_{\infty}}_{\in \, L^1(0,T)} \norm{\phi_t}^2_H.
	\end{align*}
	In a similar way, we also estimate the term $I_2$:
	\begin{align*}
		I_2 & = \ddt (\hh(\phi)u, \mu)_H - (\hh'(\phi) \phi_t \, u, \mu)_H - (\hh(\phi)u_t, \mu)_H \\
		& \le \ddt (\hh(\phi)u, \mu)_H + \hh'_\infty \norm{\phi_t}_H \norm{u}_{\infty} \norm{\mu}_H + \hh_\infty \norm{u_t}_H \norm{\mu}_H \\
		& \le \ddt (\hh(\phi)u, \mu)_H + C \norm{\mu}^2_H + C \norm{u}^2_{\infty} \norm{\phi_t}^2_H + C \norm{u_t}^2_H.
	\end{align*}	
	By putting all together, we have:
	\begin{align*}
		\frac{1}{2} \ddt \norm{\nabla \mu}^2_H + \frac{\tau}{2} \ddt \norm{\phi_t}^2_H  + \ddt G(t) & \le C \left( 1 + \norm{\sigma + \chi (1-\phi)}^2_V + \norm{\mu}^2_{\infty} + \norm{u}^2_{\infty}  \right) \norm{\phi_t}^2_H \\
		 & \quad + C \norm{\sigma_t - \chi \phi_t}^2_H + C \norm{\sigma_t}^2_H + C \norm{\mu}^2_V + C \norm{u_t}^2_H,
	\end{align*}
	where
	\[ G(t) = \frac{1}{2} (P(\phi)\mu, \mu)_H - (P(\phi) (\sigma + \chi (1-\phi)), \mu)_H - (\hh(\phi)u, \mu)_H. \]
	Now we can estimate function $G$ at time $t$ and at time $0$, by using \ref{ass:ph2}, Remark \ref{data:reg}, H\"older and Young's inequalities as follows:
	\begin{align*}
		& G(t) \ge P_0 \norm{\mu}^2_H - P_\infty \norm{\sigma + \chi (1-\phi)}_H \norm{\mu}_H - \hh_\infty \norm{u}_H \norm{\mu}_H \\
		& \qquad \ge \frac{P_0}{2} \norm{\mu}^2_H - C \norm{\sigma + \chi(1-\phi)}^2_{L^\infty(0,T;H)} - C \norm{u}^2_{L^\infty(Q_T)}, \\
		& G(0) \le P_\infty \norm{\mu(0)}^2_H + P_\infty \norm{\sigma_0 + \chi (1-\phi_0)}_H \norm{\mu(0)}_H + \hh_\infty \abs{\Omega} \norm{u}_{\Lqt \infty} \norm{\mu(0)}_H \le C.
	\end{align*}
	Finally, for some $0 < \alpha < \min\{P_0,1\}/2$, after integrating on $(0,t)$, for any $t \in (0,T)$ and using Remark \ref{data:reg}, we obtain that
	\begin{align*}
		& \alpha \norm{\mu(t)}^2_V + \frac{\tau}{2} \norm{\phi_t(t)}^2_H \\
		& \quad \le C + C \int_0^t \left( \norm{\sigma_t - \chi \phi_t}^2_H + \norm{\sigma_t}^2_H + \norm{u_t}^2_H \right) \, \de s + C \norm{\sigma + \chi(1-\phi)}^2_{L^\infty(0,T;H)} \\ 
		& \qquad + C \norm{u}^2_{L^\infty(Q_T)} + C \int_0^t \norm{\mu}^2_V \, \de s + C \int_0^t \underbrace{\left( 1 + \norm{\sigma + \chi (1-\phi)}^2_V + \norm{\mu}^2_{\infty} + \norm{u}^2_\infty  \right)}_{\in \, L^1(0,T)} \norm{\phi_t}^2_H \, \de s, 
	\end{align*}
	where the term in front of $\norm{\phi_t}^2_H$ on the right-hand side is in $L^1(0,T)$, thanks to \eqref{weaknorms:est}, \eqref{sigma:linfv}, \eqref{mu:l2w} and \ref{ass:initial2}. Recall also that $\norm{\mu(0)}_V \le C$ and $\norm{\phi_t(0)}_H \le C$ by Remark \ref{data:reg}.
	Then, by Gronwall's lemma, we infer the following uniform estimates:
	\begin{equation}
		\label{phi:w1infh}
		\tau \norm{\phi}_{W^{1,\infty}(0,T;H)} + \norm{\mu}_{L^\infty(0,T;V)} \le C.
	\end{equation}
	In particular, we now observe that 
	\[ R = P(\phi) (\sigma + \chi (1-\phi) - \mu) \in L^\infty(0,T;H), \]
	then, by comparison in \eqref{eq:phi2}, we have that
	\[ \Delta \mu = \phi_t - R - \hh(\phi) u \in L^\infty(0,T;H), \]
	where both these inclusions hold uniformly with respect to the parameters and the data. Therefore, we also have the uniform bound:
	\begin{equation}
		\label{mu:linfw}
		\norm{\mu}_{L^\infty(0,T;W)} \le C, 
	\end{equation}
	which in turn, thanks to the embedding $W \hookrightarrow L^\infty(\Omega)$, also implies that
	\[ \norm{\mu}_{L^\infty(Q_T)} \le C. \]
	
	Now, we mostly follow the procedure used in the proof of \cite[Theorem 2.5]{SS2021}, in order to prove the strict separation property and get more regularity for $\phi$. Even if the argument is very similar to the referenced one, we repeat it for the sake of completeness. First, we rewrite equation \eqref{eq:mu2} as 
	\[ \tau \partial_t \phi + B a \phi + A F'(\phi) = \mu + B J \ast \phi + \chi \sigma, \]
	then we sum the term $-\chi^2 \phi$ to both sides, to obtain the following evolution equation in $Q_T$:
	\begin{equation}
		\label{eq:mu2p}
		\tau \partial_t \phi + B a \phi + A F'(\phi) - \chi^2 \phi = S := \mu + B J \ast \phi + \chi (\sigma - \chi \phi). 
	\end{equation}
	Next, regarding the right-hand side, we observe that $\mu \in L^\infty(0,T;W)$ by \eqref{mu:linfw}, $\sigma - \chi \phi \in L^2(0,T;W)$ by \eqref{sigmaphi:l2w} and also $J \ast \phi \in L^\infty(0,T;H^2(\Omega))$ by \ref{ass:j2}, since $\phi \in L^\infty(0,T;H)$. Indeed, one just has to pass the derivatives onto the convolution kernel $J$, by using also Remark \ref{admissible}. Then, it follows that we have the uniform bound:
	\begin{equation}
		\label{S:l2w}
		\norm{S}_{L^2(0,T;H^2(\Omega))} \le C. 
	\end{equation}
	Additionally, by adding $-\chi \phi_t$ on both sides, we can rewrite \eqref{eq:sigma2} as
	\[ \partial_t (\sigma - \chi \phi) - \Delta (\sigma - \chi \phi) = - P(\phi) (\sigma + \chi (1-\phi) - \mu) - \chi \phi_t + v, \]
	where the right-hand side is in $\LT\infty H$ by \eqref{phi:w1infh}. Then, since $\sigma_0 - \chi \phi_0 \in \Lx \infty$ thanks to \ref{ass:initial2}, by parabolic regularity theory (see \cite[Theorem 7.1, page 181]{ladyzhenskaja}), it follows that $\sigma - \chi \phi$ is uniformly bounded in $\Lqt\infty$. Therefore, since $\Hx2 \hookrightarrow \Lx\infty$, we can actually infer that there exists a constant $\bar{M}$ such that
	 \begin{equation}
	 	\label{S:linfqt}
	 	\norm{S}_{\Lqt\infty} \le \bar{M}. 
	 \end{equation}
 	Now we can prove the strict separation property. By hypothesis \ref{ass:initial2}, we know that there exists $s_0 \in (0,l)$ such that $\norm{\phi_0}_{L^\infty(\Omega)} \le s_0$. Moreover, by hypothesis \ref{ass:flim}, we also know that 
	\[ \lim_{s \to (\pm l)^{\mp}} (AF'(s) - \chi^2 s) = \pm \infty. \]
	Therefore, there surely exists $s^* \in (s_0,l)$ such that 
	\begin{align}
		\label{sep:bounds}
		& AF'(s) - \chi^2 s \ge \bar{M} \quad \text{for any } s \in (s^*,l), \\
		& AF'(s) - \chi^2 s \le - \bar{M} \quad \text{for any } s \in (-l,-s^*), 
	\end{align} 
	where $\bar{M}$ is the one of \eqref{S:linfqt}. Now, we test \eqref{eq:mu2p} by $(\phi - s^*)_+$ in $H$ and we integrate on $(0,t)$, for any $t\in (0,T)$:
	\[ \tau \int_{Q_t} \phi_t (\phi - s^*)_+ \, \de x  \, \de s 
	+ B \int_{Q_t} a \phi (\phi - s^*)_+ \, \de x  \, \de s 
	= \int_{Q_t} (S - (AF'(\phi) - \chi^2 \phi)) (\phi - s^*)_+ \, \de x  \, \de s.  \]
	Then, observe that $(\phi - s^*)_+ = 0$ if $\phi < s^*$, therefore we can restrict the integrals on $Q_t \cap \{ \phi > s^* \}$. In this way, we can infer that 
	\begin{align*}
		& \tau \int_{Q_t \cap \{ \phi > s^* \}} \phi_t (\phi - s^*)_+ \, \de x  \, \de s 
		= \frac{\tau}{2} \int_{Q_t \cap \{ \phi > s^* \}} \ddt \norm{ (\phi - s^*)_+ }^2_H \, \de x  \, \de s
		= \frac{\tau}{2} \norm{ (\phi(t) - s^*)_+ }^2_H,
	\end{align*}
	where $(\phi_0 - s^*)_+ = 0$ by \ref{ass:initial2}, since $s^* > s_0$. Moreover, by \eqref{sep:bounds} and \eqref{S:linfqt}, it follows that 
	\[ \int_{Q_t \cap \{ \phi > s^* \}} (S - (AF'(\phi) - \chi^2 \phi)) (\phi - s^*)_+ \, \de x  \, \de s \le 0. \]
	Consequently, we have that 
	\[ \frac{\tau}{2} \norm{ (\phi(t) - s^*)_+ }^2_H 
	+ \int_{Q_t \cap \{ \phi > s^* \}} B a (\phi - s^*)^2_+ \, \de x  \, \de s \le 0, \]
	but the left-hand side is clearly non-negative, so the only possibility is that $(\phi(x,t) - s^*)_+ = 0$ for a.e.~$(x,t) \in Q_T$, which implies $\phi(x,t) \le s^*$ a.e.~in $Q_T$. By repeating the same procedure with $- (\phi + s^*)_+$, one can also recover the estimate from below, i.e.~$\phi(x,t) \ge - s^*$ for a.e. $(x,t) \in Q_T$. Therefore, the separation property \eqref{separation} is proved.
	
	Now, exactly as in \cite[Section 3.7]{SS2021}, we test the gradient of \eqref{eq:mu2p} by $\nabla \phi \abs{\nabla \phi}^{p-2}$ with $p>1$ to be set later. Note that, to be rigorous, here one should use a truncation argument, even within the Galerkin discretisation. However, by proceeding formally, we have:
	\begin{align*}
		& \tau (\nabla \phi_t, \nabla \phi \abs{\nabla \phi}^{p-2})_H + ( (AF''(\phi) + Ba) \nabla \phi, \nabla \phi \abs{\nabla \phi}^{p-2} )_H \\
		& \quad + (B \nabla a \, \phi, \nabla \phi \abs{\nabla \phi}^{p-2})_H - \chi^2 (\nabla \phi, \nabla \phi \abs{\nabla \phi}^{p-2})_H = (\nabla S, \nabla \phi \abs{\nabla \phi}^{p-2})_H.
	\end{align*}
	Next, by using hypotheses \ref{ass:fc0} and \ref{ass:flim}, we arrive at the inequality:
	\[ \frac{\tau}{p} \ddt \norm{\nabla \phi}^p_{p} + (c_0 - \chi^2) \norm{\nabla \phi}^p_{p} \le (\nabla S, \nabla \phi \abs{\nabla \phi}^{p-2})_H - (B \nabla a \, \phi, \nabla \phi \abs{\nabla \phi}^{p-2})_H.  \]
	Now, we integrate both sides over $(0,t)$, for any $t\in (0,T)$, and we use H\"older's inequality and the generalised Young's inequality with conjugate exponents $p$ and $p/(p-1)$, in order to get:
	\begin{align*}
		& \frac{\tau}{p} \int_\Omega \abs{\nabla \phi (t)}^p \, \de x 
		+ (c_0 -\chi^2) \int_{Q_t} \abs{\nabla \phi}^p \, \de x \, \de s \\
		& \quad = \int_\Omega \abs{\nabla \phi_0}^p \, \de x
		- \int_{Q_t} B \nabla a \, \phi \cdot \nabla \phi \abs{\nabla \phi}^{p-2} \, \de x \, \de s 
		+ \int_{Q_t} \nabla S \cdot \nabla \phi \abs{\nabla \phi}^{p-2} \, \de x \, \de s \\
		& \quad \le \int_\Omega \abs{\nabla \phi_0}^p \, \de x
		+ Bb \int_0^t \norm{\phi}_{p} \norm{\nabla \phi}_{p}^{\frac{p-1}{p}} \, \de s
		+ \int_0^t \norm{\nabla S}_{p} \norm{\nabla \phi}_{p}^{\frac{p-1}{p}} \, \de s \\
		& \quad \le \int_\Omega \abs{\nabla \phi_0}^p \, \de x
		+ Bb \, \sup_{(0,t)} \norm{\nabla \phi}_{p}^{\frac{p-1}{p}} \cdot \int_0^t \norm{\phi}_{p} \, \de s 
		+ \sup_{(0,t)} \norm{\nabla \phi}_{p}^{\frac{p-1}{p}} \cdot \int_0^t \norm{\nabla S}_{p} \, \de s \\
		& \quad \le \int_\Omega \abs{\nabla \phi_0}^p \, \de x
		+ \frac{\tau}{2p} \left( \sup_{(0,t)} \norm{\nabla \phi}_{p}^{\frac{p-1}{p}} \right)^{\frac{p}{p-1}}	
		+ C \left( \int_0^t \norm{\phi}_{p} \, \de s \right)^p 
		+ C \left( \int_0^t \norm{\nabla S}_{p} \, \de s \right)^p \\
		& \quad \le \int_\Omega \abs{\nabla \phi_0}^p \, \de x
		+ \frac{\tau}{2p} \sup_{(0,t)} \norm{\nabla \phi}_{p}^{p} 
		+ C \norm{\phi}^p_{L^1(0,T;\Lx p)} 
		+ C \norm{\nabla S}^p_{L^1(0,T;\Lx p)}, 
	\end{align*}
	where, in the last line, we used the fact that, for $p>1$, the real function $x \mapsto x^{\frac{p}{p-1}}$ is strictly increasing and the supremum is preserved under increasing functions. Then, passing to the supremum over $(0,t)$ also on the left-hand side, we get:
	\begin{align*}
		& \frac{\tau}{2p} \sup_{(0,t)} \int_\Omega \abs{\nabla \phi}^p \, \de x
		+ (c_0 - \chi^2) \int_{Q_t} \abs{\nabla \phi}^p \, \de x \, \de t \\
		& \quad \le \norm{\nabla \phi_0}_{p}^p 
		+ C \norm{\phi}^p_{L^1(0,T;L^p(\Omega))} 
		+ C \norm{\nabla S}^p_{L^1(0,T;L^p(\Omega))},
	\end{align*}
	where the terms on the right-hand side are uniformly bounded if $p\le 6$ by respectively \ref{ass:initial2}, \eqref{weaknorms:est} and \eqref{S:l2w}. Then, we can conclude that 
	\begin{equation}
		\label{phi:linfw16}
		\tau \norm{\phi}_{L^\infty(0,T; W^{1,6}(\Omega))} \le C. 
	\end{equation}
Finally, for any $i,j = 1,2,3$ we apply the differential operator $\partial_{x_i x_j}$ to \eqref{eq:mu2p} and we test the resulting equation by $\partial_{x_i x_j} \phi$, indeed:
	\begin{align*}
		& \frac{\tau}{2} \ddt \norm{\partial_{x_i x_j} \phi}^2_H + ( (AF''(\phi) + Ba) \partial_{x_i x_j} \phi, \partial_{x_i x_j} \phi )_H + ( B (\partial_{x_i} a \, \partial_{x_j} \phi + \partial_{x_j} a \, \partial_{x_i} \phi), \partial_{x_i x_j} \phi)_H \\ 
		& \qquad + (B \partial_{x_i x_j} a \, \phi, \partial_{x_i x_j} \phi)_H + (AF'''(\phi) \partial_{x_i} \phi \, \partial_{x_j} \phi, \partial_{x_i x_j} \phi)_H - \chi^2 \norm{\partial_{x_i x_j} \phi}^2_H \\
		& \quad = (\partial_{x_i x_j} S, \partial_{x_i x_j} \phi)_H.		
	\end{align*}
	Then, by using \ref{ass:fc0}, H\"older's and Young's inequality and \ref{ass:flim} together with the already proven separation property (see the following Remark \ref{F:derivatives} for more details), we infer that
	\begin{align*}
		& \frac{\tau}{2} \ddt \norm{\partial_{x_i x_j} \phi}^2_H + (c_0 - \chi^2) \norm{\partial_{x_i x_j} \phi}^2_H \\
		& \quad \le 2Bb \norm{\nabla \phi}_H \norm{\partial_{x_i x_j} \phi}_H + Bb_2 \norm{\phi}_H \norm{\partial_{x_i x_j} \phi}_H \\
		& \qquad + A \norm{F'''(\phi)}_{L^\infty(Q_T)} \norm{\nabla \phi}^2_{4} \norm{\partial_{x_i x_j} \phi}_H + \norm{\partial_{x_i x_j} S}_H \norm{\partial_{x_i x_j} \phi}_H \\		 
		& \quad \le C \norm{\partial_{x_i x_j} \phi}^2_H + C \norm{\phi}^2_V + C \norm{\nabla \phi}^4_{4} + C \norm{S}^2_{H^2(\Omega)}, 
 	\end{align*}
	 therefore, by summing on $i,j =1,2,3$, integrating on $(0,t)$ for any $t\in (0,T)$ and using Gronwall's lemma, together with \ref{ass:initial2}, \eqref{phi:linfw16} and \eqref{S:l2w}, we get the uniform estimate:
	 \begin{equation}
	 	\label{phi:linfw}
	 	\tau \norm{\phi}_{L^\infty(0,T;H^2(\Omega))} \le C.
	 \end{equation}
 	In particular, since we already knew that $\sigma - \chi \phi \in L^2(0,T;W)$ by \eqref{sigmaphi:l2w}, we also infer that
 	\begin{equation}
 		\label{sigma:l2w}
 		\norm{\sigma}_{L^2(0,T;H^2(\Omega))} \le C.
 	\end{equation}
 	To conclude, we observe that since $F\in \mathcal{C}^3((-l,l))$ and the separation property holds, starting from \eqref{phi:linfw} we can also get that
 	\[ \norm{F'(\phi)}_{L^\infty(0,T;H^2(\Omega))} \le C. \]
 	Then, by comparison in \eqref{eq:mu2p}, we can also see that
 	\[ \tau \phi_t = \mu - A F'(\phi) - B a \phi + B J \ast \phi - \chi \sigma \in L^2(0,T; H^2(\Omega)), \]
 	meaning that we also have the uniform bound
 	\begin{equation}
 		\label{phi:h1w}
 		\tau \norm{\phi}_{H^1(0,T;H^2(\Omega))} \le C.
 	\end{equation}
 	Finally, starting from all these estimates, a strong solution can then be recovered, up to passing to the limit in the discretisation framework. Moreover, estimate \eqref{strongnorms:est} can be deduced by weak lower-semicontinuity. This concludes the proof of Theorem \ref{thm:strongsols}.
\end{proof}

\begin{remark}
	\label{F:derivatives}
	Since $F \in \mathcal{C}^4((-l,l))$ by \ref{ass:flim} and the separation property \eqref{separation} holds, one can freely use the chain rule to compute the derivatives of $F'(\phi)$, as we already did in the previous proof. In particular, since $F$ and its derivatives are locally bounded in $(-l,l)$, by \eqref{separation} we can deduce that
	\[ \norm{F^{(i)}(\phi)}_{\Lqt\infty} \le C \quad \text{for any } i=1,\dots,4, \]
	where the constant $C>0$ depends only on $F$, $s^*$ and $\bar{M}$.
\end{remark}

We now want to prove that the solutions of system \eqref{eq:phi2}--\eqref{ic2} depend continuously from the controls $u$ and $v$ and from initial data $\phi_0$ and $\sigma_0$. As a byproduct, this will also give uniqueness of the strong solutions of Theorem \ref{thm:strongsols}. Indeed, we have the following result:

\begin{theorem}
	\label{thm:contdep}
	Assume hypotheses \emph{\ref{ass:coeff}--\ref{ass:fc0}} and \emph{\ref{ass:j2}--\ref{ass:initial2}}. Let $\phi_{0_1}$, $\sigma_{0_1}$, $u_1$, $v_1$ and $\phi_{0_2}$, $\sigma_{0_2}$, $u_2$, $v_2$ be two sets of data satisfying \emph{\ref{ass:uv2}} and \emph{\ref{ass:initial2}} and let $(\phi_1, \mu_1, \sigma_1)$ and $(\phi_2, \mu_2, \sigma_2)$ two corresponding strong solutions as in Theorem \emph{\ref{thm:strongsols}}. Then, there exists a constant $K>0$, depending only on the data of the system and on the norms of $\{ (\phi_{0_i}, \sigma_{0_i}, u_i, v_i) \}_{i=1,2}$, but not on their difference, such that
	\begin{equation}
		\label{contdep:estimate}
		\begin{split}
		& \tau \norm{\phi_1 - \phi_2}^2_{H^1(0,T;H) \cap L^\infty(0,T;H^2(\Omega))} + \norm{\mu_1 - \mu_2}^2_{L^2(0,T;W)} \\ 
		& \quad + \norm{\sigma_1 - \sigma_2}^2_{H^1(0,T;H) \cap L^\infty(0,T;V) \cap L^2(0,T;H^2(\Omega))} \\
		& \le K \left( \norm{u_1 - u_2}^2_{L^2(0,T;H)} + \norm{v_1 - v_2}^2_{L^2(0,T;H)} + \norm{\phi_{0_1} - \phi_{0_2}}^2_{H^2(\Omega)} + \norm{\sigma_{0_1} - \sigma_{0_2}}^2_{V}  \right).
		\end{split}
	\end{equation}
\end{theorem}

\begin{remark}
	As a consequence of Theorem \ref{thm:contdep}, we also immediately obtain \emph{uniqueness of strong solutions} to \eqref{eq:phi2}--\eqref{ic2}. Indeed, if $(\phi_1, \mu_1, \sigma_1)$ and $(\phi_2, \mu_2, \sigma_2)$ are two solutions with respect to the same data $(\phi_0, \sigma_0, u, v)$, then the right-hand side of \eqref{contdep:estimate} is equal to $0$ and uniqueness follows. 
\end{remark}

\begin{proof}
	Let $\phi = \phi_1 - \phi_2$, $\mu = \mu_1 - \mu_2$, $\sigma = \sigma_1 - \sigma_2$, $u = u_1 - u_2$, $v = v_1 - v_2$, $\phi_0 = \phi_{0_1} - \phi_{0_2}$ and $\sigma_0 = \sigma_{0_1} - \sigma_{0_2}$, then, up to adding and subtracting some terms, they solve:     
	\begin{align}
		 \partial_t \phi & = \Delta \mu + P(\phi_1) (\sigma - \chi \phi - \mu) + (P(\phi_1) - P(\phi_2)) (\sigma_2 + \chi (1-\phi_2) - \mu_2) & \nonumber \\
		& \quad- \hh(\phi_1) u - (\hh(\phi_1) - \hh(\phi_2)) u_2 & \text{in } Q_T,  \label{eq:phi3}\\
		 \mu & = \tau \partial_t \phi + A(F'(\phi_1)- F'(\phi_2)) + Ba \phi - BJ \ast \phi - \chi \sigma & \text{in } Q_T,  \label{eq:mu3} \\
		 \partial_t \sigma & = \Delta \sigma - \chi \Delta \phi - P(\phi_1) (\sigma - \chi \phi - \mu) & \nonumber \\
		& \quad - (P(\phi_1) - P(\phi_2)) (\sigma_2 + \chi (1-\phi_2) - \mu_2) + v & \text{in } Q_T, \label{eq:sigma3}
	\end{align}
	paired with boundary and initial conditions:
	\begin{align}
		& \partial_{\n} \mu = \partial_{\n} (\sigma - \chi \phi) = 0 & \text{on } \partial \Omega \times (0,T), \label{bc3} \\
		& \phi(0) = \phi_0, \quad \sigma(0) = \sigma_0 & \text{in } \Omega. \label{ic3}
	\end{align}
	Now, for the main estimate, we test equation \eqref{eq:phi3} by $\mu$ in $H$ and, by using Cauchy-Schwarz and Young's inequalities and the fact that $P$ and $\hh$ are Lipschitz functions on bounded intervals, together with the uniform estimates on the $L^\infty(Q_T)$-norms of $\phi_{1,2}$, we obtain: 
	\begin{equation}
		\label{contdep:test1mu}
		\begin{split}
			(\phi_t, \mu)_H + \norm{\nabla \mu}^2_H & \le (P(\phi_1) (\sigma - \chi \phi - \mu), \mu)_H  + \eps \norm{\mu}^2_H + C_\eps \norm{u}^2_H \\
			& \quad + C_\eps \left(1 + \norm{\sigma_2 + \chi (1-\phi_2) - \mu_2}^2_{\infty} + \norm{u_2}^2_{\infty} \right) \norm{\phi}^2_H, 
		\end{split}
		\end{equation}
	where $\eps \gs 0$ is to be chosen later and $\norm{\sigma_2 + \chi (1-\phi_2) - \mu_2}^2_{\infty} \in L^1(0,T)$, since every term is bounded in $L^2(0,T;H^2(\Omega))$ and $\Hx 2 \hookrightarrow L^\infty(\Omega)$. Next, we test equation \eqref{eq:mu3} by $- \phi_t$ in $H$ and we use Leibniz's rule in time and rewrite the resulting terms accordingly. Then, adding and subtracting $(Ba \phi, \phi_t)_H = \mezzo \ddt (Ba \phi, \phi)_H$ and by using the separation property together with the fact that $F''$ is Lipschitz on $[-s^*,s^*]$, we get:
	\begin{align*}
			& - (\mu, \phi_t)_H + \tau \norm{\phi_t}^2_H = - \ddt ( A(F'(\phi_1) - F'(\phi_2)), \phi )_H - \ddt (Ba \, \phi, \phi)_H + (Ba \phi, \phi_t)_H \\ 
			& \qquad + ( A(F''(\phi_1) - F''(\phi_2)) \phi_{1,t} + A F''(\phi_1) \phi_t, \phi )_H + (B J \ast \phi, \phi_t)_H + \chi (\sigma, \phi_t)_H \\
			& \quad \le - \ddt \left( (A(F'(\phi_1) - F'(\phi_2)) + Ba) \phi, \phi \right)_H + \chi (\sigma, \phi_t)_H + \frac{\tau}{2} \norm{\phi_t}^2_H \\
			& \qquad + C \norm{\sigma}^2_H + C (1 + \norm{\phi_{1,t}}^2_\infty ) \norm{\phi}^2_H, 
		\end{align*}
	where $\norm{\phi_{1,t}}^2_{\infty} \in L^1(0,T)$ by \eqref{phi:h1w}. Then, we test equation \eqref{eq:mu3} by $\delta \mu$ in $H$, with $\delta \gs 0$ to be chosen later, and, again since $F'$ is Lipschitz on $[-s^*,s^*]$ and the separation property holds, we infer that
	\begin{align*}
			\delta \norm{\mu}^2_H & = \delta \tau (\phi_t, \mu)_H + \delta ( A(F'(\phi_1) - F'(\phi_2)), \mu)_H + \delta (Ba \, \phi, \mu)_H \\ 
			& \quad - \delta (B J \ast \phi, \mu)_H - \delta \chi (\sigma, \mu)_H \\
			& \le \gamma \delta \tau \norm{\mu}^2_H + C_\gamma \delta \tau \norm{\phi_t}^2_H + \eps \norm{\mu}^2_H + C_{\delta,\eps} \norm{\phi}^2_H + C_{\delta,\eps} \norm{\sigma}^2_H,
	\end{align*}
	where we used Young's inequality with $\gamma \gs 0$ to be set later and $\eps \gs 0$ is the same one used before. Finally, we test equation \eqref{eq:sigma3} by $(\sigma - \chi \phi)$ and, with similar techniques, we find that
	\begin{align*}
			\frac{1}{2} \ddt \norm{\sigma}^2_H + \norm{\nabla (\sigma - \chi \phi)}^2_H & \le \chi (\sigma_t, \phi)_H - (P(\phi_1)(\sigma - \chi \phi - \mu), \sigma - \chi \phi)_H  \\
			& \quad + C (1 + \norm{\sigma_2 + \chi (1-\phi_2) - \mu_2}^2_{\infty} ) \norm{\phi}^2_H + C \norm{\sigma}^2_H + C \norm{v}^2_H, 
		\end{align*}
	where again $\norm{\sigma_2 + \chi (1-\phi_2) - \mu_2}^2_{\infty} \in L^1(0,T)$. Therefore, by summing up the previous four inequalities, after cancellations we get:
	\begin{equation}
			\label{contdep:est0}
			\begin{split}
				& \left( \frac{\tau}{2} - C_\gamma \delta \tau \right) \norm{\phi_t}^2_H + \left( \delta - 2 \eps - \gamma \delta \tau \right) \norm{\mu}^2_H + \norm{\nabla \mu}^2_H + \norm{\nabla (\sigma - \chi \phi)}^2_H + \ddt G(t) \\
				& \quad \le C_{\delta,\eps} \left( 1 + \norm{\sigma_2 + \chi (1-\phi_2) - \mu_2}^2_{\infty} + \norm{\phi_{1,t}}^2_{\infty} + \norm{u_2}^2_{\infty} \right) \norm{\phi}^2_H \\ 
				& \quad \quad + C_{\delta, \eps} \norm{\sigma}^2_H + C_\eps \norm{u}^2_H + C \norm{v}^2_H \underbrace{ - ( P(\phi_1) (\sigma - \chi \phi - \mu), \sigma - \chi \phi - \mu )_H }_{\le 0}, 
			\end{split}
	\end{equation}
	where
	\[ G(t) := \frac{1}{2} \norm{\sigma}^2_H + \left( (A(F'(\phi_1) - F'(\phi_2)) + Ba) \phi, \phi \right)_H - \chi (\sigma, \phi)_H.  \]
	Now, we can independently choose first $\eps \gs 0$, $\gamma = \gamma(\tau) \gs 0$ and then $\delta = \delta(\tau) \gs 0$, for any value of $\tau \gs 0$, such that 
	\begin{equation}
	\label{eq:epsdelta}
		\alpha_1 = \delta - 2 \eps - \gamma \delta \tau \gs 0 \quad \text{and} \quad \alpha_2 = \frac{\tau}{2} - C_\gamma \delta \tau \gs 0.
	\end{equation} 
	Moreover, by using Lagrange's theorem and the fact that $\phi_1, \phi_2 \in \mathcal{C}^0(\bar{Q_T})$ by standard embeddings, we can say that $F'(\phi_1(x,t)) - F'(\phi_2(x,t)) = F''(\bar{s}) \phi(x,t)$ for some $\bar{s}= \bar{s}(x,t) \in \R$, which is at least a measurable function, for a.e.~$(x,t) \in Q_T$. Then, by hypotheses \ref{ass:fc0} and \ref{ass:flim}, together with Cauchy-Schwarz and Young's inequalities, we can estimate $G(t)$ from below as
	\[ G(t) \ge \frac{1}{4} \norm{\sigma}^2_H + \left( c_0 - \chi^2 \right) \norm{\phi}^2_H.  \]
	Consequently, by integrating \eqref{contdep:est0} on $(0,t)$ for any $t \in (0,T)$, we infer that
	\begin{align*}
			& \frac{1}{4} \norm{\sigma}^2_H + \left( c_0 - \chi^2 \right) \norm{\phi}^2_H + \alpha_2(\tau) \int_0^t \norm{\phi_t}^2_H \, \de s + \alpha_1(\tau) \int_0^t \norm{\mu}^2_H \, \de s + \int_0^t \norm{\nabla \mu}^2_H \, \de s \\ 
			& \quad + \int_0^t \norm{\nabla (\sigma - \chi \phi)}^2_H \, \de s \le \norm{\phi_0}^2_H + \norm{\sigma_0}^2_H + C_\tau \int_0^t\norm{\sigma}^2_H \, \de s + C \int_0^T \left( \norm{u}^2_H + \norm{v}^2_H \right) \, \de s \\
			& \qquad + C_\tau \int_0^t \left( 1 + \norm{\sigma_2 + \chi (1-\phi_2) - \mu_2}^2_{\infty} + \norm{\phi_{1,t}}^2_{\infty} + \norm{u_2}^2_{\infty} \right) \norm{\phi}^2_H \, \de s, 
	\end{align*}
	and, by using Gronwall's lemma, we obtain the first continuous dependence estimate:
	\begin{equation}
		\label{contdep:est1}
		\begin{split}
				& \norm{\phi}^2_{L^\infty(0,T;H)} + \tau \norm{\phi}^2_{H^1(0,T;H)} + \norm{\sigma}^2_{L^\infty(0,T;H)} + \norm{\mu}^2_{L^2(0,T;V)} + \norm{\nabla (\sigma - \chi \phi)}^2_{L^2(0,T;H)} \\
				& \qquad \le C_\tau \left( \norm{u}^2_{L^2(0,T;H)} + \norm{v}^2_{L^2(0,T;H)} + \norm{\phi_0}^2_H + \norm{\sigma_0}^2_H \right).   
			\end{split}
	\end{equation}
	Observe that here the constant $C_\tau$ on the right-hand side depends on $\tau$, but this can be avoided under mild additional assumptions (see Remark \ref{rmk:epsdelta} below).
	Next, we test the gradient of equation \eqref{eq:mu3} by $\nabla \phi$ in $H$ and, after a careful rewriting of the terms, we get that
	\begin{align*}
		& \frac{\tau}{2} \ddt \norm{\nabla \phi}^2_H + ( (A F''(\phi_1) + Ba) \nabla \phi, \nabla \phi )_H = (\nabla \mu, \nabla \phi)_H \\ 
		& \quad + ( A(F''(\phi_1) - F''(\phi_2)) \nabla \phi_2, \nabla \phi)_H - (B \nabla a \, \phi, \nabla \phi)_H + (B \nabla J \ast \phi, \nabla \phi)_H - \chi (\nabla \sigma, \nabla \phi)_H.
	\end{align*}
	Hence, by adding and subtracting $\chi^2 (\nabla \phi, \nabla \phi)_H$ and by using hypotheses \ref{ass:j}, \ref{ass:fc0}, \ref{ass:flim}, the Lipschitz properties of $F''$ and H\"older and Young's inequalities, we infer that
	\begin{align*}
		& \frac{\tau}{2} \ddt \norm{\nabla \phi}^2_H + c_0 \norm{\nabla \phi}^2_H \\
		& \quad \le C \norm{\phi}_{4} \norm{\nabla \phi_2}_{4} \norm{\nabla \phi}_H + C \norm{\nabla \mu}^2_H + C \norm{\phi}^2_H + C \norm{\nabla \phi}^2_H + C \norm{\nabla (\sigma - \chi \phi)}^2_H \\
		& \quad \le C \big( 1 + \underbrace{\norm{\nabla \phi_2}^2_V}_{\in \, L^\infty(0,T)} \big) \norm{\phi}^2_V + C \norm{\nabla \mu}^2_H + C \norm{\nabla (\sigma - \chi \phi)}^2_H. 
	\end{align*}
	Then, by integrating on $(0,t)$, for any $t\in (0,T)$, and applying Gronwall's lemma, together with \eqref{contdep:est1}, we obtain the estimate
	\begin{align}
		\label{contdep:philinfv}
		\tau \norm{\phi}^2_{L^\infty(0,T;V)} \le C_\tau \left( \norm{u}^2_{L^2(0,T;H)} + \norm{v}^2_{L^2(0,T;H)} + \norm{\phi_0}^2_V + \norm{\sigma_0}^2_H \right),
	\end{align}
	and by comparison with the previous estimate on $\sigma - \chi \phi \in L^2(0,T;V)$, we also get that
	\begin{align}
		\label{contdep:sigmal2v}
		\norm{\sigma}^2_{L^2(0,T;V)} \le C_\tau \left( \norm{u}^2_{L^2(0,T;H)} + \norm{v}^2_{L^2(0,T;H)} + \norm{\phi_0}^2_V + \norm{\sigma_0}^2_H \right).
	\end{align}
	Now, we test equation \eqref{eq:sigma3} by $\partial_t (\sigma - \chi \phi)$ and, by using hypothesis \ref{ass:ph2} and Cauchy-Schwarz and Young's inequalities, we infer that
	\begin{align*}
	    \norm{\sigma_t}^2_H + \frac{1}{2} \ddt \norm{\nabla (\sigma - \chi \phi)}^2_H & \le \frac{1}{2} \norm{\sigma_t}^2_H + C (1 + \norm{\sigma_2 + \chi (1-\phi_2) - \mu_2}^2_{\infty} ) \norm{\phi}^2_H \\ 
		& \quad + C \norm{\phi_t}^2_H + C \norm{\sigma}^2_H + C \norm{\mu}^2_H + C \norm{v}^2_H.
	\end{align*}
	Therefore, by integrating in $(0,t)$ for any $t\in (0,T)$ and then applying Gronwall's lemma and \eqref{contdep:est1}, we deduce that
	\begin{equation}
		\label{contdep:sigmah1h}
		\norm{\sigma}^2_{H^1(0,T;H)} + \norm{\sigma}^2_{L^\infty(0,T;V)} \le C_\tau \left( \norm{u}^2_{L^2(0,T;H)} + \norm{v}^2_{L^2(0,T;H)} + \norm{\phi_0}^2_V + \norm{\sigma_0}^2_V \right).
	\end{equation}
	Next, we test equation \eqref{eq:sigma3} by $- \Delta (\sigma - \chi \phi)$ and, by similar methods, we infer that 
	\begin{align*}
		\norm{\Delta (\sigma - \chi \phi)}^2_H & \le \frac{1}{2} \norm{\Delta (\sigma - \chi \phi)}^2_H + C (1 + \norm{\sigma_2 + \chi (1-\phi_2) - \mu_2}^2_{\infty} ) \norm{\phi}^2_H \\ 
		& \qquad + C \norm{\sigma_t}^2_H + C \norm{\sigma}^2_H + C \norm{\mu}^2_H + C \norm{v}^2_H,
	\end{align*}
	then, by integrating on $(0,T)$ and using \eqref{contdep:est1} and \eqref{contdep:sigmah1h}, we get the estimate:
	\begin{equation}
		\label{contdep:sigmaphil2w}
		\norm{\sigma - \chi \phi}^2_{L^2(0,T;W)} \le C_\tau \left( \norm{u}^2_{L^2(0,T;H)} + \norm{v}^2_{L^2(0,T;H)} + \norm{\phi_0}^2_V + \norm{\sigma_0}^2_V \right).
	\end{equation}
	Moreover, in a similar way, we can also test equation \eqref{eq:phi3} by $- \Delta \mu$ and then integrate on $(0,T)$ to get:
	\begin{equation}
		\label{contdep:mul2w}
		\norm{\mu}^2_{L^2(0,T;W)} \le C_\tau \left( \norm{u}^2_{L^2(0,T;H)} + \norm{v}^2_{L^2(0,T;H)} + \norm{\phi_0}^2_V + \norm{\sigma_0}^2_V \right).
	\end{equation}
	Finally, for any $i,j = 1,2,3$, we apply the differential operator $\partial_{x_i x_j}$ to \eqref{eq:mu3}, which makes sense in $H$, and we test the resulting equation by $\partial_{x_i x_j} \phi$. Then, after careful rewriting of the terms arising from the derivatives of $F$, we get:
	\begin{align*}
		& (\partial_{x_i x_j} \mu, \partial_{x_i x_j} \phi)_H = \frac{\tau}{2} \ddt \norm{\partial_{x_i x_j} \phi}^2_H + ( (AF''(\phi_1) + Ba) \, \partial_{x_i x_j} \phi, \partial_{x_i x_j} \phi )_H \\
		& \quad + ( (F''(\phi_1) - F''(\phi_2)) \partial_{x_i x_j} \phi_2, \partial_{x_i x_j} \phi )_H + ( F'''(\phi_1) (\partial_{x_i} \phi_1 + \partial_{x_i} \phi_2) \, \partial_{x_j} \phi, \partial_{x_i x_j} \phi)_H \\
		& \quad + ( (F'''(\phi_1) - F'''(\phi_2)) \partial_{x_i} \phi_2 \, \partial_{x_j} \phi_2, \partial_{x_i x_j} \phi )_H + ( B (\partial_{x_i} a \, \partial_{x_j} \phi + \partial_{x_j} a \, \partial_{x_i} \phi), \partial_{x_i x_j} \phi )_H \\ 
		& \quad + (B \partial_{x_i x_j} a \, \phi, \partial_{x_i x_j} \phi)_H - B ( \partial_{x_i} (\partial_{x_j} J \ast \phi), \partial_{x_i x_j} \phi )_H \\ 
		& \quad - \chi ( \partial_{x_i x_j} (\sigma - \chi \phi), \partial_{x_i x_j} \phi )_H - \chi^2 \norm{\partial_{x_i x_j} \phi}^2_H.
	\end{align*}
	Hence, by using hypotheses \ref{ass:fc0}, \ref{ass:j2}, \ref{ass:flim}, Remark \ref{F:derivatives} and the generalised H\"older's and Young's inequalities, together with \eqref{gn:ineq} and \eqref{agmon}, we infer that
	\begin{align*}
		& \frac{\tau}{2} \ddt \norm{\partial_{x_i x_j} \phi}^2_H  + (c_0 - \chi^2) \norm{\partial_{x_i x_j} \phi}^2_H \\
		& \quad \le C \norm{\phi}_{\infty} \norm{\partial_{x_i x_j} \phi_2}_H \norm{\partial_{x_i x_j} \phi}_H + C \norm{\nabla \phi_1 + \nabla \phi_2}_{6} \norm{\nabla \phi}_{3} \norm{\partial_{x_i x_j} \phi}_H \\ 
		& \qquad + C \norm{\phi}_{6} \norm{\nabla \phi_2}^2_{6} \norm{\partial_{x_i x_j} \phi}_H + C \norm{\partial_{x_i x_j} \phi}^2_H + C \norm{\phi}^2_V + C \norm{\sigma - \chi \phi}^2_W + C \norm{\mu}^2_W \\
		& \quad \le C \norm{\partial_{x_i x_j} \phi_2}_H \norm{\phi}^{1/2}_V \norm{\phi}^{3/2}_{H^2(\Omega)} + C \norm{\nabla \phi_1 + \nabla \phi_2}_{6} \norm{\nabla \phi}^{1/2}_H \norm{\phi}^{3/2}_{H^2(\Omega)} \\
		& \qquad + C \norm{\phi}_{6} \norm{\nabla \phi_2}^2_{6} \norm{\partial_{x_i x_j} \phi}_H + C \norm{\partial_{x_i x_j} \phi}^2_H + C \norm{\phi}^2_V + C \norm{\sigma - \chi \phi}^2_W + C \norm{\mu}^2_W \\
		& \quad \le C \norm{\phi}^2_{H^2(\Omega)} + C \left( 1 + \norm{\phi_2}^4_{H^2(\Omega)} + \norm{\nabla \phi_1 + \nabla \phi_2}^4_{6} + \norm{\nabla \phi_2}^4_{6} \right) \norm{\phi}^2_V \\
		& \qquad + C \norm{\sigma - \chi \phi}^2_W + C \norm{\mu}^2_W,
	\end{align*}
	where $\norm{\phi_2}^4_{\Hx 2} + \norm{\nabla \phi_1 + \nabla \phi_2}^4_{6} + \norm{\nabla \phi_2}^4_{6} \in L^\infty (0,T)$, since, by \eqref{phi:linfw}, $\phi_2 \in L^\infty(0,T;H^2(\Omega))$, which is embedded into $L^\infty(0,T; W^{1,6}(\Omega))$. Therefore, by summing on $i,j =1,2,3$, integrating on $(0,t)$, for any $t\in (0,T)$, and using Gronwall's lemma, together with the previous estimates \eqref{contdep:philinfv}, \eqref{contdep:sigmaphil2w}, \eqref{contdep:mul2w}, we infer that
	\begin{equation}
		\label{contdep:philinfw}
		\tau \norm{\phi}^2_{L^\infty(0,T;H^2(\Omega))} \le C_\tau \left( \norm{u}^2_{L^2(0,T;H)} + \norm{v}^2_{L^2(0,T;H)} + \norm{\phi_0}^2_{H^2(\Omega)} + \norm{\sigma_0}^2_V \right).
	\end{equation} 
	Moreover, by comparison with \eqref{contdep:sigmaphil2w}, we also deduce that
		\begin{equation}
		\label{contdep:sigmal2w}
		\norm{\sigma}^2_{L^2(0,T;H^2(\Omega))} \le C_\tau \left( \norm{u}^2_{L^2(0,T;H)} + \norm{v}^2_{L^2(0,T;H)} + \norm{\phi_0}^2_{H^2(\Omega)}+ \norm{\sigma_0}^2_V \right).
	\end{equation}
	This concludes the proof of Theorem \ref{thm:contdep}, since all the constants that appear in front of the estimates depend only on parameters and possibly on the norms of $\{ (\phi_{0_i}, \sigma_{0_i}) \}_{i=1,2}$, but not on their difference.
\end{proof}

\begin{remark}
	\label{rmk:epsdelta}
	We wish to point out that, under the additional assumption that $0 \ls \tau \le M$, for some $M \gs 0$, one can show that all constants $C_\tau$ appearing in the proof are actually independent of $\tau$ for any $\tau \in (0,M]$. This can be useful in applications, where $\tau$ is generally kept very small and is possibly tending to $0$. 
	
	Indeed, going back to \eqref{eq:epsdelta}, one can first choose $\eps \gs 0$ and $0 \ls \gamma \ls 1/M$ in such a way that 
	\[ \alpha_1 = \delta - 2\eps - \gamma \delta \tau \gs 0 \quad \forall \tau \in (0,M]. \]
	Then, $C_\gamma$ is now fixed independently of $\tau \in (0,M]$, therefore one can choose $\delta \gs 0$ small enough such that 
	\[ \alpha_2 = \frac{\tau}{2} - C_\gamma \delta \tau = \tau \bar{\alpha_2} \quad \text{with $\bar{\alpha_2} \gs 0$ independently of $\tau \in (0,M]$}. \]
	At this point, it is clear that the constants $C_\tau$, appearing from \eqref{contdep:est1} onwards, depend only on the other parameters and on $M$, which is fixed. 
\end{remark}

\begin{remark}
	In the previous proof, in order to close the first continuous dependence estimate \eqref{contdep:est1}, we need a positive term on the left-hand side for $\norm{\mu}^2_H$. This is obtained by testing \eqref{eq:mu3} by $\delta \mu$, with a small constant $\delta$, and leads to either constants depending on $\tau$ or to an additional hypothesis on $\tau$ as in Remark \ref{rmk:epsdelta}. Another possibility could be to directly use hypothesis \ref{ass:ph2} on $P_0 > 0$ on the term $(P(\phi_1) \mu, \mu)$, which can be brought on the left when testing \eqref{eq:phi3} by $\mu$. 
	Namely, inequality \eqref{contdep:test1mu} would become 
	\begin{align*}
		(\phi_t, \mu)_H + \norm{\nabla \mu}^2_H + (P(\phi_1) \mu, \mu)_H & \le (P(\phi_1) (\sigma - \chi \phi), \mu)_H  + \eps \norm{\mu}^2_H + C_\eps \norm{u}^2_H \\
		& \quad + C_\eps \left(1 + \norm{\sigma_2 + \chi (1-\phi_2) - \mu_2}^2_{\infty} + \norm{u_2}^2_{\infty} \right) \norm{\phi}^2_H, 
	\end{align*}
	which then, by \ref{ass:ph2}, Cauchy-Schwarz and Young, would immediately imply
	\begin{align*}
		(\phi_t, \mu)_H + \norm{\nabla \mu}^2_H + (P_0 - 3\eps) \norm{\mu}^2_H & \le C_\eps \norm{\phi}^2_H + C_\eps \norm{\sigma}^2_H + C_\eps \norm{u}^2_H \\
		& \quad + C_\eps \left(1 + \norm{\sigma_2 + \chi (1-\phi_2) - \mu_2}^2_{\infty} + \norm{u_2}^2_{\infty} \right) \norm{\phi}^2_H, 
	\end{align*} 
	thus removing the need of testing \eqref{eq:mu3} by $\delta \mu$.
	However, we choose the first method, since the hypothesis $P_0 > 0$ could be restrictive in applications, so we would like to avoid heavy use of it, even if it is necessary in the proof of Theorem \ref{thm:strongsols}. Moreover, the procedure we used could possibly be adapted to even different reaction terms, provided that one is able to prove existence of strong solutions. The same strategy shall be pursued also in Section \ref{sec:optcont}.
\end{remark}

\section{Optimal control problem}
\label{sec:optcont}

From now on, we consider the initial data $\phi_0, \sigma_0$, satisfying \ref{ass:initial2}, fixed. We recall the optimal control problem that we want to study: 

\bigskip
\noindent(CP) \textit{Minimise the cost functional}
\begin{equation}
	\begin{split}
		\mathcal{J}(\phi, \sigma, u, v) & = \, \frac{\alpha_{\Omega}}{2} \int_{\Omega} |\phi(T) - \phi_{\Omega}|^2 \,\de x  + \frac{\alpha_Q}{2} \int_{0}^{T} \int_{\Omega} |\phi - \phi_Q|^2 \,\de x \,\de t \\
		& \quad + \frac{\beta_{\Omega}}{2} \int_{\Omega} |\sigma(T) - \sigma_{\Omega}|^2 \,\de x  + \frac{\beta_Q}{2} \int_{0}^{T} \int_{\Omega} |\sigma - \sigma_Q|^2 \,\de x \,\de t \\ 
		& \quad + \frac{\alpha_u}{2} \int_{0}^{T} \int_{\Omega} |u|^2 \,\de x \,\de t + \frac{\beta_v}{2} \int_{0}^{T} \int_{\Omega} |v|^2 \,\de x \,\de t,
	\end{split}
\end{equation}
\textit{subject to the control constraints}
\begin{align} 
	& u \in \mathcal{U}_{ad} := \{ u \in L^{\infty}(Q_T) \cap H^1(0,T;H) \mid u_{\text{min}} \le u \le u_{\text{max}} \text{ a.e.~in } Q_T, \norm{u}_{H^1(0,T;H)} \le M \} \nonumber \\
	& v \in \mathcal{V}_{ad} := \{ v \in L^{\infty}(Q_T) \mid v_{\text{min}} \le v \le v_{\text{max}} \text{ a.e.~in } Q_T \}
\end{align}
\textit{and to the state system \eqref{eq:phi2}-\eqref{ic2}}.
\bigskip

Regarding the parameters at play, we make the following hypotheses:
\begin{enumerate}[font = \bfseries, label = C\arabic*., ref=\bf{C\arabic*}]
	\item\label{C1} $\alpha_\Omega, \alpha_Q, \alpha_u, \beta_\Omega, \beta_Q, \beta_v \ge 0$, but not all equal to $0$.
	\item\label{C2} $\phi_\Omega \in L^2(\Omega)$, $\sigma_\Omega \in H^1(\Omega)$ and $\phi_Q, \sigma_Q \in L^2(Q_T)$.
	\item\label{C3} $u_{\text{min}}, u_{\text{max}}, v_{\text{min}}, v_{\text{max}} \in L^\infty(Q_T)$, with $u_{\text{min}} \ge 0$ a.e. in $\Omega$. Let also $M'>0$ be such that for any $u \in \mathcal{U}_{ad}$ and for any $v \in \mathcal{V}_{ad}$ 
	\[ \norm{u}_{L^\infty(Q_T)}, \norm{v}_{L^\infty(Q_T)} \le M'. \]
	\item\label{C4} $P, \hh \in \mathcal{C}^2(\R) \cap W^{1,\infty}(\R)$.
\end{enumerate}

\begin{remark}
	The hypothesis $u_{\text{min}} \ge 0$ comes from modelling assumptions, see also \cite{CSS2021} and references therein. Moreover, we need the stronger hypothesis $\sigma_\Omega \in H^1(\Omega)$ to prove well-posedness of the adjoint system in Section $5.3$.
	
	Finally, we would like to comment on the structure of $\Uad$. While the best choice, from an applicative viewpoint, would be a subset of $\Lqt \infty$ with box constraints, e.g.~$\{ u \in L^{\infty}(Q_T) \mid u_{\text{min}} \le u \le u_{\text{max}} \text{ a.e.~in } Q_T \}$, the presence of $u$ as a source term in the Cahn-Hilliard equation demands higher regularity, if one wants to achieve strong well-posedness. Indeed, we also need a bound on the $\HT 1 H$--norm, as can be seen by \ref{ass:uv2} and Theorems \ref{thm:strongsols} and \ref{thm:contdep}.
\end{remark}

By Theorems \ref{thm:strongsols} and \ref{thm:contdep}, we know that for any $(u,v) \in \Uad \times \Vad$ there exists a unique strong solution $(\phi, \mu, \sigma) \in \mathbb{X}$ to \eqref{eq:phi2}--\eqref{ic2}, where 
\begin{align*}
	\mathbb{X} & := (W^{1,\infty}(0,T;H) \cap H^1(0,T;H^2(\Omega)) \cap L^\infty(0,T;H^2(\Omega))) \\ 
	& \quad \times L^\infty(0,T;W) \times (H^1(0,T;H) \cap L^\infty(0,T,V) \cap L^2(0,T;H^2(\Omega))),
\end{align*}
therefore the optimal control problem (CP) is well-defined. Our goal is to prove existence of an optimal control and then find the first-order necessary optimality conditions. 

\begin{theorem}
	\label{thm:excont}
	Assume hypotheses \emph{\ref{ass:coeff}--\ref{ass:fc0}}, \emph{\ref{ass:j2}--\ref{ass:initial2}} and \emph{\ref{C1}}--\emph{\ref{C4}}. Then the optimal control problem (CP) admits at least one solution $(\ub, \vb) \in \Uad \times \Vad$, such that if $(\phib, \mub, \sigmab)$ is the solution to \eqref{eq:phi2}--\eqref{ic2} associated to $(\ub, \vb)$, one has that
	\begin{equation}
		\mathcal{J}(\phib, \sigmab, \ub, \vb) = \min_{(u,v) \,\in \,\Uad \times \Vad} \, \mathcal{J}(\phi, \sigma, u, v).
	\end{equation}
\end{theorem}

\begin{proof}
	Let $\{ (u_n, v_n) \}_{n\in\N} \subset \Uad \times \Vad$ be a minimising sequence such that
	\begin{equation}
		0 \le \inf_{(u,v) \,\in \,\Uad \times \Vad} \, \mathcal{J}(\phi, \sigma, u, v) = \lim_{n \to + \infty} \mathcal{J}(\phi_n, \sigma_n, u_n, v_n),
	\end{equation}
	where $(\phi_n, \mu_n, \sigma_n)$ are the solutions to \eqref{eq:phi2}--\eqref{ic2} associated to $(u_n, v_n)$, with the regularities given by Theorem \ref{thm:strongsols}.
	
	Since $\{ (u_n, v_n) \}_{n\in\N} \subset \Uad \times \Vad$, we have that $\{ u_n \}$ is uniformly bounded in $L^\infty(Q_T) \cap H^1(0,T;H)$ and $\{v_n\}$ is uniformly bounded in $L^\infty(Q_T)$, therefore, by the Banach-Alaouglu theorem, we can deduce that there exists $(\ub, \vb) \in (L^\infty(Q_T) \cap H^1(0,T;H)) \times L^\infty(Q_T)$ such that, up to a subsequence,
	\begin{align*}
		& u_n \weakstar \ub \quad \text{weakly star in } L^\infty(Q_T) \cap H^1(0,T;H), \\
		& v_n \weakstar \vb \quad \text{weakly star in } L^\infty(Q_T).
	\end{align*}
	Moreover, since $\Uad \times \Vad$ is convex and closed in the larger space $H^1(0,T;H) \times L^2(Q_T)$, it is also weakly sequentially closed and thus $(\ub, \vb) \in \Uad \times \Vad$. 
	
	Now, consider the solutions $(\phi_n, \mu_n, \sigma_n) \in \mathbb{X}$ corresponding to $(u_n, v_n)$ for every $n\in\N$, then, by the uniform estimate \eqref{strongnorms:est}, we have that these solutions are uniformly bounded with respect to $n$ in the spaces of strong solutions, given by Theorem \ref{thm:strongsols}. Therefore, again by Banach-Alaoglu, we can say that, up to a subsequence,
	\begin{align*}
		& \phi_n \weakstar \phib \quad \text{weakly star in } W^{1,\infty}(0,T;H) \cap H^1(0,T;H^2(\Omega)) \cap L^\infty(0,T;H^2(\Omega)), \\
		& \mu_n \weakstar \mub \quad \text{weakly star in } L^\infty(0,T;W), \\ 
		& \sigma_n \weakstar \sigmab \quad \text{weakly star in } H^1(0,T;H) \cap L^\infty(0,T,V) \cap L^2(0,T;H^2(\Omega)). 
	\end{align*}
	In particular, by the compact embeddings of Aubin-Lions-Simon (see \cite[Section 8, Corollary 4]{S1986}), it follows that $\phi_n \to \phib$ strongly in $\mathcal{C}^0([0,T], H^s(\Omega))$ for any $0<s<2$, which implies that $\phi_n \to \phib$ strongly in $\mathcal{C}^0(\overline{Q_T})$, since $H^s(\Omega) \hookrightarrow \mathcal{C}^0(\overline{Q_T})$ if $s > 3/2$ for $N=3$. In particular, by uniform convergence, $\phib$ still satisfies the strict separation property. This fact, by the continuity of the functions $P$, $\hh$ and $F'$, implies that 
	\[ P(\phi_n) \to P(\phib), \, F'(\phi_n) \to F(\phib), \, \hh(\phi_n) \to \hh(\phib) \quad \text{strongly in } \mathcal{C}^0(\overline{Q_T}). \]
	Moreover, in the same way, we also have that $\sigma_n \to \sigmab$ strongly in $\mathcal{C}^0([0,T], H)$. Therefore, by using the weak and strong convergences written above, it is easy to show that, starting from the equations \eqref{eq:phi2}--\eqref{ic2} satisfied by $(\phi_n, \mu_n, \sigma_n)$ with respect to $(u_n,v_n)$, also the limit functions $(\phib, \mub, \sigmab) \in \mathbb{X}$ satisfy \eqref{eq:phi2}--\eqref{ic2} with respect to $(\ub, \vb)$. Then, we can infer that 
	\[ \inf_{(u,v) \,\in \,\Uad \times \Vad} \, \mathcal{J}(\phi, \sigma, u, v)  \le \mathcal{J}(\phib, \sigmab, \ub, \vb).  \]
	
	Now, we observe that the functional $\mathcal{J}$, if defined on the larger space $H^1(0,T;V,V^*)^2 \times L^2(Q_T)^2$, where $H^1(0,T;V,V^*) = H^1(0,T;V^*) \cap L^2(0,T;V) \hookrightarrow \mathcal{C}^0([0,T], H)$, is strongly continuous and convex, therefore it is weakly lower-semicontinuous with respect to this weaker topology. Consequently, we can deduce that
	\[ \mathcal{J}(\phib, \sigmab, \ub, \vb) \le \liminf_{n \to + \infty} \mathcal{J}(\phi_n, \sigma_n, u_n, v_n) =  \inf_{(u,v) \,\in \,\Uad \times \Vad} \, \mathcal{J}(\phi, \sigma, u, v)  \le \mathcal{J}(\phib, \sigmab, \ub, \vb), \]
	which means that $(\ub, \vb) \in \Uad \times \Vad$ is an optimal control. This concludes the proof of Theorem \ref{thm:excont}.
\end{proof}

\subsection{Linearised system}

As an \emph{ansatz} for the Fréchet-derivative of the control-to-state operator, which maps any $(u,v) \in  \Uad \times \Vad$ into the corresponding solution of the state system, we now start studying the linearised version of our system.
Indeed, we fix an optimal state $(\phib, \mub, \sigmab) \in \mathbb{X}$ corresponding to $(\ub, \vb) \in \Uad \times \Vad$ and linearise near $(\ub, \vb)$: 
\[ \phi = \phib + \xi, \, \mu = \mub + \eta, \, \sigma = \sigmab + \rho, \, u = \ub + h, \, v = \vb + k,  \]
with $h \in L^\infty(Q_T) \cap H^1(0,T;H)$ and $k \in L^\infty(Q_T)$. Then, by approximating the non-linearities at the first order of their Taylor expansion, we see that $(\xi, \eta, \rho)$ satisfy the equations:
\begin{align}
	& \partial_t \xi = \Delta \eta + P'(\phib) (\sigmab + \chi (1-\phib) - \mub) \xi + P(\phib) (\rho - \chi \xi - \eta) \nonumber \\
	& \qquad \quad - \hh'(\phib) \ub \, \xi - \hh(\phib) h && \text{in } Q_T,  \label{eq:xi}\\
	& \eta = \tau \partial_t \xi + AF''(\phib) \xi + Ba \xi - BJ \ast \xi - \chi \rho && \text{in } Q_T,  \label{eq:eta} \\
	& \partial_t \rho = \Delta \rho - \chi \Delta \xi - P'(\phib) (\sigmab + \chi (1-\phib) - \mub) \xi - P(\phib) (\rho - \chi \xi - \eta) + k && \text{in } Q_T, \label{eq:rho}
\end{align}
together with boundary and initial conditions:
\begin{alignat}{2}
	& \partial_{\n} \eta = \partial_{\n} (\rho - \chi \xi) = 0 \qquad && \text{on } \partial \Omega \times (0,T), \label{bcl} \\
	& \xi(0) = 0, \quad \rho(0) = 0 \qquad && \text{in } \Omega. \label{icl}
\end{alignat}
Observe that to derive the linearised system, we needed to start from $h \in L^\infty(Q_T) \cap H^1(0,T;H)$ and $k \in L^\infty(Q_T)$, however the system actually makes sense even for $h,k \in L^2(0,T;H)$. Indeed, we will prove well-posedness in this more general case.

We have the following result about existence and uniqueness of solutions to \eqref{eq:xi}--\eqref{icl}:

\begin{theorem}
	\label{thm:linearised}
	Assume hypotheses \emph{\ref{ass:coeff}--\ref{ass:fc0}}, \emph{\ref{ass:j2}--\ref{ass:initial2}} and let $(\phib, \mub, \sigmab) \in \mathbb{X}$ be the strong solution to \eqref{eq:phi2}--\eqref{ic2}, corresponding to $(\ub, \vb) \in \Uad \times \Vad$. Then, for any $h \in L^2(0,T;H)$ and for any $k \in L^2(0,T;H)$, the linearised system \eqref{eq:xi}--\eqref{icl} admits a unique solution such that
	\begin{align*}
		& \xi \in H^1(0,T;H) \cap L^\infty(0,T; V), \\
		& \eta \in L^2(0,T;W), \\
		& \rho \in H^1(0,T;H) \cap L^\infty(0,T;V) \cap L^2(0,T;W), \\
		& \rho - \chi \xi \in L^2(0,T;W),
	\end{align*}
	which fulfils \eqref{eq:xi}--\eqref{icl} almost everywhere in the respective sets. 
\end{theorem}

\begin{proof}
	We proceed with only formal estimates, which can be done rigorously through a Faedo-Galerkin discretisation scheme, with discrete spaces made of eigenvectors of the operator $\mathcal{N}$. Then, one can pass to the limit in a standard matter to recover the estimates also in the continuous case. For the first estimate, we test in $H$ equation \eqref{eq:xi} by $\eta$ and we get: 
	\begin{align*}
		(\xi_t, \eta)_H + \norm{\nabla \eta}^2_H  & = (P'(\phib) (\sigmab + \chi (1-\phib) - \mub) \xi, \eta)_H + (P(\phib)(\rho - \chi \xi - \eta), \eta)_H \\
		& \quad - (\hh'(\phib) \ub \xi, \eta)_H - (\hh(\phib) h, \eta)_H. 
	\end{align*}
	Then, by using hypothesis \ref{ass:ph2} and Cauchy-Schwarz and Young's inequality with $\eps > 0$ to be chosen later, we obtain:
	\begin{align*}
		(\xi_t, \eta)_H + \norm{\nabla \eta}^2_H  & \le \eps \norm{\eta}^2 + C_\eps \left( 1 + \norm{\sigmab + \chi (1-\phib) - \mub}^2_{\infty} + \norm{\ub}^2_{\infty} \right) \norm{\xi}^2_H \\
		& \quad + C_\eps \norm{\rho}^2_H + C_\eps \norm{h}^2_H + (P(\phib) (\rho - \chi \xi - \eta), \eta)_H, 
	\end{align*}
	where we recall the uniform bound on $\norm{\sigmab + \chi (1-\phib) - \mub}^2_{\infty} \in L^1(0,T)$, since $H^2(\Omega) \hookrightarrow L^\infty(\Omega)$, and on $\norm{\ub}^2_{\infty} \in \Lt \infty$ by \ref{ass:uv2}.
	Next, we test \eqref{eq:eta} in $H$ by $-\xi_t$, which is possible within the discretisation, and, by using Leibniz's rule for the time-derivative, we get: 
	\begin{align*}
		& -(\eta, \xi_t) + \tau \norm{\xi_t}^2_H + \frac{1}{2} \ddt \left( (AF''(\phib) +Ba) \xi, \xi \right)_H \\
		& \quad =  (AF'''(\phib) \phib_t \xi, \xi)_H + (B J \ast \xi, \xi_t)_H + \chi (\rho, \xi_t)_H.
	\end{align*}
	We further estimate the right-hand side by using H\"older's and Young's inequalities and by recalling that, by Theorem \ref{thm:strongsols}, we have the separation property for $\phib$, indeed:
	\begin{align*}
		& -(\eta, \xi_t) + \tau \norm{\xi_t}^2_H + \frac{1}{2} \ddt \left( (AF''(\phib) +Ba) \xi, \xi \right)_H \\
		& \quad \le A\norm{F'''(\phib)}_{\infty} \norm{\phib_t}_{\infty} \norm{\xi}^2_H +  B a^* \norm{\xi}_H \norm{\xi_t}_H + \chi (\rho, \xi_t)_H \\
		& \quad \le \frac{\tau}{2} \norm{\xi_t}^2_H + C (1 + \norm{\phib_t}_{\infty}) \norm{\xi}^2_H + \chi (\rho, \xi_t)_H,
	\end{align*}
	where $\phib_t \in \LT 2 {\Hx 2} \hookrightarrow \LT 2 {\Lx \infty}$ by Theorem \ref{thm:strongsols}.
	Moreover, we test \eqref{eq:rho} in $H$ by $\rho - \chi \xi$ and, by similar techniques, we infer that
	\begin{align*}
		& \mezzo \ddt \norm{\rho}^2_H + \norm{\nabla (\rho - \chi \xi)}^2_H - \chi (\rho_t, \xi)_H \\
		& \quad = (P'(\phib) (\sigmab + \chi (1-\phib) - \mub) \xi, \rho - \chi \xi)_H - (P(\phib)(\rho - \chi \xi - \eta), \rho - \chi \xi)_H + (k, \rho - \chi \xi)_H \\
		& \quad \le C \norm{\rho}^2_H + C (1 +  \norm{\sigmab + \chi (1-\phib) - \mub}^2_\infty) \norm{\xi}^2_H + C \norm{k}^2_H - (P(\phib)(\rho - \chi \xi - \eta), \rho - \chi \xi)_H. 
	\end{align*}
	Then, by summing up all three inequalities, after cancellations, we get:
	\begin{align}
	\label{lin:enest}
		& \mezzo \ddt G(t) + \frac{\tau}{2} \norm{\xi_t}^2_H + \norm{\nabla \eta}^2_H + \norm{\nabla (\rho - \chi \xi)}^2_H + \overbrace{ (P(\phib)(\rho - \chi \xi - \eta), \rho - \chi \xi - \eta)_H }^{\ge 0} \nonumber \\
		& \quad \le \eps \norm{\eta}^2_H + C_\eps \left( 1 +  \norm{\sigmab + \chi (1-\phib) - \mub}^2_\infty + \norm{\phib_t}_{\infty} + \norm{\ub}^2_{\infty} \right) \norm{\xi}^2_H \\ 
		& \qquad + C_\eps \norm{\rho}^2_H + C_\eps \norm{h}^2_H + C_\eps \norm{k}^2_H, \nonumber
	\end{align}
	where 
	\[ G(t) = ((AF''(\phib) +Ba) \xi, \xi )_H + \norm{\rho}^2_H - \chi (\rho, \xi)_H. \]
	Observe that, by hypothesis \ref{ass:fc0} and Young's inequality, we can infer that
	\[ G(t) \ge (c_0 - \chi^2) \norm{\xi}^2_H + \frac{3}{4} \norm{\rho}^2_H, \]
	moreover, by the initial conditions, we know that $G(0) =0$.
	Finally, we test again \eqref{eq:eta} in $H$ by $\delta \eta$ to obtain:
	\[ \delta \norm{\eta}^2_H = \delta \tau (\xi_t, \eta)_H + \delta (AF''(\phib) \xi, \eta)_H + \delta (Ba \xi, \eta)_H - \delta (B J \ast \xi, \eta)_H - \delta \chi (\rho, \eta)_H. \]
	By using Remark \ref{F:derivatives}, Cauchy-Schwarz and Young's inequalities with $\gamma > 0$ to be chosen later and the same $\eps > 0$ used before, we can estimate
	\[ \delta \norm{\eta}^2_H \le C_\gamma \delta \tau \norm{\xi_t}^2_H + \gamma \delta \tau \norm{\eta}^2_H + \eps \norm{\eta}^2_H + C_{\delta, \eps} \norm{\xi}^2_H + C_{\delta, \eps} \norm{\rho}^2_H.  \]
	At this point, we can sum this last inequality to \eqref{lin:enest}, integrate everything over $(0,t)$, for any $t \in (0,T)$, and infer that
	\begin{equation*}
		\begin{split}
		& \mezzo (c_0 - \chi^2) \norm{\xi(t)}^2_H + \frac{3}{8} \norm{\rho(t)}^2_H + \left( \frac{\tau}{2} - C_\gamma \delta \tau \right) \int_0^t \norm{\xi_t}^2_H \, \de s \\
		& \qquad + \left( 1 - 2\eps - \gamma \delta \tau \right) \int_0^t \norm{\eta} \, \de s + \int_0^t \norm{\nabla \eta}^2_H \, \de s + \int_0^t \norm{\nabla (\rho - \chi \xi)}^2_H \, \de s \\
		& \quad \le C_{\eps, \delta} \int_0^t \norm{\rho}^2_H \, \de s + C_\eps \int_0^t \norm{h}^2_H \, \de s + C \int_0^t \norm{k}^2_H \, \de s \\
		& \qquad + \int_0^t C_{\eps, \delta} \left( 1 +  \norm{\sigmab + \chi (1-\phib) - \mub}^2_\infty + \norm{\phib_t}_{\infty} + \norm{\ub}^2_{\infty} \right) \norm{\xi}^2_H \, \de s.
		\end{split}
	\end{equation*}
	Now, exactly as in the proof of Theorem \ref{thm:contdep} (see equation \eqref{eq:epsdelta}), we can respectively choose $\eps \gs 0$, $\gamma = \gamma(\tau) \gs 0$ and then $\delta = \delta(\tau) \gs 0$ such that the constants above are strictly positive. Therefore, we can apply Gronwall's inequality and recover the estimate:
	\begin{equation}
	\label{lin:gronwall}
		\begin{split}
		& \tau \norm{\xi}^2_{H^1(0,T;H)} + \norm{\xi}^2_{L^\infty(0,T;H)} + \norm{\rho}^2_{L^\infty(0,T;H)} + \norm{\eta}^2_{L^2(0,T;V)} \\ 
		& \quad + \norm{\nabla (\rho - \chi \xi)}^2_{L^2(0,T;H)} \le C_\tau \left( \norm{h}^2_{L^2(Q_T)} + \norm{k}^2_{L^2(Q_T)} \right), 
		\end{split}
	\end{equation}
	with $C_\tau > 0$, depending only on the parameters of the system and on $\tau>0$. By reasoning as in Remark \ref{rmk:epsdelta}, this dependence on $\tau$ can be avoided, if one additionally assumes that the range of values that $\tau$ can assume is bounded from above. For simplicity, from now on we will remove the explicit dependence on $\tau$ from the constants, keeping in mind that the only constant that could depend on $\tau$ is the one in \eqref{lin:gronwall}.

	Now, we can test \eqref{eq:eta} in $H$ by $- \Delta \xi$, which makes sense in the Galerkin discretisation, and by using integration by parts formulae, without any extra boundary terms, due to the discrete spaces, we have:
	\begin{align*}
		& (\nabla \eta, \nabla \xi)_H = \frac{\tau}{2} \ddt \norm{\nabla \xi}^2_H + (AF'''(\phib) \nabla \phib \, \xi, \nabla \xi)_H + A(F''(\phib) \nabla \xi, \nabla \xi)_H + (Ba \nabla \xi, \nabla \xi)_H \\
		& \qquad + (B \nabla a \, \xi, \nabla \xi)_H - (B \nabla J \ast \xi, \nabla \xi)_H - \chi (\nabla \rho, \nabla \xi)_H.
	\end{align*}
	In addition, if we add $\chi^2 (\nabla \xi, \nabla \xi)_H$ on both sides of the equality, then, by rearranging the terms and by using hypothesis \ref{ass:fc0} and Cauchy-Schwarz and Young's inequalities, we infer that
	\begin{equation}
		\label{lin:est1}
		\begin{split}
		& \frac{\tau}{2} \ddt \norm{\nabla \xi}^2_H + (c_0 - \chi^2) \norm{\nabla \xi}^2_H \\
		& \quad \le C \norm{\nabla \xi}^2_H + C \norm{\nabla \eta}^2_H + C \norm{\xi}^2_H + C \norm{\nabla (\rho - \chi \xi)}^2_H + (AF'''(\phib) \nabla \phib \, \xi, \nabla \xi)_H.
		\end{split}
	\end{equation}
	In order to estimate the last term, we use Remark \ref{F:derivatives}, H\"older's inequality, \eqref{gn:ineq} and the generalised Young's inequality as follows:
	\begin{align*}
		& (AF'''(\phib) \nabla \phib \, \xi, \nabla \xi)_H \le A \norm{F'''(\phib)}_{\infty} \norm{\nabla \phi}_{6} \norm{\xi}_{3} \norm{\nabla \xi}_H \\ 
		& \quad \le C_T \norm{\nabla \phi}_{6} \norm{\xi}^{1/2}_H \norm{\nabla \xi}^{3/2}_H \le C \norm{\nabla \xi}^2_H + C  \underbrace{\norm{\nabla \phi}^4_{6}}_{\in \, \Lx \infty} \norm{\xi}^2_H.
	\end{align*}
	Then, going back to \eqref{lin:est1} and integrating on $(0,t)$ for any $t \in (0,T)$, we arrive at
	\[ \frac{\tau}{2} \norm{\nabla \xi(t)}^2_H \le C \int_0^t \left( \norm{\nabla \xi}^2_H + \norm{\nabla \eta}^2_H + \norm{\nabla (\rho - \chi \xi)}^2_H + \left( 1 + \norm{\nabla \phi}^4_{6} \right) \norm{\xi}^2_H \right) \, \de s, \]
	and, by using \eqref{lin:gronwall} together with Gronwall's inequality, we deduce that 
	\begin{equation}
		\label{xi:linfv}
		\tau \norm{\xi}^2_{\LT \infty V} \le C \left( \norm{h}^2_{L^2(Q_T)} + \norm{k}^2_{L^2(Q_T)} \right).
	\end{equation}
	Moreover, since $\rho - \chi \xi \in \LT 2 V$ by \eqref{lin:gronwall}, we also infer that 
	\[ \norm{\rho}^2_{\LT 2 V} \le C \left( \norm{h}^2_{L^2(Q_T)} + \norm{k}^2_{L^2(Q_T)} \right). \]
	Finally, we test \eqref{eq:rho} in $H$ by $\partial_t (\rho - \chi \xi)$, which is still possible within the discretisation, and by usual techniques we get: 
	\begin{align*}
		& \norm{\rho_t}^2_H + \mezzo \ddt \norm{\nabla (\rho - \chi \xi)}^2_H = \chi (\rho_t, \xi_t)_H - (P'(\phib) (\sigmab + \chi (1-\phib) - \mub) \xi, \rho_t - \chi \xi_t)_H \\ 
		& \qquad + (P(\phib)(\rho - \chi \xi - \eta), \rho_t - \chi \xi_t)_H + (k, \rho_t - \chi \xi_t)_H \\
		& \quad \le \mezzo \norm{\rho_t}^2_H + C \norm{\xi_t}^2_H + C \norm{\rho}^2_H + C \norm{\eta}^2_H + C (1 + \norm{\sigmab + \chi (1-\phib) - \mub}^2_\infty ) \norm{\xi}^2_H + \norm{k}^2_H.
	\end{align*}
	Then, by integrating on $(0,t)$ for any $t \in (0,T)$ and by using Gronwall's lemma and \eqref{lin:gronwall}, we infer that
	\[ \norm{\rho}^2_{\HT 1 H} + \norm{\rho - \chi \xi}^2_{\LT \infty V} \le C \left( \norm{h}^2_{L^2(Q_T)} + \norm{k}^2_{L^2(Q_T)} \right) \] 
	and, by \eqref{xi:linfv}, we also get
	\[ \norm{\rho}^2_{\LT \infty V} \le C \left( \norm{h}^2_{L^2(Q_T)} + \norm{k}^2_{L^2(Q_T)} \right). \] 
	Before concluding, observe that, by comparison in \eqref{eq:xi} and \eqref{eq:rho}, we can also deduce that
	\[ \norm{\eta}^2_{L^2(0,T;W)} + \norm{\rho - \chi \xi}^2_{L^2(0,T;W)} \le C \left( \norm{h}^2_{L^2(Q_T)} + \norm{k}^2_{L^2(Q_T)} \right). \]
	Indeed, due to the regularity of strong solutions given by Theorem \ref{thm:strongsols}, one just has to notice that  
	\begin{align*}
		& \norm{P'(\phib) (\sigmab + \chi (1-\phib) - \mub) \xi}^2_{L^2(0,T;H)} \\
		& \quad \le P'_\infty \norm{\xi}^2_{L^\infty(0,T;H)} \underbrace{\int_0^T \norm{\sigmab + \chi (1-\phib) - \mub}^2_\infty \, \de t}_{\le C} \le  C \left( \norm{h}^2_{L^2(Q_T)} + \norm{k}^2_{L^2(Q_T)} \right).
	\end{align*} 
	At this point, with all these uniform estimates, it is easy to pass to the limit in the discretisation framework and prove the existence of a strong solution with the prescribed regularities. Moreover, since the system is linear, it is straightforward to prove the uniqueness of the solution from the energy estimate \eqref{lin:gronwall}. This concludes the proof of Theorem \ref{thm:linearised}.
\end{proof}

\subsection{Differentiability of the control-to-state operator}

We now want to study the \emph{control-to-state operator} $\mathcal{S}$, which associates to any control $(u,v) \in \Uad \times \Vad$ the corresponding solution of the system \eqref{eq:phi2}--\eqref{ic2}. We introduce the following space:
\begin{align*}
	\mathbb{Y} & := (H^1(0,T;H) \cap L^\infty(0,T;H^2(\Omega))) \times L^2(0,T;W) \\
	& \quad  \times (H^1(0,T;H) \cap L^\infty(0,T,V) \cap L^2(0,T;H^2(\Omega))). 
\end{align*}
Observe that the space of strong solutions $\mathbb{X}$ is continuously embedded into $\mathbb{Y}$, which is exactly the space where we proved the continuous dependence estimates. Indeed, from Theorem \ref{thm:strongsols} and Theorem \ref{thm:contdep} we respectively know that
\[ \mathcal{S}: (L^\infty(Q_T) \cap H^1(0,T;H)) \times L^\infty(Q_T) \to \mathbb{X} \quad \text{is \emph{well-defined} and}  \]
\[ \mathcal{S}: (L^\infty(Q_T) \cap H^1(0,T;H)) \times L^\infty(Q_T) \to \mathbb{Y} \quad \text{is \emph{locally Lipschitz-continuous}.}  \]
Now, for $R>0$, we fix an open set $\mathcal{U}_R \times \mathcal{V}_R \subseteq (L^\infty(Q_T) \cap H^1(0,T;H)) \times L^\infty(Q_T)$ such that $\Uad \times \Vad \subseteq \mathcal{U}_R \times \mathcal{V}_R$. Indeed, by hypothesis \ref{C3}, we can take:
\begin{align*}
	& \mathcal{U}_R := \{ u \in L^{\infty}(Q_T) \cap H^1(0,T;H) \mid \norm{u}_{L^\infty(Q_T)} \le M' + R, \norm{u}_{H^1(0,T;H)} \le M + R \} \\
	& 	\mathcal{V}_R := \{ v \in L^{\infty}(Q_T) \mid \norm{v}_{L^\infty(Q_T)} \le M' + R \}.
\end{align*}
Note that, in $\mathcal{U}_R \times \mathcal{V}_R$, the continuous dependence estimate of Theorem \ref{thm:contdep} holds with $K$ depending only on $R$ and the fixed data of the system.
Our aim is to show that $\mathcal{S}: \mathcal{U}_R \times \mathcal{V}_R \to \mathbb{W}$ is also \emph{Fréchet-differentiable} in the larger space:
\[ \mathbb{W} = (H^1(0,T;H) \cap L^\infty(0,T;V)) \times L^2(0,T;V) \times (H^1(0,T;H) \cap L^\infty(0,T;V)). \]
Indeed, we can prove the following theorem:

\begin{theorem}
	\label{thm:frechet}
	Assume hypothesis \emph{\ref{ass:coeff}--\ref{ass:fc0}}, \emph{\ref{ass:j2}--\ref{ass:initial2}} and \emph{\ref{C4}}. Then $\mathcal{S}: \mathcal{U}_R \times \mathcal{V}_R \to \mathbb{W}$ is Fréchet-differentiable, i.e. for any $(\ub,\vb) \in \mathcal{U}_R \times \mathcal{V}_R$ there exists a unique Fréchet-derivative $D\mathcal{S}(\ub, \vb) \in \mathcal{L}((L^\infty(Q_T) \cap H^1(0,T;H)) \times L^\infty(Q_T), \mathbb{W})$ such that:
	\begin{equation}
		\label{frechet:diff}
		\frac{ \norm{ \mathcal{S}(\ub+h, \vb+k) - \mathcal{S}(\ub, \vb) - D\mathcal{S}(\ub, \vb)(h,k)  }_{\mathbb{W}} }{ \norm{(h,k)}_{(L^2(Q_T))^2} } \to 0 \quad \text{as } \norm{(h,k)}_{(L^2(Q_T))^2} \to 0.
	\end{equation}
	Moreover, for any $(h,k) \in (L^\infty(Q_T) \cap H^1(0,T;H)) \times L^\infty(Q_T)$, the Fréchet-derivative at $(h,k)$, which we denote by $D\mathcal{S}(\ub, \vb)(h,k)$, is defined as the solution $(\xi, \eta, \rho)$ to the linearised system \eqref{eq:xi}--\eqref{icl} corresponding to $(\phib, \mub, \sigmab) = \mathcal{S}(\ub, \vb)$, with data $(h,k)$. 
\end{theorem}

\begin{remark}
	Note that, by Theorem \ref{thm:linearised}, $D\mathcal{S}(\ub, \vb)$ as defined above actually belongs to the space of continuous linear operators $\mathcal{L}((L^\infty(Q_T) \cap H^1(0,T;H)) \times L^\infty(Q_T), \mathbb{W})$. Observe also that \eqref{frechet:diff} shows Fréchet-differentiability with respect to the $L^2(Q_T)$ norm, but clearly, since $(L^\infty(Q_T) \cap H^1(0,T;H)) \times L^\infty(Q_T) \hookrightarrow (L^2(Q_T))^2$, this also implies Fréchet-differentiability in the correct space.
\end{remark}

\begin{proof}
	Clearly, it is sufficient to prove the result for any small enough perturbation $(h,k)$, i.e. we fix $\Lambda > 0$ and consider only perturbations such that
	\begin{equation}
		\label{hk:bound}
		\norm{(h,k)}_{(L^2(Q_T))^2} \le \Lambda.
	\end{equation} 
	Now, we fix $(\ub, \vb)$ and $(h,k)$ as above and consider
	\begin{align*}
		& (\phi, \mu, \sigma) := \mathcal{S}(\ub+h, \vb+k), \\
		& (\phib, \mub, \sigmab) := \mathcal{S}(\ub, \vb), \\
		& (\xi, \eta, \rho) \text{ as the solution to \eqref{eq:xi}--\eqref{icl} with respect to } (h,k).
	\end{align*}
	In order to show the Fréchet-differentiability, then, it is enough to show that there exists a constant $C>0$, depending only on the parameters of the system and possibly on $\Lambda$, and an exponent $s > 2$ such that
	\[ \norm{ (\phi, \mu, \sigma) - (\phib, \mub, \sigmab) - (\xi, \eta, \rho) }^2_{\mathbb{W}} \le C \norm{(h,k)}^s_{(L^2(Q_T))^2}. \]
	If we introduce the additional variables
	\begin{align*}
		& \psi := \phi - \phib - \xi \in H^1(0,T;H) \cap L^\infty(0,T; V), \\
		& \zeta  := \mu - \mub - \eta \in L^2(0,T;W), \\
		& \theta := \sigma - \sigmab - \rho \in H^1(0,T;H) \cap L^\infty(0,T;V) \cap L^2(0,T;W), \\
		& \theta - \chi \psi \in L^2(0,T;W).
	\end{align*}
	which by Theorem \ref{thm:strongsols} and Theorem \ref{thm:linearised} enjoy the regularities shown above, then this is equivalent to showing that
	\begin{equation}
		\label{frechet:aim}
		\norm{ (\psi, \zeta, \theta) }_{\mathbb{W}}^2 \le C \norm{(h,k)}^s_{(L^2(Q_T))^2}.
	\end{equation}
	Moreover, by inserting the equations solved by the variables in the definitions of $\psi$, $\zeta$ and $\theta$ and exploiting the linearity of the involved differential operators, we infer that these new variables satisfy the equations:
	\begin{alignat}{2}
		& \partial_t \psi = \Delta \zeta + Q^h - U^h \,\,\, && \text{in } Q_T,  \label{eq:psi}\\
		& \zeta = \tau \partial_t \psi + AF^h + Ba \psi - BJ \ast \psi - \chi \theta \quad && \text{in } Q_T,  \label{eq:zeta} \\
		& \partial_t \theta = \Delta \theta - \chi \Delta \psi - Q^h \quad && \text{in } Q_T, \label{eq:theta}
	\end{alignat}
	together with boundary and initial conditions:
	\begin{alignat}{2}
		& \partial_{\n} \zeta = \partial_{\n} (\theta - \chi \psi) = 0 \qquad && \text{on } \partial \Omega \times (0,T), \label{bcd} \\
		& \psi(0) = 0, \quad \theta(0) = 0 \qquad && \text{in } \Omega, \label{icd}
	\end{alignat}
	where:
	\begin{align*}
		 F^h & = F'(\phi) - F'(\phib) - F''(\phib)\xi, \\
		 U^h & = \hh(\ub + h) - \hh(\phib) \ub - \hh(\phib) - \hh(\phib) \ub \xi, \\
		 Q^h & = P(\phi)(\sigma + \chi(1-\phi) - \mu) - P(\phib)(\sigmab + \chi (1- \phib) - \mub) \\ 
		& \quad - P(\phib)(\rho - \chi \xi - \eta) - P'(\phib) (\sigmab - \chi (1-\phib) - \mub) \xi.
	\end{align*}
	Before going on, we can rewrite in a better way the terms $F^h$, $U^h$ and $Q^h$, by using the following version of Taylor's theorem with integral remainder for a real function $f \in \mathcal{C}^2$ at a point $x_0 \in \R$:
	\[ f(x) = f(x_0) + f'(x_0) (x-x_0) + \left( \int_0^1 (1-z) f''(x_0 + z(x-x_0)) \, \de z \right) (x-x_0)^2. \]
	Indeed, with straightforward calculations one can see that
	\begin{align*}
		& F^h = F''(\phib) \underbrace{ (\phi - \phib - \xi) }_{= \psi} + R_1^h (\phi - \phib)^2, \\
		& U^h = \hh'(\phib) \underbrace{ (\phi - \phib - \xi) }_{= \psi} \ub + (\hh(\phi) - \hh(\phib)) h + R_2^h (\phi - \phib)^2 \ub, 
	\end{align*}
	and also, up to adding and subtracting some additional terms, that
	\begin{align*}
		& Q^h = P(\phib) (\theta - \chi \psi - \zeta) + P'(\phib) (\sigmab + (1-\chi) \phib - \mub) \, \psi  \\
		& \qquad + (P(\phi) - P(\phib))[ (\sigma - \sigmab) - \chi (\phi - \phib) - (\mu - \mub) ] + R_3^h (\sigmab + (1-\chi) \phib - \mub) (\phi - \phib)^2,
	\end{align*}
	where
	\begin{align*}
		& R_1^h = \int_0^1 (1-z) F'''(\phib + z(\phi - \phib)) \, \de z, \\
		& R_2^h = \int_0^1 (1-z) \hh''(\phib + z(\phi - \phib)) \, \de z, \\
		& R_3^h = \int_0^1 (1-z) P''(\phib + z(\phi - \phib)) \, \de z.
	\end{align*}
	Regarding these last remainder terms, we can immediately say that there exists a constant $C>0$, depending only on the parameters, $\Lambda$ and $T$, such that
	\begin{equation}
		\label{remainder1}
		\norm{R_1^h}_{L^\infty(Q_T)}, \norm{R_2^h}_{L^\infty(Q_T)}, \norm{R_3^h}_{L^\infty(Q_T)} \le C_\Lambda.
	\end{equation}
	Indeed, by using the fact that $F \in \mathcal{C}^3$ and $\phi$, $\phib$ are bounded uniformly in $L^\infty(Q_T)$ with values in $[-s^*, s^*]$ by Theorem \ref{thm:strongsols}, we have that
	\begin{align*}
		\norm{R_1^h}_{L^\infty(Q_T)} \le \int_0^1 (1-z) \norm{F'''(\phib + z(\phi- \phib)) }_{L^\infty(Q_T)} \, dz \le C_{\Lambda} \int_0^1 (1-z) \, \de z \le C_{\Lambda}.
	\end{align*}
	The other two terms are done in an analogous way, by exploiting hypothesis \ref{C4}. Moreover, for later purposes, we also observe that
	\begin{equation}
		\label{remainder2}
		\norm{\nabla R_1^h}_{L^\infty(0,T;L^6(\Omega))} \le C_{\Lambda}.
	\end{equation}
	Indeed, by using Lebesgue's theorem to differentiate inside the integral sign, Jensen's inequality, Fubini's theorem, the embedding $H^2(\Omega) \hookrightarrow W^{1,6}(\Omega)$, $F\in \mathcal{C}^4$, the separation property and the uniform bound on $\phi$, $\phib$ in $L^\infty(0,T; H^2(\Omega))$ as before, we have:
	\begin{align*}
		\norm{\nabla R_1^h}_{6} & \le \left( \int_\Omega \left( \int_0^1 (1-z) \abs{(\nabla \phib + z (\nabla \phi - \nabla \phib) ) F^{(4)}(\phib + z (\phi - \phib))} \, \de z \right)^6 \de x  \right)^{1/6} \\
		& \le C \Bigg( \int_0^1 (1-z) \norm{ (\nabla \phib + z (\nabla \phi - \nabla \phib) ) F^{(4)}(\phib + z (\phi - \phib)) }^6_{6} \, \de z \Bigg)^{1/6}\\
		& \le C \Bigg( \int_0^1 (1-z) \underbrace{ \norm{ F^{(4)}(\phib + z (\phi - \phib)) }^6_{L^\infty(Q_T)} }_{\le C_{\Lambda}} \norm{ \nabla \phib + z (\nabla \phi - \nabla \phib) }^6_{6} \, \de z \Bigg)^{1/6} \\
		& \le C_{\Lambda} \left( 2 \norm{\phib}_{\Wx{1,6}} + \norm{\phi}_{\Wx{1,6}} \right) \\
		& \le C_{\Lambda} \left( 2 \norm{\phib}_{L^\infty(0,T;H^2(\Omega))} + \norm{\phi}_{L^\infty(0,T;H^2(\Omega))} \right) \le \bar{C}_\Lambda.
	\end{align*}
	Then, by taking the essential supremum on $(0,T)$ we deduce \eqref{remainder2}.
	
	Now we can start with a priori estimates on the system \eqref{eq:psi}--\eqref{icd}. Firstly, we test \eqref{eq:psi} in $H$ by $\zeta$ and, by using \eqref{remainder1}, Cauchy-Schwarz and Young's inequalities, the embedding $\Hx 2 \hookrightarrow \Lx \infty$ and the local Lipschitz continuity of $\hh \in \mathcal{C}^1$, we get:
	\begin{align*}
		& (\psi_t, \zeta)_H + \norm{\nabla \zeta}^2_H = (Q^h, \zeta)_H - (\hh'(\phib) \ub \psi, \zeta )_H - ((\hh(\phi) - \hh(\phib)) h, \zeta)_H + (R_2^h (\phi - \phib)^2 \ub, \zeta)_H \\
		& \quad \le (Q^h, \zeta)_H + \frac{\eps}{2} \norm{\zeta}^2_H + C_\eps \norm{\psi}^2_H + C_\eps \norm{\phi - \phib}^4_H + \norm{\hh(\phi) - \hh(\phib)}_\infty \norm{h}_H \norm{\zeta}_H \\
		& \quad \le (Q^h, \zeta)_H + \eps \norm{\zeta}^2_H + C_\eps \norm{\psi}^2_H + C_\eps \norm{\phi - \phib}^4_H + C_\eps \norm{\phi - \phib}^2_{\Hx 2} \norm{h}^2_H, 
	\end{align*}
	where $\eps > 0$ from Young's inequality is to be chosen later. Next, to estimate the term $(Q^h, \zeta)_H$, we use use again \eqref{remainder1}, H\"older and Young's inequalities, the embedding $V \hookrightarrow \Lx 4$ and the local Lipschitz continuity of $P \in \mathcal{C}^1$, indeed:
	\begin{align*}
		(Q^h, \zeta)_H & =  (P(\phib) (\theta - \chi \psi - \zeta), \zeta)_H + (P'(\phib) (\sigmab + (1-\chi) \phib - \mub) \, \psi, \zeta)_H \\
		& \quad + ((P(\phi) - P(\phib))[ (\sigma - \sigmab) - \chi (\phi - \phib) - (\mu - \mub) ], \zeta)_H \\ 
		& \quad + (R_3^h (\sigmab + (1-\chi) \phib - \mub) (\phi - \phib)^2, \zeta)_H \\
		& \le (P(\phib) (\theta - \chi \psi - \zeta), \zeta)_H + \eps \norm{\zeta}^2_H + C_\eps (1 + \norm{\sigmab + (1-\chi) \phib - \mub}^2_{\infty}) \norm{\psi}^2_H \\
		& \quad + C_\eps \norm{\phi - \phib}^2_V ( \norm{\sigma - \sigmab}^2_V + \chi \norm{\phi - \phib}^2_V + \norm{\mu - \mub}^2_V ) \\
		& \quad + C_\eps \norm{\sigmab + (1-\chi) \phib - \mub}^2_{\infty} \norm{\phi - \phib}^4_V,
	\end{align*}
	where again $\eps >0$ is yet to be chosen and without loss of generality, up to changing the constants $C>0$, can be thought the same as before. Then, by putting all together and integrating on $(0,t)$ for any $t \in (0,T)$, we have: 
	\begin{align*}
		& \int_0^t (\psi_t, \zeta)_H \, \de s + \int_0^t \norm{\nabla \zeta}^2_H \, \de s - \int_0^t (P(\phib) (\theta - \chi \psi - \zeta), \zeta)_H \, \de s \le 2\eps \int_0^t \norm{\zeta}^2_H \, \de s \\ 
		& \quad + C_\eps \int_0^t \norm{\theta}^2_H \, \de s + C_\eps \int_0^t (1 + \norm{\sigmab + (1-\chi) \phib - \mub}^2_{\infty}) \norm{\psi}^2_H \, \de s + C_\eps \int_0^T \norm{\phi - \phib}^4_H \, \de s \\ 
		& \quad + C_\eps \norm{\phi - \phib}^2_{L^\infty(0,T;\Hx 2)} \int_0^T \norm{h}^2_H \, \de s + C_\eps \norm{\phi - \phib}^4_{L^\infty(0,T;V)} \int_0^T  \norm{\sigmab + (1-\chi) \phib - \mub}^2_{\infty} \, \de s \\
		& \quad + C_\eps \norm{\phi - \phib}^2_{L^\infty(0,T;V)} \int_0^T ( \norm{\sigma - \sigmab}^2_V + \chi \norm{\phi - \phib}^2_V + \norm{\mu - \mub}^2_V ) \, \de s. 
	\end{align*}
	At this point, we use the continuous dependence estimate found in Theorem \ref{thm:contdep} and the fact that $\norm{\sigmab + (1-\chi) \phib - \mub}^2_{\infty} \in L^1(0,T)$ uniformly, to obtain:
	\begin{align}
		\label{diff:est1}
			& \int_0^t (\psi_t, \zeta)_H \, \de s + \int_0^t \norm{\nabla \zeta}^2_H \, \de s - \int_0^t (P(\phib) (\theta - \chi \psi - \zeta), \zeta)_H \le 2\eps \int_0^t \norm{\zeta}^2_H \, \de s + C_\eps \int_0^t \norm{\theta}^2_H \, \de s \nonumber \\
			& \quad + C_\eps \int_0^t (1 + \norm{\sigmab + (1-\chi) \phib - \mub}^2_{\infty}) \norm{\psi}^2_H \, \de s + C_\eps \norm{h}^4_{L^2(Q_T)} + C_\eps \norm{k}^4_{L^2(Q_T)}.  
	\end{align}
	Next, we test \eqref{eq:zeta} in $H$ by $- \partial_t \psi$ and, by using Leibniz's rule, \eqref{remainder1}, Cauchy-Schwarz and Young's inequalities, $F \in \mathcal{C}^3$, Remark \ref{F:derivatives} and the embedding $\Hx 2 \hookrightarrow \Lx \infty$, as well as $V \hookrightarrow \Lx4$, we infer that
	\begin{align*}
		& (\zeta, - \psi_t)_H + \tau \norm{\psi_t}^2_H + \mezzo \ddt ( (AF''(\phib) + Ba) \psi, \psi )_H - \chi (\theta, \psi_t)_H \\
		& \quad = (A F'''(\phib) \phib_t \psi, \psi)_H - (A R_1^h (\phi - \phib)^2, \psi_t)_H + (B J \ast \psi, \psi_t)_H \\
		& \quad \le \frac{\tau}{2} \norm{\psi_t}^2_H + C (1 + \norm{\phib_t}_{\Hx 2}) \norm{\psi}^2_H + C \norm{\phi - \phib}_V^4, 
	\end{align*}
	where $\norm{\phib_t}_{\Hx 2} \in \Lt 2$ by Theorem \ref{thm:strongsols}. Then, by integrating on $(0,t)$, for any $t \in (0,T)$, and, by exploiting hypothesis \ref{ass:fc0} and the continuous dependence result of Theorem \ref{thm:contdep}, we find that
	\begin{equation}
		\label{diff:est2}
		\begin{split}
			& \int_0^t (\zeta, - \psi_t) \, \de s + \frac{c_0}{2} \norm{\psi(t)}^2_H + \frac{\tau}{2} \int_0^t \norm{\psi_t}^2_H \, \de s - \int_0^t \chi (\theta, \psi_t)_H \, \de s \le \\
			& \quad \le C \int_0^t (1 + \norm{\phib_t}_{\Hx 2}) \norm{\psi}^2_H \, \de s +  C \norm{h}^4_{L^2(Q_T)} + C \norm{k}^4_{L^2(Q_T)}.
		\end{split}
	\end{equation}
	For the third estimate, we test \eqref{eq:theta} in $H$ by $\theta - \chi \psi$ and with similar techniques we find that
	\begin{align*}
		&\mezzo \ddt \norm{\theta}^2_H - \chi (\theta_t, \psi)_H + \norm{\nabla (\theta - \chi \psi)}^2_H = ( - Q^h, \theta - \chi \psi)_H \\
		& \quad \le - (P(\phib) (\theta - \chi \psi - \zeta), \theta - \chi \psi)_H + C \norm{\theta}^2_H + C (1 + \norm{\sigmab + (1-\chi) \phib - \mub}^2_{\infty}) \norm{\psi}^2_H \\
		& \qquad + C \norm{\phi - \phib}^2_V ( \norm{\sigma - \sigmab}^2_V + \chi \norm{\phi - \phib}^2_V + \norm{\mu - \mub}^2_V ) + C \norm{\sigmab + (1-\chi) \phib - \mub}^2_{\infty} \norm{\phi - \phib}^4_V.
	\end{align*}
	Therefore, by integrating on $(0,t)$ for any $t\in (0,T)$ and by using again Theorem \ref{thm:contdep}, we obtain:
	\begin{equation}
		\label{diff:est3}
		\begin{split}
			& \mezzo \norm{\theta(t)}^2_H + \int_0^t \norm{\nabla (\theta - \chi \psi)}^2_H \, \de s - \int_0^t \chi (\theta_t, \psi)_H \, \de s \\ 
			& \quad + \int_0^t (P(\phib) (\theta - \chi \psi - \zeta), \theta - \chi \psi)_H \, \de s \le C \int_0^t \norm{\theta}^2_H \, \de s \\ 
			& \qquad + C \int_0^t (1 + \norm{\sigmab + (1-\chi) \phib - \mub}^2_\infty) \norm{\psi}^2_H \, \de s + C \norm{h}^4_{L^2(Q_T)} + C \norm{k}^4_{L^2(Q_T)}.  \\ 
		\end{split}
	\end{equation}
	Finally, we test \eqref{eq:zeta} in $H$ by $\delta \zeta$, with $\delta \gs 0$ to be chosen later, getting:
	\begin{align*}
		\delta \norm{\zeta}^2_H & = \delta \tau (\psi_t, \zeta)_H + \delta (AF''(\phib) \psi, \zeta)_H + \delta (AR_1^h(\phi - \phib)^2, \zeta)_H \\ 
		& \quad + \delta (Ba \psi, \zeta)_H - \delta (B J \ast \psi, \zeta)_H - \delta \chi (\theta, \zeta)_H.
	\end{align*} 
	Exactly as we already did in the proof of Theorems \ref{thm:contdep} and \ref{thm:linearised}, by using Cauchy-Schwarz and Young's inequalities with $\gamma > 0$ to be chosen later and the same $\eps > 0$ used before, up to changing the constants $C>0$, we can estimate
	\[ \delta \norm{\zeta}^2_H \le C_\gamma \delta \tau \norm{\psi_t}^2_H + \gamma \delta \tau \norm{\zeta}^2_H + \eps \norm{\zeta}^2_H + C_{\eps, \delta} \norm{\psi}^2_H + C_{\eps, \delta} \norm{\theta}^2_H + C_{\eps, \delta} \norm{\phi - \phib}^4_V.  \]
	Then, by integrating on $(0,t)$, for any $t\in (0,T)$ and using Theorem \ref{thm:contdep}, we obtain:
	\begin{align}
		\label{diff:est4}
			& \left( 1 - \eps - \gamma \delta \tau \right) \int_0^t \norm{\zeta}^2_H \, \de s \\ 
			& \quad \le C_\gamma \delta \tau  \int_0^t \norm{\psi_t}^2_H \, \de s + C_{\eps, \delta} \int_0^t \left( \norm{\psi}^2_H + \norm{\theta}^2_H \right) \, \de s +  C_{\eps, \delta} \norm{h}^4_{L^2(Q_T)} + C_{\eps, \delta} \norm{k}^4_{L^2(Q_T)}. \nonumber
	\end{align}
	Now, we can sum the inequalities \eqref{diff:est1}, \eqref{diff:est2}, \eqref{diff:est3} and \eqref{diff:est4}, so that after cancellations and after observing that
	\[ - \int_0^t \left( \chi (\psi_t, \theta)_H + \chi (\psi, \theta_t)_H \right) \, \de s = - \frac{\chi}{2} (\psi(t), \theta(t))_H \ge - \frac{\chi^2}{2} \norm{\psi(t)}^2_H - \frac{1}{4} \norm{\theta(t)}^2_H, \]
	we arrive at the inequality:
	\begin{align*}
		&  \mezzo (c_0-\chi^2) \norm{\psi(t)}^2_H + \frac{3}{8} \norm{\theta(t)}^2_H + \left( 1 - 3\eps - \gamma \delta \tau \right) \int_0^t \norm{\zeta}^2_H \, \de s \\
		& \qquad + \left( \frac{\tau}{2} - C_\gamma \delta \tau \right) \int_0^t \norm{\psi_t}^2_H \, \de s + \int_0^t \norm{\nabla (\theta - \chi \psi)}^2_H \, \de s + \int_0^t \norm{\nabla \zeta}^2_H \, \de s \\
		& \qquad + \int_0^t \underbrace{(P(\phib) (\theta - \chi \psi - \zeta), \theta - \chi \psi - \zeta)_H}_{\ge 0} \, \de s \\ 
		& \quad \le C_{\eps,\delta} \int_0^t \norm{\theta}^2_H \, \de s + C_{\eps,\delta} \int_0^t (1 + \norm{\sigmab + (1-\chi) \phib - \mub}^2_{\infty} + \norm{\phib_t}_{\Hx 2}) \norm{\psi}^2_H \, \de s \\ 
		& \qquad + C_{\eps,\delta} \norm{h}^4_{L^2(Q_T)} + C_{\eps,\delta} \norm{k}^4_{L^2(Q_T)}.
	\end{align*}
	Observe that once again, exactly as in the proof of Theorem \ref{thm:contdep} (see equation \eqref{eq:epsdelta}), we can respectively choose $\eps \gs 0$, $\gamma = \gamma(\tau) \gs 0$ and then $\delta = \delta(\tau) \gs 0$ such that the constants above are strictly positive, Therefore, we can apply Gronwall's inequality and recover the estimate:
	\begin{equation}
	\label{diff:gronwall}
		\begin{split}
			& \tau \norm{\psi}^2_{H^1(0,T;H)} + \norm{\psi}^2_{L^\infty(0,T;H)} + \norm{\zeta}^2_{L^2(0,T;V)} + \norm{\theta}^2_{L^\infty(0,T;H)} \\
			& \quad + \norm{\nabla (\theta - \chi \psi)}^2_{L^2(0,T;H)} \le C_\tau \norm{h}^4_{L^2(Q_T)} + C_\tau \norm{k}^4_{L^2(Q_T)},
		\end{split}
	\end{equation}
	where the constant $C_\tau>0$ depends only on the parameters of the system, $R$,  $\Lambda$ and possibly on $\tau$. As in the previous proofs, the dependence on $\tau$ can be avoided by arguing as in Remark \ref{rmk:epsdelta}.  
	
	To conclude, we just need two more estimates. First we test the gradient of \eqref{eq:zeta}  by $\nabla \psi$ in $H$ and, after using the chain rule and adding and subtracting $\chi^2 (\nabla \psi, \nabla \psi)_H$, we have:
	\begin{align*}
		& (\nabla \zeta, \nabla \psi)_H = \frac{\tau}{2} \ddt \norm{\nabla \psi}^2_H + ( (AF''(\phib) + Ba ) \nabla \psi, \nabla \psi )_H + (A F'''(\phib) \nabla \phib \psi, \nabla \psi)_H \\
		& \quad +  (B \nabla a \psi, \nabla \psi)_H + (A \nabla R_1^h (\phi - \phib)^2, \nabla \psi)_H + (2A R_1^h (\phi - \phib) (\nabla \phi - \nabla \phib), \nabla \psi)_H \\
		& \quad - (B \nabla J \ast \psi. \nabla \psi)_H - \chi (\nabla (\theta - \chi \psi), \nabla \psi)_H - \chi^2 (\nabla \psi, \nabla \psi)_H.
	\end{align*}
	Then, by using \eqref{remainder1}, \eqref{remainder2}, $F \in \mathcal{C}^3$, Remark \ref{F:derivatives}, H\"older's inequality, \eqref{gn:ineq} and Sobolev embeddings, we can infer that
	\begin{align*}
		& \frac{\tau}{2} \ddt \norm{\nabla \psi}^2_H + (c_0 - \chi^2) \norm{\nabla \psi}^2_H \\
		& \quad \le \norm{\nabla \zeta}_H \norm{\nabla \psi}_H + C \norm{F'''(\phib) }_{\infty} \norm{\nabla \phib}_{6} \norm{\psi}_H^{1/2} \norm{\nabla \psi}_H^{3/2} + 2Bb \norm{\psi}_H \norm{\nabla \psi}_H \\
		& \qquad + A \norm{\nabla R_1^h}_{6} \norm{\phi - \phib}^2_{6} \norm{\nabla \psi}_H + 2A \norm{R_1^h}_{\infty} \norm{\phi - \phib}_{6} \norm{\nabla \phi - \nabla \phib}_{3} \norm{\nabla \psi}_H \\
		& \qquad + \chi \norm{\nabla (\theta - \chi \psi)}_H \norm{\nabla \psi}_H \\
		&  \quad \le C \norm{\nabla \psi}^2_H + C \norm{\nabla \zeta}^2_H + C \norm{\nabla (\theta - \chi \psi)}^2_H + C(1 + \norm{\nabla \phib}^4_{6}) \norm{\psi}^2_H \\
		& \qquad + \norm{\phi - \phib}^4_V + \norm{\phi - \phib}^4_{\Hx 2},  
	\end{align*}
	where $\norm{\nabla \phib}^4_{6} \in L^\infty(0,T)$. Therefore, by integrating on $(0,t)$ for any $t\in (0,T)$, using Theorem \eqref{thm:contdep} and the previous estimate \eqref{diff:gronwall}, after applying Gronwall's lemma we obtain the following:
	\[ \norm{\psi}^2_{L^\infty(0,T;V)} \le C \norm{h}^4_{L^2(Q_T)} + C \norm{k}^4_{L^2(Q_T)}. \]
	Moreover, by comparison with the estimate for $\nabla (\theta - \chi \psi)$, we can also see that
	\[ \norm{\theta}^2_{L^2(0,T;V)} \le C \norm{h}^4_{L^2(Q_T)} + C \norm{k}^4_{L^2(Q_T)}. \]
	Finally, we test \eqref{eq:theta} in $H$ by $\partial_t (\theta - \chi \psi)$ and, with Cauchy-Schwarz and Young's inequalities, we get:
	\begin{align*}
		& \norm{\theta_t}^2_H + \mezzo \ddt \norm{\nabla (\theta - \chi \psi)}^2_H = \chi (\theta_t, \psi_t)_H - (Q^h, \theta_t - \chi \psi_t)_H \\
		& \quad \le \mezzo \norm{\theta_t}^2_H + C \norm{\psi_t}^2_H + C \norm{Q^h}^2_H. 
	\end{align*}
	Regarding $\norm{Q^h}^2_H$, one can argue exactly like we did for the first estimate and then, after integrating in time and applying Theorem \ref{thm:contdep}, arrive at:
	\begin{align*}
		& \mezzo \int_0^t \norm{\theta_t}^2_H \, \de s + \norm{\nabla (\theta - \chi \psi)(t)}^2_H \\
		& \quad \le C \int_0^t \left( \norm{\psi_t}^2_H + \norm{\theta}^2_H + \norm{\zeta}^2_H \right) \, \de s + C \int_0^t (1 + \norm{\sigmab + (1-\chi) \phib - \mub}^2_{\infty}) \norm{\psi}^2_H \, \de s \\
		& \qquad + C \norm{h}^4_{L^2(Q_T)} + C \norm{k}^4_{L^2(Q_T)}.  
	\end{align*}
	Hence, by applying Gronwall's lemma and the previous estimate \eqref{diff:gronwall}, we infer that
	\[ \norm{\theta}^2_{H^1(0,T;H)} + \norm{\theta - \chi \psi}^2_{L^\infty(0,T;V)} \le C \norm{h}^4_{L^2(Q_T)} + C \norm{k}^4_{L^2(Q_T)}, \]
	and then also by comparison 
	\[ \norm{\theta}^2_{L^\infty(0,T;V)} \le C \norm{h}^4_{L^2(Q_T)} + C \norm{k}^4_{L^2(Q_T)}. \]
	In the end, by putting all together, we showed that
	\[ \norm{\psi}^2_{H^1(0,T;H) \cap L^\infty(0,T;V)} + \norm{\zeta}^2_{L^2(0,T;V)} + \norm{\theta}^2_{H^1(0,T;H) \cap L^\infty(0,T;V)} \le  C \norm{h}^4_{L^2(Q_T)} + C \norm{k}^4_{L^2(Q_T)},  \]
	which is exactly \eqref{frechet:aim} with $s=4>2$ and so we are done.
\end{proof}

\subsection{Adjoint system and first-order necessary conditions}

In order to write down the necessary conditions of optimality and make them useful for applications, we now need to introduce the adjoint system to the optimal control problem (CP). Indeed, we fix an optimal state $(\phib, \mub, \sigmab) = \mathcal{S}(\ub, \vb)$, then, by using the formal Lagrangian method with adjoint variables $(p,q,r)$, one can find that the adjoint system, which is formally solved by these variables, has the following form:
\begin{alignat}{2}
	& - \partial_t (p + \tau q) + AF''(\phib) q + Baq - BJ \ast q + \chi \Delta r + \chi P(\phib) (p-r) \nonumber \\ 
	& \qquad - P'(\phib) (\sigmab + \chi (1-\phib) - \mub)(p-r) + p \hh'(\phib) \ub = \alpha_Q (\phib - \phi_Q) \,\,\, && \text{in } Q_T,  \label{eq:p}\\
	& - q - \Delta p + P(\phib)(p-r) = 0 \quad && \text{in } Q_T,  \label{eq:q} \\
	& - \partial_t r - \Delta r - \chi q  - P(\phib) (p-r) = \beta_Q (\sigmab - \sigma_Q) \quad && \text{in } Q_T, \label{eq:r}
\end{alignat}
together with the following boundary and final conditions:
\begin{alignat}{2}
	& \partial_{\n} p = \partial_{\n} r = 0 \qquad && \text{on } \partial \Omega \times (0,T), \label{bca} \\
	& (p + \tau q)(T) = \alpha_\Omega (\phib(T) - \phi_\Omega), \quad r(T) = \beta_\Omega (\sigmab(T) - \sigma_\Omega) \qquad && \text{in } \Omega. \label{fca}
\end{alignat}
First, we prove the well-posedness of this adjoint system in the following theorem.

\begin{theorem}
	\label{thm:adjoint}
	Assume hypotheses \emph{\ref{ass:coeff}--\ref{ass:fc0}}, \emph{\ref{ass:j2}--\ref{ass:initial2}} and \emph{\ref{C1}--\ref{C4}} and let $(\phib, \mub, \sigmab) \in \mathbb{X}$ be the strong solution to \eqref{eq:phi2}--\eqref{ic2}, corresponding to $(\ub, \vb) \in \Uad \times \Vad$. Then the adjoint system \eqref{eq:p}--\eqref{fca} admits a \emph{unique strong solution} such that
	\begin{align*}
		& p + \tau q \in H^1(0,T;H) \cap L^\infty(0,T; H), \\
		& \tau q \in L^2(0,T;H), \\
		& p \in L^2(0,T;W), \\
		& r \in H^1(0,T;H) \cap L^\infty(0,T;V) \cap L^2(0,T;W),
	\end{align*}
	which fulfils \eqref{eq:p}--\eqref{fca} almost everywhere in the respective sets. 
\end{theorem}

\begin{proof}
	We observe that \eqref{eq:p}--\eqref{fca} is a backward \emph{linear} system, so once we prove well-posedness through an energy estimate, the uniqueness of the solution will follow easily. Also in this case the proof can be made rigorous through a Galerkin approximation scheme, the details of which will be left to the reader. We will limit ourselves to derive formal energy estimates, which can then be used to pass to the limit in the approximation. 
	
	We start by testing \eqref{eq:p} in $H$ by $p + \tau q$, so that, by also adding and subtracting $\tau q$ on every occurrence of $p$ in the source terms, we get:
	\begin{align*}
		& - \mezzo \ddt \norm{p + \tau q}^2_H = - (AF''(\phib) q, p +\tau q)_H - (Baq, p+ \tau q)_H + (BJ \ast q, p + \tau q)_H \\ 
		& \quad - \chi (\Delta r, p+\tau q)_H - \chi (P(\phib) (p + \tau q - \tau q -r), p +\tau q)_H \\  
		& \quad + (P'(\phib) (\sigmab + \chi (1-\phib) - \mub)(p + \tau q - \tau q -r), p + \tau q)_H \\
		&  \quad - ( (p + \tau q -\tau q) \hh'(\phib) \ub, p +\tau q)_H + (\alpha_Q (\phib - \phi_Q), p +\tau q)_H.
	\end{align*}
	Then, by using the separation property for $\phib$, the regularity properties of the functions $F,P,\hh$, Cauchy-Schwarz inequality and Young's inequality with $\eps > 0$ to be chosen later, we infer that
	\begin{equation}
		\label{adj:est1}
		\begin{split}
			& - \mezzo \ddt \norm{p + \tau q}^2_H \le \tau \eps \norm{q}^2_H + \frac{1}{8} \norm{\Delta r}^2_H + C_{\eps} \norm{r}^2_H \\
			& \quad + C_{\eps} (1 + \norm{\sigmab + \chi (1 - \phib) - \mub}^2_{\infty} + \norm{\ub}^2_{\infty}) \norm{p+ \tau q}^2_H + C \norm{\alpha_Q (\phib - \phi_Q)}^2_H, 
		\end{split}
	\end{equation}
	where the constants $C_{\eps}>0$ depend only on the parameters and on $\eps$. Next, we multiply \eqref{eq:q}  by $M >0$, to be chosen later, and we test it in $H$ by $p$. By adding and subtracting $\tau q$ where necessary, we get:
	\[ M (- q, p +\tau q - \tau q)_H + M \norm{\nabla p}^2_H + M \underbrace{ (P(\phib) p, p) }_{\ge 0} - M (P(\phib) r, p +\tau q -\tau q) = 0. \]
	Then, by using Cauchy-Schwarz and Young's inequalities with the same $\eps > 0$ as before, without loss of generality, we deduce that
	\begin{equation}
		\label{adj:est2}
		M \tau \norm{q}^2_H + M \norm{\nabla p}^2_H \le \tau \eps \norm{q}^2_H + C_{\eps,M} \norm{p + \tau q}^2_H + C_{\eps,M} \norm{r}^2_H.
	\end{equation}
	Now we test \eqref{eq:r} in $H$ by $r$, so that
	\[ - \mezzo \ddt \norm{r}^2_H + \norm{\nabla r}^2_H = \chi (q,r)_H + (P(\phib)(p+ \tau q - \tau q - r), r)_H - (\beta_Q (\sigmab - \sigma_Q), r)_H.  \]
	Then, again by Cauchy-Schwarz and Young, we obtain:
	\begin{equation}
		\label{adj:est3}
		- \mezzo \ddt \norm{r}^2_H + \norm{\nabla r}^2_H \le \tau \eps \norm{q}^2_H + \delta \norm{q}^2_H + C_{\eps} \norm{p + \tau q}^2_H + C_{\delta} \norm{r}^2_H + C \norm{\beta_Q (\sigmab - \sigma_Q)}^2_H,
	\end{equation}
	with $\delta >0$ again to be chosen later. Finally, we need to compensate the term $\frac{1}{8} \norm{\Delta r}^2_H$ in \eqref{adj:est1}. This is where we will need the stronger hypothesis $\sigma_\Omega \in V$ in \ref{C2}, after time-integration. Indeed, now we test \eqref{eq:r} in $H$ by $- \Delta r$, which is possible within the discretization. By adding and subtracting again $\tau q$ where necessary, we get:
	\[
		- \mezzo \ddt \norm{\nabla r}^2_H + \norm{\Delta r}^2_H = \chi (q, \Delta r) + (P(\phib) (p + \tau q - \tau q -r), \Delta r)_H + (\beta_Q (\sigmab- \sigma_Q), \Delta r)_H.
	\]
	We now estimate each term on the right-hand side by using Cauchy-Schwarz's and Young's inequality, indeed:
	\begin{equation}
		\label{adj:est4}
		\begin{split}
		 - \mezzo \ddt \norm{\nabla r}^2_H + \norm{\Delta r}^2_H & \le \frac{3}{4} \norm{\Delta r}^2_H + \chi^2 \norm{q}^2_H + \tau P^2_\infty \norm{q}^2_H \\
		& \quad + C \norm{p +\tau q}^2_H + C \norm{r}^2_H + C \norm{\beta_Q (\sigmab- \sigma_Q)}^2_H.
		\end{split}
	\end{equation}
	Now, we sum all inequalities \eqref{adj:est1}, \eqref{adj:est2}, \eqref{adj:est3}, \eqref{adj:est4} and we obtain the following: 
	\begin{align*}
		& - \mezzo \ddt \left( \norm{p + \tau q}^2_H + \norm{r}^2_H + \norm{\nabla r}^2_H \right) + M \norm{\nabla p}^2_H  + \norm{\nabla r}^2_H \\
		& \qquad + \left( \tau M - \chi^2 - \tau P^2_\infty - \tau \eps - \delta \right) \norm{q}^2_H + \frac{1}{8} \norm{\Delta r}^2_H \\ 
		& \quad \le C_\eps (1 + \norm{\sigmab + \chi (1 - \phib) - \mub}^2_\infty + \norm{\ub}^2_\infty) \norm{p+ \tau q}^2_H + C_\delta \norm{r}^2_H \\ 
		& \qquad + C \norm{\beta_Q (\sigmab- \sigma_Q)}^2_H + C \norm{\alpha_Q (\phib - \phi_Q)}^2_H.
	\end{align*}
	At this point, we can respectively choose $M = M(\tau) > 0$ big enough and then $\eps= \eps(\tau) >0$ and $\delta=\delta(\tau)>0$ small enough such that
	\[ \left( \tau M - \chi^2 - \tau P^2_\infty - \tau \eps - \delta \right) := \gamma(\tau) > 0. \]
	Therefore, by integrating on $(t,T)$, for any $t \in (0,T)$, and using the prescribed final data, we eventually arrive at the inequality:
	\begin{align*}
		& \mezzo \norm{(p+\tau q)(t)}^2_H + \mezzo \norm{r(t)}^2_H + \mezzo \norm{\nabla r(t)}^2_H + M(\tau) \int_t^T \norm{\nabla p}^2_H \, \de s + \int_t^T \norm{\nabla r}^2_H \, \de s \\
		& \qquad + \gamma(\tau) \int_t^T \norm{q}^2_H \, \de s + \frac{1}{8} \int_t^T \norm{\Delta r}^2_H \, \de s \\
		& \quad \le C \norm{\alpha_\Omega (\phib(T) - \phi_\Omega)}^2_H + C \norm{\beta_\Omega (\sigmab(T) - \sigma_\Omega)}^2_V + C_\tau \int_t^T \norm{r}^2_H \, \de s \\
		& \qquad + C_\tau \int_t^T (1 + \norm{\sigmab + \chi (1 - \phib) - \mub}^2_\infty + \norm{\ub}^2_\infty) \norm{p+ \tau q}^2_H \, \de s \\ 
		& \qquad + C \int_t^T \norm{\beta_Q (\sigmab- \sigma_Q)}^2_H \, \de s + C \int_t^T \norm{\alpha_Q (\phib - \phi_Q)}^2_H \, \de s. 
	\end{align*}
	Hence, by the regularities given by Theorem \ref{thm:strongsols} and \ref{C2}, it is possible to apply Gronwall's lemma to obtain the following uniform estimates:
	\begin{equation}
		\label{adj:gronwall}
		\begin{split}
			& \norm{p + \tau q}^2_{L^\infty(0,T;H)} + \tau \norm{q}^2_{L^2(0,T;H)} + \norm{p}^2_{L^2(0,T;V)} + \norm{r}^2_{L^\infty(0,T;V) \cap L^2(0,T;W)} \\
			& \qquad \le C_\tau \Big( \norm{\alpha_\Omega (\phib(T) - \phi_\Omega)}^2_H + \norm{\beta_\Omega (\sigmab(T) - \sigma_\Omega)}^2_V \\
			& \qquad \qquad \qquad + \norm{\beta_Q (\sigmab- \sigma_Q)}^2_{L^2(0,T;H)} + \norm{\alpha_Q (\phib - \phi_Q)}^2_{L^2(0,T;H)}  \Big),
		\end{split}
	\end{equation}
	where $C_\tau>0$ is a constant depending only on the parameters of the system. Now, by comparison in \eqref{eq:q}, since $P \in L^\infty$ and $p, q, r \in L^2(0,T;H)$ we can deduce that also $\Delta p$ is uniformly bounded in $L^2(0,T;H)$. In an analogous way, by comparison in \eqref{eq:r}, we can see that $r_t$ is also uniformly bounded in $L^2(0,T;H)$. Therefore we have that
	\begin{align*}
		\norm{p}^2_{L^2(0,T;W)} + \norm{r}^2_{H^1(0,T;H)} & \le C_\tau \Big( \norm{\alpha_\Omega (\phib(T) - \phi_\Omega)}^2_H + \norm{\beta_\Omega (\sigmab(T) - \sigma_\Omega)}^2_V \\
		& \qquad \,\,\, + \norm{\beta_Q (\sigmab- \sigma_Q)}^2_{L^2(0,T;H)} + \norm{\alpha_Q (\phib - \phi_Q)}^2_{L^2(0,T;H)}  \Big).
	\end{align*}
	Finally, we observe that by H\"older's inequality and Sobolev embeddings:
	\begin{align*}
		& \int_0^T \norm{ P(\phib) (\sigmab + \chi (1-\phib) - \mub) (p-r)}^2_H \, \de t \le P^2_\infty \int_0^T \norm{ \sigmab + \chi (1-\phib) - \mub }^2_{4} \norm{p-r}^2_{4} \, \de t \\
		& \qquad \le P_\infty^2 \norm{ \sigmab + \chi (1-\phib) - \mub }^2_{L^\infty(0,T;V)} \int_0^T \norm{p-r}^2_V \, \de t,
	\end{align*}
	which is uniformly bounded by Theorem \ref{thm:strongsols} and \eqref{adj:gronwall}. Then, by comparison in \eqref{eq:p}, we deduce that also 
	\begin{align*}
		& \norm{p + \tau q}^2_{H^1(0,T;H)} \le C_\tau \Big( \norm{\alpha_\Omega (\phib(T) - \phi_\Omega)}^2_H + \norm{\beta_\Omega (\sigmab(T) - \sigma_\Omega)}^2_V \\
		& \qquad \qquad \qquad \qquad \qquad + \norm{\beta_Q (\sigmab- \sigma_Q)}^2_{L^2(0,T;H)} + \norm{\alpha_Q (\phib - \phi_Q)}^2_{L^2(0,T;H)}  \Big).
	\end{align*}
	With all these uniform estimates, one can easily pass to the limit in a discretisation framework an prove the existence of a solution. Then, by linearity of the system, estimate \eqref{adj:gronwall} provides also uniqueness. 
\end{proof}

\begin{remark}
	Notice that in this case, even if $\tau$ is bounded from above as in Remark \ref{rmk:epsdelta}, the constant $C_\tau$ in \eqref{adj:gronwall} cannot be made independent of $\tau$. This could be possible if $\chi = 0$, indeed the definition of the coefficient $\gamma$ would become:
	\[ (\tau M - \tau P^2_\infty - \tau \eps) = \tau (M - P^2_\infty - \eps) = \tau \gamma, \]
	with $\gamma$ independent of $\tau$. However, if for instance one intends to try sending $\tau \to 0$, then, in general, different estimates would be required. Regarding this problem, we refer the reader to \cite{RSS2021}.
\end{remark}

Finally, with the adjoint variables, we can determine and then simplify the first-order necessary conditions. Indeed, we have the following result:

\begin{theorem}
	\label{thm:optcond}
	Assume hypotheses \emph{\ref{ass:coeff}--\ref{ass:fc0}}, \emph{\ref{ass:j2}--\ref{ass:initial2}} and \emph{\ref{C1}}--\emph{\ref{C4}}. Let $(\ub, \vb) \in \Uad \times \Vad$ be an optimal control for \emph{(CP)} and let $(\phib, \mub, \sigmab) = \mathcal{S}(\ub, \vb)$ be the corresponding optimal state, i.e.~the solution of \eqref{eq:phi2}--\eqref{ic2} with these $(\ub, \vb)$. Let also $(p,q,r)$ be the adjoint variables to $(\phib, \sigmab, \mub)$, i.e.~the solutions to the adjoint system \eqref{eq:p}--\eqref{fca}. Then, they satisfy the following variational inequality, which holds for any $(u,v) \in \Uad \times \Vad$:
	\begin{equation}
		\label{var:ineq}
			\int_0^T \int_\Omega (- \hh(\phib) p + \alpha_u \ub )(u - \ub) \, \de x  \, \de t + \int_0^T \int_\Omega ( r + \beta_v \vb) (v - \vb) \, \de x  \, \de t \ge 0.
	\end{equation}
\end{theorem}

\begin{proof}
	Before starting, observe that, by the regularities guaranteed by Theorems \ref{thm:strongsols}, \ref{thm:linearised} and \ref{thm:adjoint}, all the integrals that we are going to write make sense in $\Lqt1$. 
	
	First observe that the cost functional $\mathcal{J}$ is convex and Fréchet-differentiable in the space $\mathcal{C}^0([0,T];H) \times \mathcal{C}^0([0,T];H) \times L^2(Q_T) \times L^2(Q_T)$. Next, in Theorem \ref{thm:frechet} we showed that the solution operator $\mathcal{S}$ is Fréchet-differentiable from $\mathcal{U}_R \times \mathcal{V}_R \subseteq (L^\infty(Q_T) \cap H^1(0,T;H)) \times L^\infty(Q_T)$ to $\mathbb{W}$. Moreover, since by standard results $\LT \infty V \cap \HT 1 H$ is embedded with continuity in  $\C 0 H$, we also have that the operator $(\mathcal{S}_1, \mathcal{S}_3)$ that selects the first and third components of $\mathcal{S}$ is Fréchet-differentiable from $\mathcal{U}_R \times \mathcal{V}_R$ to $(\C 0 H)^2$. Therefore, we can consider the \emph{reduced cost functional} $f: (\Lqt \infty \cap \HT 1 H) \times \Lqt \infty \to \R$, defined as 
	\[ f(u,v) := \mathcal{J}(\mathcal{S}_1(u,v), \mathcal{S}_3(u,v), u, v),  \]  
	which, by the chain rule, is Fréchet-differentiable in $\mathcal{U}_R \times \mathcal{V}_R$. 
	
	At this point, we can rewrite our optimal control problem (CP) through the reduced cost functional as the minimisation problem
	\[ \argmin_{(u,v) \in \, \Uad \times \Vad} f(u,v). \]
	Then, if $(\ub, \vb)$ is optimal, since $\Uad \times \Vad$ is convex and $f$ is Fréchet-differentiable, it has to satisfy the necessary optimality condition
	\[ f'(\ub, \vb) [ (u - \ub, v - \vb) ] \ge 0 \quad \forall (u,v) \in \Uad \times \Vad. \]
	When computing explicitly the derivative of $f$, we get that for any $(u,v) \in \Uad \times \Vad$ 
	\begin{align*}
		& \int_\Omega \alpha_\Omega (\phib(T) -\phi_\Omega) \xi(T) \, \de x  \, \de t + \int_0^T \int_\Omega \alpha_Q (\phib - \phi_Q) \xi \, \de x  \\
		& + \int_\Omega \beta_\Omega (\sigmab(T) -\sigma_\Omega) \rho(T) \, \de x  \, \de t + \int_0^T \int_\Omega \beta_Q (\sigmab - \sigma_Q) \rho \, \de x  \\
		& + \int_0^T \int_\Omega  \alpha_u \ub (u - \ub) \, \de x  \, \de t + \int_0^T \int_\Omega \beta_v \vb (v - \vb) \,\de x  \, \de t \ge 0,
	\end{align*}
	where $\xi = D \mathcal{S}_1(\ub,\vb)[u - \ub, v - \vb]$ and $\rho = D \mathcal{S}_3(\ub,\vb)[u - \ub, v - \vb]$ are the components of the solution $(\xi, \eta, \rho)$ to the linearised system \eqref{eq:xi}--\eqref{icl} in $(\phib, \mub, \sigmab)$ corresponding to $(h,k) = (u - \ub, v - \vb)$.
	
	Now observe that the right-hand sides and the final conditions of the adjoint system appear in this inequality, therefore by substituting equations \eqref{eq:p}, \eqref{eq:r} and \eqref{fca} in the previous expression, we find that for any $(u,v) \in \Uad \times \Vad$ 
	\begin{align*}
		& \int_\Omega (p+\tau q)(T) \xi(T) \, \de x  + \int_0^T \int_\Omega \big( - \partial_t (p +\tau q)  + AF''(\phib) q + Baq - BJ \ast q + \chi \Delta r \\ 
		& \qquad + \chi P(\phib) (p-r) - P'(\phib) (\sigmab + \chi (1-\phib) - \mub)(p-r) + p \hh'(\phib) \ub \big) \xi \, \de x  \, \de t \\ 
		& \quad + \int_\Omega r(T) \rho(T) \, \de x  + \int_0^T \int_\Omega \big(- \partial_t r - \Delta r - \chi q  - P(\phib) (p-r) \big) \rho \, \de x  \, \de t \\
		& \quad + \int_0^T \int_\Omega  \alpha_u \ub (u - \ub) \, \de x  \, \de t + \int_0^T \int_\Omega \beta_v \vb (v - \vb) \,\de x  \, \de t \ge 0.
	\end{align*}
	Now we integrate by parts in time, by using also the initial conditions \eqref{icl} on the linearised system, and in space, by using the boundary conditions \eqref{bca} and \eqref{bcl}, and, after cancellations, we find that equivalently for any $(u,v) \in \Uad \times \Vad$ 
	 \begin{align*}
	 	& \int_0^T \int_\Omega \big( p \xi_t  + \tau q \xi_t  + AF''(\phib) \xi q + Ba \xi q - B(J \ast \xi) q  + \chi \Delta \xi \, r \\ 
	 	& \qquad + \chi P(\phib) (p-r) \xi  - P'(\phib) (\sigmab + \chi (1-\phib) - \mub)(p-r) \xi + p \hh'(\phib) \ub \, \xi  \big) \, \de x  \, \de t \\ 
	 	& \quad + \int_0^T \int_\Omega \big( \rho_t r - \Delta \rho r - \chi \rho q  - P(\phib) (p-r) \rho \big) \, \de x  \, \de t \\
	 	& \quad + \int_0^T \int_\Omega  \alpha_u \ub (u - \ub) \, \de x  \, \de t + \int_0^T \int_\Omega \beta_v \vb (v - \vb) \,\de x  \, \de t \ge 0,
	 \end{align*}
 	where we also used the symmetry of the kernel $J$. By factoring out $p$, $q$ and $r$ respectively, we can rewrite the previous inequality as
 	\begin{align*}
 		& \int_0^T \int_\Omega p \left( \xi_t - P(\phib)(\rho - \chi \xi) - P'(\phib) (\sigmab + \chi (1-\phib) - \mub) \xi + \hh'(\phib) \ub \, \xi \right) \, \de x  \, \de t \\ 
 		& \quad + \int_0^T \int_\Omega q \left( \tau \xi_t + AF''(\phib) \xi + Ba \xi - BJ\ast \xi - \chi \rho  \right) \, \de x  \, \de t \\
 		& \quad + \int_0^T \int_\Omega r \left( \rho_t - \Delta \rho + \chi \Delta \xi + P(\phib)(\rho -\chi \xi) + P'(\phib) (\sigmab + \chi (1-\phib) - \mub) \xi \right) \, \de x  \, \de t \\
 		& \quad +  \int_0^T \int_\Omega  \alpha_u \ub (u - \ub) \, \de x  \, \de t + \int_0^T \int_\Omega \beta_v \vb (v - \vb) \,\de x  \, \de t \ge 0.
 	\end{align*}
 	Finally, we use equation \eqref{eq:q} and again integration by parts to also get that 
 	\begin{align*}
 		0 & = \int_0^T \int_\Omega \left( - q - \Delta p + P(\phib)(p-r) \right) \eta \, \de x  \, \de t \\
 		 & = \int_0^T \int_\Omega - \eta q  - \Delta \eta \, p +   P(\phib)(p-r) \eta \, \de x  \, \de t.
 	\end{align*}
 	Then, by adding this to the previous inequality, we at last infer that for any $(u,v) \in \Uad \times \Vad$ 
 	\begin{align*}
 		& \int_0^T \int_\Omega p \left( \xi_t - \Delta \eta - P(\phib)(\rho - \chi \xi - \eta) - P'(\phib) (\sigmab + \chi (1-\phib) - \mub) \xi + \hh'(\phib) \ub \, \xi \right) \, \de x  \, \de t \\ 
 		& \quad + \int_0^T \int_\Omega q \left( - \eta + \tau \xi_t + AF''(\phib) \xi + Ba \xi - BJ\ast \xi - \chi \rho  \right) \, \de x  \, \de t \\
 		& \quad + \int_0^T \int_\Omega r \left( \rho_t - \Delta \rho + \chi \Delta \xi + P(\phib)(\rho -\chi \xi -\eta) + P'(\phib) (\sigmab + \chi (1-\phib) - \mub) \xi \right) \, \de x  \, \de t \\
 		& \quad +  \int_0^T \int_\Omega  \alpha_u \ub (u - \ub) \, \de x  \, \de t + \int_0^T \int_\Omega \beta_v \vb (v - \vb) \,\de x  \, \de t \ge 0.
 	\end{align*}
 	To conclude, we notice that the expressions enclosed in the parentheses are exactly the equations \eqref{eq:xi}, \eqref{eq:eta}, \eqref{eq:rho} of the linearised system, up to their source terms. Hence, by substituting those into our inequality, we find that for any $(u,v) \in \Uad \times \Vad$ 
 	\begin{align*}
 		& \int_0^T \int_\Omega - p \, \hh(\phib) (u - \ub) \, \de x  \, \de t + \int_0^T \int_\Omega r (v -\vb) \, \de x  \, \de t \\ 
 		& \quad +  \int_0^T \int_\Omega  \alpha_u \ub (u - \ub) \, \de x  \, \de t + \int_0^T \int_\Omega \beta_v \vb (v - \vb) \,\de x  \, \de t \ge 0,
 	\end{align*}
 	which is exactly \eqref{var:ineq}. This concludes the proof of Theorem \ref{thm:optcond}.
\end{proof}

\begin{remark}
	Observe that, since $\Uad \times \Vad$ is closed and convex, \eqref{var:ineq} means that, if $\alpha_u >0$ and $\beta_v >0$, the optimal control $(\ub, \vb)$ is exactly the $L^2(Q_T)^2$-orthogonal projection of $( \alpha_u^{-1} \hh(\phib) \, p,  - \beta_v^{-1} r )$ onto $\Uad \times \Vad$. 
	In particular, if $u \equiv 0$, i.e.~we are only looking for a chemotherapy $v$ through the nutrient, it can be shown that, due to the structure of $\Vad$, the $L^2(Q_T)$-projection of $- \beta_v^{-1} r$ has the explicit form:
	\[ \vb(x,t) = \min \left\{ v_{\text{max}}(x,t), \max \left\{ - \beta_v^{-1} r(x,t), v_{\text{min}}(x,t) \right\} \right\} \quad \text{for a.e. } (x,t) \in Q_T. \]
	If we also want to optimise for a radiotherapy $u$, the actual description of the projection onto the space $\Uad$ is more difficult and has to be tackled numerically, due to the presence of the additional constraint $\norm{u}_{\HT 1 H} \le M$. However, this cannot be avoided, since the extra regularity on $u$ is needed to prove the global boundedness of the variable $\phi$, which is crucial for the whole analysis. Note that this kind of restriction is not unprecedented for controls acting as source terms in Cahn-Hilliard type equations, see for instance \cite{CGRS2022}. 
\end{remark}

\section*{Acknowledgements}

The author wishes to express his gratitude to professor Elisabetta Rocca for introducing him to this line of research and for several fruitful discussions and suggestions, without which this work would not have been possible. The author also wishes to thank the anonymous referees, who carefully read the manuscript and provided many comments that improved the quality of the paper. This research activity has been performed in the framework of the MIUR-PRIN Grant 2020F3NCPX ``Mathematics for industry 4.0 (Math4I4)'' and the GNAMPA (Gruppo Nazionale per l'Analisi Matematica, la Probabilit\`a e le loro Applicazioni) of INdAM (Istituto Nazionale di Alta Matematica).

\section*{Declarations}

The author has no conflict of interest to declare.


\end{document}